\documentclass{article}
\usepackage[utf8]{inputenc}
\usepackage{amsmath, amsthm}
\usepackage{amssymb}
\usepackage{todonotes}
\usepackage{tikz-cd}
\usepackage{enumerate}
\usepackage[utf8]{inputenc}
\usepackage{amsmath}
\usepackage{newunicodechar}
\usepackage{diagrams}
\usepackage{comment}

\newcommand{\shi}{\text{\usefont{U}{min}{m}{n}\symbol{'127}}}

\DeclareFontFamily{U}{min}{}
\DeclareFontShape{U}{min}{m}{n}{<-> udmj30}{}

\newtheorem{theorem}{Theorem}
\newtheorem*{theorem*}{Theorem}
\newtheorem*{corollary*}{Corollary}
\newtheorem{definition}[theorem]{Definition}
\newtheorem{lemma}[theorem]{Lemma}
\newtheorem{prop}[theorem]{Proposition}
\newtheorem{remark}[theorem]{Remark}
\newtheorem{claim}[theorem]{Claim}
\newtheorem{corollary}[theorem]{Corollary}
\newtheorem{question}{Question}

\newcommand{\Col}{\mathrm{Col}}
\newcommand{\col}{\mathrm{Col}}
\newcommand{\Add}{\mathrm{Add}}
\newcommand{\add}{\mathrm{Add}}
\newcommand{\la}{\langle}
\newcommand{\ra}{\rangle}
\newcommand{\rest}{\restriction}
\newcommand{\Sh}{\shi}
\newcommand{\id}{\mathrm{id}}
\newcommand{\p}{\mathcal{P}}

\newcommand{\bbA}{\mathbb{A}}
\newcommand{\bbB}{\mathbb{B}}
\newcommand{\bbC}{\mathbb{C}}

\newcommand{\bbP}{\mathbb{P}}
\newcommand{\bbQ}{\mathbb{Q}}
\newcommand{\bbR}{\mathbb{R}}
\newcommand{\bbS}{\mathbb{S}}
\newcommand{\bbT}{\mathbb{T}}
\newcommand{\bbU}{\mathbb{U}}

\newcommand{\GCH}{\mathrm{GCH}}
\newcommand{\SCH}{\mathrm{SCH}}
\newcommand{\ZFC}{\mathrm{ZFC}}

\DeclareMathOperator{\dom}{dom}
\DeclareMathOperator{\ran}{ran}
\DeclareMathOperator{\len}{len}
\DeclareMathOperator{\cf}{cf}
\DeclareMathOperator{\acc}{acc}
\DeclareMathOperator{\crit}{crit}
\DeclareMathOperator{\ot}{ot}
\DeclareMathOperator{\Ult}{Ult}
\DeclareMathOperator{\cof}{cof}

\date{}
\author{Monroe Eskew and Yair Hayut}
\title{Dense Ideals}
\begin{document}

\maketitle
\begin{abstract}
    In this paper, we obtain the consistency, relative to large cardinals, of the existence of dense ideals on every successor of a regular cardinal simultaneously. Using a consequent transfer principle, we show that in this model there is a $\sigma$-complete, $\aleph_1$-dense ideal on $\aleph_{n+1}$ for every $n < \omega$, answering a question of Foreman. 

    Using this construction we show the consistency of the existence of various irregular ultrafilters on $\omega_n$, the consistency of the Foreman-Laver reflection property for the chromatic number of graphs for all possible pairs of cardinals below $\aleph_\omega$, and the simultaneous consistency of the partition hypotheses $\mathrm{PH}_n(\omega_m)$ for $n < m$.
\end{abstract}
\section{Introduction}
A common philosophical stance in set theory is that a statement that follows from a large cardinal axiom is morally true, even if independent from the standard axioms of set theory, $\ZFC$. For example, the statement that there is no definable non-measurable subset of the reals (e.g.\ in the structure $\la \mathbb R; \mathbb N, +, \cdot, exp\ra$) is considered true even though it is false in $L$, as it follows from the existence of sufficiently many large cardinals. 

Nevertheless, by the Levy-Solovay theorem, large cardinals are indifferent to small forcing such as the ones that can wiggle the value of the continuum, so large cardinals cannot decide whether the continuum hypothesis is true. In the paper \cite{Foreman86}, Foreman suggests an alternative to the standard large cardinal hierarchy, the \emph{generic} large cardinals. Those axioms are assertions about the existence of elementary embeddings in a generic extension of the universe.  They are not linearly ordered as the standard large cardinal axioms are, but they can hold at small cardinals (such as $\aleph_n$), thus giving us information about the universe's structure at those cardinals.\footnote{This paradigm has its limitations, see \cite{Eskew2020}.}

The strength of an elementary embedding $j: V \to M$ is measured by the similarity of the obtained model $M$ to the universe $V$, relative to the critical point of $j$. For a generic elementary embedding, there is another important parameter: the forcing notion that introduces the elementary embedding. The more modest the forcing notion is, the stronger the generic large cardinal axiom becomes. To get an elementary embedding with a critical point $\mu^+$, where $\mu$ is a cardinal, $\mu^+$ must be collapsed. When $\mu$ is regular, the Levy collapse $\col(\mu,\mu^+)$ is the most natural candidate for a forcing notion that minimizes the modification of the universe in this case. 

Similarly to the connection between $\sigma$-complete ultrafilters and definable elementary embeddings, there is a tight connection between ideals with various strong combinatorial properties and generic elementary embeddings. For an ideal $I$ on a set $X$, the forcing notion $\p(X)/I$ introduces a generic elementary embedding as an ultrapower using the generic filter for the forcing. 

Unlike traditional large cardinal axioms, generic large cardinal axioms can be mutually inconsistent. Thus, one of the goals in the study of generic large cardinals and ideals with strong properties is to isolate the limitations for the co-existence of such ideals at many small regular cardinals. 

The main theorem of this paper is Theorem \ref{thm:dense-ideals-everywhere}.
\begin{theorem*}[Main Theorem]
It is consistent, relative to a huge cardinal, that for every regular cardinal $\mu$, there is a normal ideal $I$ on $\mu^+$ such that
\[\p(\mu^+) / I \cong \mathcal{B}(\Col(\mu,\mu^+)).\]
\end{theorem*}
 By Proposition~\ref{gch}, the existence of such ideals implies $\GCH$, which implies that the forcing $\Col(\mu,\mu^+)$ has cardinality $\mu^+$.  We conclude that the ideal $I$ on $\mu^+$ is $\mu^+$-dense. 
 
 This result is optimal as it obtains the maximal class of successor cardinals which can carry a dense ideal. For a singular cardinal $\mu$, there can be no $\mu^{+}$-dense, normal ideal on $\mu^+$, by \cite[Lemma 2]{Eskew2020}.

 Following the terminology of \cite[Section 11.4]{foremanhandbook}, our main theorem shows that it is consistent for every successor of a regular cardinal to be \emph{minimally generically almost-huge}.\footnote{This means that for $\kappa = \mu^+$, the forcing $\col(\mu,\kappa)$ introduces an elementary embedding $j \colon V \to M$, where $M$ is a ${<}j(\kappa)$-closed transitive subclass of the generic extension $V[G]$.}  Our result cannot be improved to get any one cardinal minimally generically huge, as minimal generic hugeness of $\kappa$ implies that $\kappa^+$ is not generically  measurable via a $\kappa$-closed forcing, by \cite[Theorem 1]{genhuge}.

Using the Boolean transfer principle (see Subsection \ref{subsec: boolean-transfer}), we obtain Corollary \ref{cor:aleph_1-dense-ideals}:
\begin{corollary*}[Main Corollary]
    It is consistent, relative to a huge cardinal, that for every $n$, there is a uniform, $\sigma$-complete, $\omega_1$-dense ideal on $\omega_{n+1}$.
\end{corollary*}

This corollary has many combinatorial consequences, and we devote Section \ref{sec: applications} to stating some of them.

Our results answer the following open questions from the literature:
\begin{itemize}
    \item The Main Corollary answers \cite[Question 15]{foremanhandbook}.
    \item Theorem~\ref{up_sizes} obtains all possible values, subject to the constraints of $\GCH$, for sizes of ultrapowers among the $\omega_n$'s, answering \cite[Question 16]{foremanhandbook} below $\aleph_\omega$.\footnote{Other parts of Question 16 were answered in the first author's thesis.}
    \item The Main Theorem implies that it is consistent to have very strongly layered ideals on all $\omega_n$ simultaneously, answering \cite[Question 17]{foremanhandbook}.
    \item The Main Corollary also shows that Foreman's computation of the forbidden segments is sharp in the lower part, as asked in \cite[Question 34]{foremanhandbook}.\footnote{In Subsection \ref{subsec: FK-spectrum}, we improve the computation of the forbidden segments, pointing to a more accurate formulation of Question 34, which remains open.}
    \item In Subsection \ref{subsec: graph-colorings}, we derive as a corollary from the existence of uniform dense ideals the minimal possible chromatic number for the Erd\H{o}s-Hajnal graphs, under $\GCH$, for $\omega_n$, $n < \omega$, addressing a problem of Todor\'{c}evi\v{c}, \cite[Problem 8.17]{todorcevic}.
    \item In Subsection~\ref{subsec: partition-hypothesis}, we obtain an upper bound for the consistency strength of the partition hypotheses of \cite{BBMT}, addressing Question 9.5 of that work.
\end{itemize}

Let us describe the structure of the paper. 
In Section \ref{sec: preliminaries}, we will cover the basic technical properties of forcings, projections, and embeddings needed for the definition of our main forcing notion.

In Section \ref{sec: Dual-shioya-collapse}, we introduce the main technical tool of this paper, which is the \emph{Dual Shioya collapse}. This forcing notion is a minor variant of one developed by Shioya in \cite{ShioyaDense}. It collapses the interval between a pair of uncountable regular cardinals $\mu$ and $\kappa$ in a relatively universal way. Intuitively, the forcing glues together generically arbitrary $\mu$-closed forcing notions of cardinality $<\kappa$.

In order to lift an elementary embedding after the Dual Shioya collapse, we need to make sure that many steps of the generic collapse behave as inverse limits, a property that generically fails quite often. In Section \ref{sec: uniformization}, we introduce a forcing notion that will allow us to restore a large set on which inverse limits exist, which will be crucial in order to construct the required master conditions. 

In Section \ref{sec: dense-ideals-1-2}, we prove the consistency of the simultaneous existence of dense ideals at $\omega_1$ and $\omega_2$ relative to a suitable large cardinal hypothesis. This is an improvement on Foreman's result \cite{foreman98}, where he obtained a dense ideal on $\omega_1$ together with a very strongly layered ideal $\omega_2$, which he showed has many of the same combinatorial consequences at the level of $\omega_2$.  
The construction of the ideal for $\omega_1$ differs from the construction of a dense ideal on $\mu^+$ for uncountable regular $\mu$. The reason for that is that the Dual Shioya collapse is not well-defined when the lower cardinal is $\omega$. Luckily, the problem of constructing dense ideals on $\omega_1$ was already solved in several ways, some of which are compatible with the rest of the construction. 

In Section \ref{sec: more-dense-ideals} we prove our main consistency result by showing how to lift an almost-huge embedding after a two-step iteration of the forcing from Section \ref{sec: Dual-shioya-collapse} and Section \ref{sec: uniformization}. 

In Section \ref{sec: applications}, we derive a couple of $\ZFC$-consequences from the existence of ideals as in the Main Theorem. This section
does not require any knowledge about the results of the other sections and can be read independently.
We start by obtaining some simple consequences of these ideals and then, in Subsection \ref{subsec: boolean-transfer} we recall the principle of Boolean transfer that was developed by Foreman and Woodin and apply it in our case in order to get a variety of ideals of different density and completeness over $\omega_n$ for natural $n$. In \cite{foreman98}, Foreman developed a strong and general version of this transfer principle which applies to many types of ideals. We present a simplified version of the proof which is suitable for our case. Using the Boolean transfer principle, we repeat Foreman's argument for our ideals and obtain in Subsection \ref{subsec: irregular-ultrafilters} highly irregular ultrafilters. In Subsection \ref{subsec: FK-spectrum}, we present an improved argument for the computation of a class of regular cardinals on which no $\sigma$-complete $\omega_2$-saturated ideal can exist. 

In Subsection \ref{subsec: graph-colorings}, we use those irregular ultrafilters in order to compute the chromatic number of the Erd\H{o}s-Hajnal graph for $\omega_n^{\omega_m}$ and obtain the Foreman-Laver reflection principle for the chromatic number of graphs between all possible pairs of cardinals below $\aleph_\omega$. Finally, in Subsection \ref{subsec: partition-hypothesis}, we derive the consistency of the Partition Hypothesis of \cite{BBMT} and conclude (using an unpublished work, \cite{BLZ}) the consistency of the triviality of the cohomology group $\hat{\mathrm{H}}^n(\omega_m,\mathcal{A})$ for $n < m$.  

In Section \ref{sec: questions}, we list some open problems.

Throughout the paper we work in $\ZFC$ and state explicitly any large cardinal hypothesis that we are using. We refer the reader to \cite{Jech} for basic facts about forcing and to \cite{Kanamori} for definitions and basic properties of large cardinals.
\section{Preliminaries}\label{sec: preliminaries}
Let us first recall a few basic forcing facts.  A partially ordered set (poset) $\bbP$ is called \emph{separative} when for all $p,q \in \bbP$, if $p \nleq q$, then there is $r \leq p$ such that $r \perp q$, i.e.\ there is no element below both $r$ and $q$.  The key property of separative posets is that $p \leq q$ if and only if $p \Vdash q \in \dot G$, where $\dot G$ is the canonical name for the generic filter.  A non-separative poset may be made separative by forming the \emph{seperative quotient}, where we put $p \preceq q$ iff all $r \leq p$ are compatible with $q$ and defining the equivalence relation $p \sim q \leftrightarrow p \preceq q \preceq p$.  A generic for a poset can be readily translated into one for its separative quotient and vice versa. Thus, we will implicitly identify a poset and its separative quotient, unless it can cause confusion.

If $\bbP$ is a poset and $\dot\bbQ$ is a $\bbP$-name for a poset, then the iteration $\bbP*\dot\bbQ$ is taken to be the set of pairs $\la p,\dot q\ra \in H_\theta$, where $\theta$ is the least regular cardinal such that $\bbP,\dot\bbQ\in H_\theta$, such that $\bbP$ forces $\dot q \in \dot\bbQ$, with the order defined as $\la p_1,\dot q_1 \ra\leq\la p_0,\dot q_0\ra$ iff $p_1 \leq p_0$ and $p_1 \Vdash \dot q_1 \leq \dot q_0$.  This ordering is usually not antisymmetric, but it becomes a partial order when we mod out by the equivalence relation $\la p_1,\dot q_1 \ra\sim\la p_0,\dot q_0\ra$ iff $p_0 = p_1$ and $p_0 \Vdash \dot q_0 = \dot q_1$.  We officially take this quotient as our definition of two-step iteration, but we will suppress mention of it for notational convenience.  It is easy to see that this poset is separative if $\bbP$ is separative and $\dot\bbQ$ is forced to be separative.

There are two main ways that one forcing can be thought of as a subforcing of another: projections and complete embeddings.

\begin{definition}
    For posets $\bbP$ and $\bbQ$, a map $\pi: \bbQ \to \bbP$ is a \emph{projection} when it is order-preserving, its range is dense in $\bbP$, and it has the property that for all $q \in \bbQ$ and all $p \leq \pi(q)$, there is $q' \leq q$ such that $\pi(q') \leq p$.
\end{definition}

The statement that $\ran\pi$ is dense in $\bbP$ is made superfluous by requiring that $\bbP$ and $\bbQ$ have maximum elements $1_\bbP$ and $1_\bbQ$ respectively, and $\pi(1_\bbQ) = 1_\bbP$.  If $\bbP$ or $\bbQ$ is missing a maximum element, then these can be artificially added and a given projection can be extended to map one to the other.  We also note that if $\bbP$ is separative, $\bbQ$ has a maximum element $1_\bbQ$, and $\pi: \bbQ \to \bbP$ is a projection in the above sense, then $\pi(1_\bbQ)$ must be a maximum element of $\bbP$.

The above notion of projection is weaker than what is often used in the literature (for example in \cite{AbrahamHandbook}), as we do not require that whenever $p \leq \pi(q)$, we can find $q' \leq q$ with $\pi(q') = p$.  A key advantage of using the weaker notion is that, if $D \subseteq \bbQ$ is dense, then $\pi \rest D$ is also a projection from $D$ to $\bbP$.  And this weaker notion still does the usual jobs:  If $\pi: \bbQ \to \bbP$ is a projection and $H \subseteq \bbQ$ is generic, then $\pi[H]$ generates a generic filter for $\bbP$, which we denote by $\pi(H)$.  Forcing with $\bbQ$ is equivalent to first taking a generic $G \subseteq \bbP$ and then forcing with the quotient $\bbQ/G = \pi^{-1}[G].$


A map $e: \bbP \to \bbQ$ is a \emph{complete embedding} when it preserves order and incompatibility and sends maximal antichains to maximal antichains.  Such a map is called a \emph{dense embedding} when its range is dense in $\bbQ$.  If $e: \bbP \to \bbQ$ is a complete embedding, then forcing with $\bbQ$ is equivalent to first taking a generic $G \subseteq \bbP$ and then forcing with the quotient $\bbQ/G = \{ q \in \bbQ: q$ is compatible with $e(p)$ for all $p \in G\}$.
For any separative poset $\bbP$, we can form the Boolean completion $\mathcal{B}(\bbP)$, the complete Boolean algebra of regular open subsets of $\bbP$, in which $\bbP$ appears as a dense set.  For complete Boolean algebras $\bbA,\bbB$, the existence of a projection $\pi: \bbB \to \bbA$ is equivalent to the existence of a complete embedding $e: \bbA \to \bbB$.

For ordinals $\kappa$ and $\lambda$, $\col(\kappa,\lambda)$ is the collection of partial functions from $\kappa$ to $\lambda$ defined at ${<}\kappa$-many points, ordered by reverse inclusion.  $\add(\kappa)$ is $\col(\kappa,2)$.  $\col(\kappa,{<}\lambda)$ is the collection of partial functions $p$ from $\lambda \times \kappa$ to $\lambda$ defined at ${<}\kappa$-many points, such that for all $\la\alpha,\beta\ra\in\dom p$, $p(\alpha,\beta)<\alpha$.  This is isomorphic to the ${<}\kappa$-support product of $\col(\kappa,\alpha)$ for $\alpha<\lambda$.  For $\lambda<\lambda'$, there is a natural projection from $\col(\kappa,{<}\lambda')$ to $\col(\kappa,{<}\lambda)$ given by $p \mapsto p \rest (\lambda \times \kappa)$.  For notational simplicity,  we will usually write $p \rest \lambda$ for this operation.

Let us recall the notions of strategic closure.
For a poset $\bbP$ and an ordinal $\delta$, we define two games $\mathcal{G}^\mathrm{I}_\delta(\bbP)$ and $\mathcal{G}^\mathrm{II}_\delta(\bbP)$.  Two players alternate playing elements of $\bbP$ in a descending sequence, with Player I making the first move.  At limit stages, Player I plays first in $\mathcal{G}^\mathrm{I}_\delta(\bbP)$, and Player II plays first in $\mathcal{G}^\mathrm{II}_\delta(\bbP)$.  Player II wins if the game lasts for $\delta$-many rounds; otherwise, Player I wins.  For a cardinal $\kappa$, we say that $\bbP$ is \emph{$\kappa$-strategically closed} if Player II has a winning strategy for $\mathcal{G}^\mathrm{II}_\kappa(\bbP)$, and \emph{strongly $\kappa$-strategically closed} if Player II has a winning strategy for $\mathcal{G}^\mathrm{I}_\kappa(\bbP)$.

We will use the following variation on Easton's Lemma:
\begin{lemma}
\label{easton}
    Suppose $\kappa$ is a regular cardinal and $\bbP,\bbQ$ are posets such that $|\bbP|<\kappa$ and $\bbQ$ is $\kappa$-distributive.  Then $\bbP$ preserves that $\bbQ$ is $\kappa$-distributive.
\end{lemma}

\begin{proof}
    Suppose $G \times H \subseteq \bbP \times \bbQ$ is generic.  Let $f \in V[G][H]$ be a function from some $\alpha<\kappa$ to the ordinals.  Let $\tau \in V[H]$ be a $\bbP$-name such that $\tau^G = f$.  Since $|\bbP|<\kappa$, we may assume $\tau \in V$.  Thus $\tau^G \in V[G]$.
\end{proof}

\subsection{Complete $\kappa$-closure}
Throughout this subsection, $\kappa$ is a regular uncountable cardinal.
\begin{definition}
    A poset is called \emph{completely $\kappa$-closed} when all descending sequences of length $<\kappa$ have a greatest lower bound or infimum (glb or inf).
\end{definition}

\begin{definition}
Suppose $\bbP$ is a completely $\kappa$-closed poset.  
    A set $X \subseteq \bbP$ is called \emph{$\kappa$-closed} when for all $\delta<\kappa$ and all descending sequences $\la p_i : i < \delta \ra \subseteq X$, $\inf_i p_i \in X$.
\end{definition}

\begin{definition}
    If $\bbP,\bbQ$ are completely $\kappa$-closed posets, then we say an order-preserving map $\pi : \bbP \to \bbQ$ is \emph{$\kappa$-continuous} when for all $\delta<\kappa$ and all descending sequences $\la p_i : i < \delta \ra \subseteq \bbP$, $\pi(\inf_i p_i) = \inf_i \pi(p_i)$.
\end{definition}

\begin{lemma}[Folklore]
Suppose $\bbP$ is a completely $\kappa$-closed poset.  Then every directed $D \in [\bbP]^{<\kappa}$ has a greatest lower bound. 
\end{lemma}

\begin{proof}
    We argue by induction on the cardinality $\delta$ of directed subsets of $\bbP$.  For $\delta = \omega$, given a directed $D \subseteq \bbP$, we select a descending sequence $\la p_i : i < \omega \ra \subseteq D$ such that for every $d \in D$, there is $i<\omega$ such that $p_i \leq d$.  Then $\inf_i p_i = \inf(D)$.  If the hypothesis holds for all $\gamma<\delta$, let $D$ be a directed set of size $\delta$, and let $\la D_\alpha: \alpha<\delta \ra$ be a $\subseteq$-increasing continuous sequence of subsets of $D$ such that each $D_\alpha$ is directed and of size $<\delta$.  By induction, each $D_\alpha$ has a glb, $p_\alpha$.  The descending sequence $\la p_\alpha: \alpha<\delta \ra$ has a glb, $p_\delta$, which is a lower bound to $D$.  Any other lower bound of $D$ must be a lower bound of each $D_\alpha$ and thus below each $p_\alpha$, so $p_\delta = \inf(D)$.
\end{proof}

\begin{lemma}
Suppose $\bbP,\bbQ$ are completely $\kappa$-closed posets and $\pi : \bbP \to \bbQ$ is a $\kappa$-continuous order-preserving map.  Then for all directed $D \in [\bbP]^{<\kappa}$ , $\pi(\inf D) = \inf \pi[D]$.
\end{lemma}

\begin{proof}
If $D$ is countable, then there is a chain $C \subseteq D$ such that for all $d \in D$, there is $c \in C$ with $c \leq d$, so $\inf C = \inf D$ and $\inf \pi[C] = \inf \pi[D]$.  By assumption, $\pi(\inf C) = \inf \pi[C]$.  

Suppose $\delta<\kappa$ and the claim holds for all directed $D$ of cardinality $<\delta$.  Let $D$ be hypothesized, and let $\la D_\alpha : \alpha < \delta\ra$ be a continuous increasing sequence of sets such that each $D_\alpha$ is a directed subset of $D$ of size $<\delta$.  Let $d_\alpha = \inf D_\alpha$.  By induction, $\pi(d_\alpha) = \inf \pi [D_\alpha]$.  We have that $\inf D = \inf_{\alpha<\delta} d_\alpha$, and by $\kappa$-continuity for chains, $\pi(\inf D) = \inf_{\alpha<\delta} \pi(d_\alpha) = \inf_{\alpha<\delta}(\inf \pi[D_\alpha]) = \inf \pi[D]$.
\end{proof}

\begin{lemma}
\label{int_denseclosed}
    Suppose $\bbP$ is a completely $\kappa$-closed poset.  If $\delta<\kappa$ is a cardinal and $\{ D_\alpha : \alpha<\delta \}$ is a collection of $\kappa$-closed dense subsets of $\bbP$, then $\bigcap_{\alpha<\delta} D_\alpha$ is $\kappa$-closed and dense. 
\end{lemma}

\begin{proof}
    The intersection is easily seen to be $\kappa$-closed, by the $\kappa$-closure of each $D_\alpha$.  To show density, first, let us argue for the case $\delta = 2$.  Let $p \in \bbP$ be arbitrary.  We choose a descending sequence $p_0 \geq p_1 \geq p_2 \geq \dots$ below $p$ such that $p_i \in D_0$ for even $i$ and $p_i \in D_1$ for odd $i$.  By the closure of $D_0$ and $D_1$, $p_\omega = \inf_i p_i \in D_0 \cap D_1$.  
    It follows that the intersection of finitely many $\kappa$-closed dense sets is $\kappa$-closed and dense.

    For infinite $\delta$, we argue by induction.  Suppose the claim holds for all $\gamma<\delta$.  Given $p_0 \in \bbP$, we construct a descending sequence $\la p_\alpha : \alpha \leq \delta \ra$ such that $p_{\alpha+1} \in \bigcap_{i\leq\alpha} D_i$, and for limit $\alpha$, $p_\alpha = \inf_{i<\alpha} p_i$.  The successor step is done using the induction hypothesis.  If $\alpha$ is a limit, then for each $\beta<\alpha$, $\{ p_i : \beta<i<\alpha \} \subseteq D_\beta$, and so $p_\alpha \in D_\beta$ by $\kappa$-closure.  It follows that $p_\delta \in \bigcap_{\alpha<\delta} D_\alpha$.
\end{proof}

\begin{lemma}
\label{inv_denseclosed}
    Suppose $\bbP,\bbQ$ are completely $\kappa$-closed posets, $D$ is a $\kappa$-closed dense subset of $\bbQ$, and $\pi : D \to \bbP$ is a $\kappa$-continuous projection.  Then for each $\kappa$-closed dense $E \subseteq \bbP$, $\pi^{-1}[E]$ is a $\kappa$-closed dense subset of $\bbQ$.
\end{lemma}

\begin{proof}
    To show closure, suppose $\delta<\kappa$ and $\la p_\alpha : \alpha<\delta \ra$ is a descending sequence contained in $D$ such that $\pi(p_\alpha) \in E$ for all $\alpha<\delta$.  By the closure of $D$, $p_\delta = \inf_\alpha p_\alpha \in D$, and by the continuity and order-preservation of $\pi$ and the closure of $E$, $\pi(p_\delta) = \inf_\alpha \pi(p_\alpha) \in E$.  
    
    To show density, let $q_0 \in D$.  Let $p_0 \leq \pi(q_0)$ be in $E$, and let $q_1 \leq q_0$ be in $D$ and such that $\pi(q_1) \leq p_0$.  Continue in this manner, building descending sequences $p_0 \geq p_1 \geq p_2 \geq \dots$ and $q_0 \geq q_1 \geq q_2 \geq \dots$ such that $p_i \in E$, $q_i \in D$, and $\pi(q_i) \geq p_i \geq \pi(q_{i+1})$.  Then $\pi(\inf_i q_i) = \inf_i \pi(q_i) = \inf_i p_i \in E$.
\end{proof}

\begin{lemma}
\label{composition}
    Suppose $\bbP_0,\bbP_1,\bbP_2$ are completely $\kappa$-closed posets, and for $i<2$, $D_i \subseteq \bbP_i$ is a $\kappa$-closed dense subset, and $\pi_i : D_i \to \bbP_{i+1}$ is a $\kappa$-continuous projection.  Then there is a $\kappa$-closed dense $E \subseteq \bbP_0$ on which $\pi_1\circ\pi_0$ is defined and is a $\kappa$-continuous projection.
\end{lemma}
\begin{proof}
    Let $E = \pi_0^{-1}[D_1]$.   Then $E$ is dense and $\kappa$-closed by Lemma~\ref{inv_denseclosed}, and $\pi_1\circ\pi_0$ is defined on $E$.  It is clear that $\pi_1\circ\pi_0$ is order-preserving and $\kappa$-continuous.  To show that it is a projection, let $p_0 \in E$ and let $p_2 \leq \pi_1(\pi_0(p_0))$.   Let $p_1 \leq \pi_0(p_0)$ be such that $\pi_1(p_1) \leq p_2$.  Let $p_0' \leq p_0$ be such that $\pi_0(p_0') \leq p_1$.  Then $\pi_1(\pi_0(p_0')) \leq \pi_1(p_1) \leq p_2$.
\end{proof}

The following theorem is true for every regular cardinal $\kappa$. It is attributed to McAloon in \cite[Section 4, Theorem 1]{Grigorieff} for the case $\kappa = \omega$. The uncountable case can be proved by combining McAloon's proof and the ideas in the proof of \cite[Lemma 3]{Magidor82}. 

A proof, under an additional separativity assumption on $\bbP$, also appears in \cite{CummingsHandbook}. We give the proof, adjusted to the non-separative case, for completeness. A similar argument is used in the proof of Lemma \ref{3step}. 

\begin{theorem}
\label{mcaloon}
If $\bbP$ is a $\kappa$-closed poset of size $\lambda>\kappa$ that collapses $\lambda$ to $\kappa$, then there is a dense embedding from $\{ q \in \col(\kappa,\lambda) : \dom q$ is a successor ordinal$\}$ to $\bbP$.  If $\bbP$ is completely $\kappa$-closed, then there is a $\kappa$-continuous dense embedding from $\{ q \in \col(\kappa,\lambda) : \dom q$ is an ordinal$\}$ to $\bbP$.
\end{theorem}
\begin{proof}
    Let $\dot f$ be a $\bbP$-name for a surjection from $\kappa$ to the generic filter $\dot G$.  We build a tree $T \subseteq \bbP$ of height $\kappa$ with the following properties:
    \begin{enumerate}
        \item $1_\bbP$ is the root of $T$.
        \item For each $t \in T$, there is a set $S(t) \subseteq T$ of size $\lambda$ such that for each $s \in S(t)$ $s < t$, and there is no $x \in T$ such that $s<x<t$.
        \item  For each descending chain $\la t_i : i < \delta \ra$ in $T$, where $\delta<\kappa$, if $\inf_i t_i$ exists in $\bbP$, then $\inf_i t_i \in T$.  Otherwise, there is a set of $\lambda$-many pairwise incompatible lower bounds in $T$ to $\la t_i: i <\delta \ra$, all of which are maximal in $T$ among such lower bounds.
        \item $T$ is dense in $\bbP$.
    \end{enumerate}
    This suffices to build the desired dense embedding.  We build $T$ inductively by levels, $\la T_\alpha : \alpha<\kappa \ra$, where $T_\alpha$ is the set of all $t \in T$ such that the chain in $T$ above $t$ has length $\alpha$.

    Suppose we have $\la T_i: i<\alpha\ra$, and inductively assume each $T_i$ is a maximal antichain in $\bbP$.  If $\alpha = \beta +1$, then for each $t \in T_\beta$, there are densely many $p \leq t$ such that for some $q \in \bbP$, $p \Vdash \dot f(\check\beta) = \check q$, and $p \leq q$.  This is because, for any $p \leq t$, there is some $p' \leq p$ and some $q$ such that $p' \Vdash \dot f(\check\beta) = \check q$, and since $p'$ forces that both itself and $q$ are in the generic filter, we can take $p'' \leq p',q$.  Pick a maximal antichain $A_t$ of such $p$ of size $\lambda$, which is possible since $\bbP$ is nowhere $\lambda$-c.c.  Let $T_\alpha = \bigcup \{ A_t : t \in T_\beta \}$.
    
    If $\alpha$ is a limit, then for all chains $\vec c = \la t_i : i<\alpha\ra$ where $t_i \in T_i$, if $\inf_i t_i$ exists, let $A_{\vec c} = \{\inf_i t_i\}$, and otherwise choose a maximal antichain $A_{\vec c}$ of size $\lambda$ of lower bounds to $\la t_i : i < \alpha \ra$.  Let $T_\alpha = \bigcup \{ A_{\vec c} : \vec c$ is a chain through $\la T_i : i < \alpha \ra \}$.  To verify that $T_\alpha$ is a maximal antichain, let $p \in \bbP$.  Recursively choose descending sequences $\la p_i : i <\alpha \ra$ and $\la t_i : i < \alpha\ra$ such that $p_0 = p$, and for all $i<\alpha$, $t_i \in T_i$ and $p_{i+1} \leq p_i,t_i$, using the assumption that each $T_i$ is a maximal antichain.  Then there is a lower bound $p_\alpha$ to the sequences $\la p_i : i<\alpha \ra$ and $\la t_i : i<\alpha\ra$, and $p_\alpha$ is compatible with some $t \in T_\alpha$.

    To show that $T$ is dense in $\bbP$, let $p \in \bbP$.  Let $q \leq p$ be such that for some $\alpha<\kappa$, $q \Vdash \dot f(\check\alpha) = \check p$.  Let $t \in T_{\alpha+1}$ be such that $q$ is compatible with $t$.  Then for some $r\in \bbP$, $t \Vdash \dot f(\check\alpha) = \check r$ and $t \leq r$.  Thus some condition forces $\dot f(\check\alpha) = \check p = \check r$, and therefore $p = r \geq t \in T$.
\end{proof}

\begin{lemma}\label{lem:embedding-from-projection}
Suppose $\kappa$ is a regular uncountable cardinal,
$\bbP$ and $\bbQ$ are completely $\kappa$-closed posets, $|\bbQ| =\lambda > |\bbP|$, $\bbQ$ collapses $\lambda$ to $\kappa$, and $\pi : \bbQ \to \bbP$ is a $\kappa$-continuous projection.  

Then, there is a $\kappa$-closed dense set $D \subseteq \bbP \times \col(\kappa,\lambda)$ and a $\kappa$-continous dense embedding $e : D \to \bbQ$ with the property that $\pi(e(p,c)) = p$ for all $\la p,c \ra \in D$.
\end{lemma}

\begin{proof}
Whenever $G \subseteq \bbP$ is generic, $\pi^{-1}[G]$ has a $\kappa$-closed dense subset isomorphic to the tree of conditions in $\col(\kappa,\lambda)$ whose domain is an ordinal.  This is because $\pi$ is $\kappa$-continuous, so if $\delta<\kappa$ and $\la q_i : i <\delta \ra \subseteq \bbQ$ is a descending sequence such that $\pi(q_i) \in G$, then $\pi(\inf_{i<\delta} q_i) \in G$ as well, implying that $\bbQ_n/G$ is completely $\kappa$-closed. By Theorem~\ref{mcaloon}, in $V[G]$ there is a $\kappa$-continuous dense embedding $\varphi : \col(\kappa,\lambda) \to \bbQ/G$.  Let $\dot\varphi$ be a name for such an embedding.

For $\la p,c\ra \in \bbP \times \col(\kappa,\lambda)$, if there is $q \in \bbQ$ such that $p = \pi(q)$ and $p \Vdash \dot\varphi(\check c) = \check q$, then let $e(p,c) = q$.  Since $\dot\varphi$ is forced to be a function, such a $q$ is unique if it exists.  Since $\dot\varphi$ is forced to be order-preserving, $e$ is also order-preserving.  If $e(p_1,c_1) = q_1 \leq e(p_0,c_0) = q_0$, then $\pi(q_1) = p_1 \leq \pi(q_0) = p_0$, and since $p_1 \Vdash \dot\varphi(\check c_1) = \check q_1 \leq \check q_0 = \dot\varphi(\check c_0)$, we must have $c_1 \leq c_0$.  Thus $e$ is an order-isomorphism with its range.  It is easy to see that the domain of definition of $e$ is $\kappa$-closed, and $e$ is $\kappa$-continuous.

Let us show that the domain of definition of $e$ is dense.  Let $\la p_0,c_0 \ra \in \bbP \times \col(\kappa,\lambda)$ be arbitrary.  There is some $p_0' \leq p_0$ and some $q_0 \in \bbQ$ such that $p_0' \Vdash \check q_0 \in \check\bbQ/\dot G$ and $\dot\varphi(\check c_0) = \check q_0$.  
Since $p_0'$ and $\pi(q_0)$ are compatible, there is $q_0' \leq q_0$ such that $\pi(q_0') \leq p_0'$.
Then we can find $p_1 \leq p_0'$, $c_1 \leq c_0$, and $q_1 \leq q_0'$ such that $p_1 \Vdash \dot\varphi(\check c_1) = \check q_1$, and $p_1\leq\pi(q_1)$.
Then we can find $q_1' \leq q_1$ such that $\pi(q_1') \leq p_1$, and subsequently find $p_2 \leq \pi(q_1')$, $c_2 \leq c_1$, and $q_2 \leq q_1'$ such that $p_2 \Vdash \dot\varphi(\check c_2) = \check q_2$, and $p_2\leq\pi(q_2)$.
Continue in this way, constructing sequences $p_0 \geq p_1 \geq \dots$, $c_0 \geq c_1 \geq \dots$, and $q_0 \geq q_1 \geq\dots$, so that for $i\geq 1$, $p_i \Vdash \dot\varphi(\check c_i) = \check q_i$ and $p_{i+1} \leq \pi(q_{i+1}) \leq p_i$.  Define $p_\omega = \inf_{i} p_i$, $c_\omega = \inf_{i} c_i$, and $q_\omega = \inf_{i} q_i$.  By the continuity of $\pi$ and $\dot\varphi$, $p_\omega \Vdash \dot\varphi(\check c_\omega) = \check q_\omega$, and $\pi(q_\omega) = p_\omega$.

Finally, let us show that the range of $e$ is dense.  Let $q \in \bbQ$ be arbitrary.  There is $\la p_0,c_0\ra\in\bbP \times \col(\kappa,\lambda)$ and $q_0 \leq q$ such that $p_0 \Vdash \dot\varphi(\check c_0) = \check q_0$.
Then find $q_0' \leq q_0$ such that $\pi(q_0') \leq p_0$, and subsequently find $p_1 \leq \pi(q_0')$, $c_1 \leq c_0$, and $q_1 \leq q_0'$ such that $p_1 \Vdash \dot\varphi(\check c_1) = \check q_1$, and $p_1\leq\pi(q_1)$.
Continue in this way, constructing sequences $p_0 \geq p_1 \geq \dots$, $c_0 \geq c_1 \geq \dots$, and $q_0 \geq q_1 \geq\dots$, such that $p_i \Vdash \dot\varphi_n(\check c_i) = \check q_i$, and $p_{i+1} \leq \pi(q_{i+1}) \leq p_i$.  Define $p_\omega = \inf_{i} p_i$, $c_\omega = \inf_{i} c_i$, and $q_\omega = \inf_{i} q_i \leq q$.  By the continuity of $\pi$ and $\dot\varphi$, $p_\omega \Vdash \dot\varphi(\check c_\omega) = \check q_\omega$, and $\pi(q_\omega) = p_\omega$.
\end{proof}   

\subsection{Inverse limits}
Suppose $\vec\bbP = \la \bbP_\alpha : \alpha < \delta \ra$ is a sequence of posets and $\vec\pi = \la \pi_{\beta\alpha} : \alpha<\beta<\delta \ra$ is a sequence such that $\pi_{\beta\alpha} : \bbP_\beta \to \bbP_\alpha$ is a projection.  Suppose also that whenever $\alpha<\beta<\gamma<\delta$, then $\pi_{\gamma\alpha} = \pi_{\beta\alpha}\circ\pi_{\gamma\beta}$.   Such a pair of sequences is called an \emph{inverse system}, and $\delta$ is referred to as the \emph{length} of the system.  We define the \emph{inverse limit} of the system, $\varprojlim(\vec\bbP,\vec\pi)$, to be the set of all sequences $\la p_i : i <\delta \ra$ such that for $\alpha<\beta<\delta$, $p_\alpha = \pi_{\beta\alpha}(p_\beta)$.  The ordering on $\varprojlim(\vec\bbP,\vec\pi)$ is pointwise.

We also allow inverse systems on which the maps $\pi_{\beta\alpha}$ are not necessarily defined at every point of $\bbP_\beta$, and the maps are not assumed to commute everywhere.
In this case, we take $\varprojlim(\vec\bbP,\vec\pi)$, to be the set of all sequences $\la p_i : i <\delta \ra$ such that for all $\alpha<\beta<\delta$, $p_\beta \in \dom \pi_{\beta\alpha}$ and $p_\alpha = \pi_{\beta\alpha}(p_\beta)$.  

Given an inverse system like this, recursively define sets $E_\alpha^{\la\vec\bbP,\vec\pi\ra} \subseteq \bbP_\alpha$ as follows.  Let $E_0^{\la\vec\bbP,\vec\pi\ra}  = \bbP_0$, and for $0<\beta$, let $E_\beta^{\la\vec\bbP,\vec\pi\ra}  = \bigcap_{\alpha<\beta} \pi^{-1}_{\beta\alpha}[E_\alpha^{\la\vec\bbP,\vec\pi\ra} ]$.  
It is easy to check by induction that for $\alpha<\delta$, $p\in E_\alpha^{\la\vec\bbP,\vec\pi\ra}$ if and only if for each finite sequence $\alpha_0<\alpha_1<\dots<\alpha_n=\alpha$, $\pi_{\alpha_1\alpha_0}\circ\dots\circ\pi_{\alpha_n\alpha_{n-1}}(p)$ is defined.

Sometimes, for notational convenience, we will refer to a map $\pi_{\alpha\alpha}$ in an inverse system, by which we mean the identity map.

\begin{lemma}
    Suppose $\la\vec\bbP,\vec\pi\ra$ is an inverse system of length $\delta$, each $\bbP_\alpha$ is separative, and for all $\alpha<\delta$, the map $\vec p \mapsto \vec p(\alpha)$ is a projection.  Then $\varprojlim(\vec\bbP,\vec\pi)$ is separative.
\end{lemma}

\begin{proof}
    Suppose $\vec p,\vec q \in \varprojlim(\vec\bbP,\vec\pi)$ and $\vec p \nleq \vec q$.  Then for some $\alpha<\delta$, $\vec p(\alpha) \nleq \vec q(\alpha)$.  There is $x \leq \vec p(\alpha)$ such that $x \perp \vec q(\alpha)$, and by hypothesis, there is $\vec r \leq \vec p$ such that $\vec r(\alpha) \leq x$.  Clearly $\vec r \perp \vec q$.
\end{proof}

\begin{definition}
    For a regular uncountable cardinal $\kappa$, a pair $\la\vec\bbP,\vec\pi\ra$ is called a \emph{$\kappa$-good inverse system} when for some ordinal $\delta$,
    \begin{enumerate}
        \item $\vec\bbP = \la\bbP_i : i <\delta\ra$ is a sequence of completely $\kappa$-closed posets.
        \item $\vec\pi = \la \pi_{ji} : i<j<\delta \ra$ is a sequence of maps such that for all $i<j<\delta$, $D_{ji} = \dom \pi_{ji}$ is a dense $\kappa$-closed subset of $\bbP_j$, and $\pi_{ji} : D_{ji} \to \bbP_i$ is a $\kappa$-continuous projection. 
        \item  For $i<j<k<\delta$, $\pi_{ki} = \pi_{ji} \circ \pi_{kj}$ at all points for which these are defined.
    \end{enumerate}
\end{definition}

\begin{lemma}
\label{invlim_ctble}
Suppose $\kappa$ is a regular uncountable cardinal and $\la\vec\bbP,\vec\pi\ra$ is a $\kappa$-good inverse system of length $\omega$.  Then:
    \begin{enumerate}
        \item $\varprojlim(\vec\bbP,\vec\pi)$ is completely $\kappa$-closed.
        \item For each $i<\omega$, the map $\vec p \mapsto \vec p(i)$ is a $\kappa$-continuous projection from $\varprojlim(\vec\bbP,\vec\pi)$ to $\bbP_i$.
    \end{enumerate}
\end{lemma}
\begin{proof}
    To show the first claim, suppose $\delta<\kappa$ and $\la\vec p_\alpha : \alpha < \delta\ra$ is a descending sequence in $\varprojlim(\vec\bbP,\vec\pi)$.  Then for each $i<\omega$, $\la\vec p_\alpha(i): \alpha<\delta\ra$ is a descending sequence with an infimum $p_i^*$.  Let $\vec p_\delta = \la p_i^* : i<\omega \ra$.  The continuity of each $\pi_{ji}$ implies that $\vec p_\delta \in\varprojlim(\vec\bbP,\vec\pi)$.  If $\vec q$ is any other lower bound to the sequence $\la\vec p_\alpha : \alpha < \delta\ra$, then we would have $\vec q(i) \leq p^*_i$ for each $i$, so $\vec p_\delta = \inf_\alpha \vec p_\alpha$.
    
    For the second claim, first note that $\kappa$-continuity follows from the above paragraph, since if $\delta<\kappa$  and $\la \vec p_\alpha : \alpha < \delta \ra$ is descending, then $\inf_\alpha \vec p_\alpha$ is given by taking $\inf_\alpha \vec p_\alpha(i)$ at each coordinate.
    Now let $E_i = E_i^{\la\vec\bbP,\vec\pi\ra}$ for $i<\omega.$
    By Lemmas~\ref{int_denseclosed} and \ref{inv_denseclosed}, each $E_i$ is a $\kappa$-closed dense subset of $\bbP_i$.  
    To show that $\vec p \mapsto \vec p(n)$ is a projection, suppose $\vec p \in \varprojlim(\vec\bbP,\vec\pi)$ and $r \leq \vec p(n)$.  Let $q_n \in E_n$ be below $r$.  Using the fact that $\pi_{n+1,n}$ is a projection, let $q_{n+1} \in E_{n+1}$ be such that $q_{n+1} \leq \vec p(n+1)$ and $\pi_{n+1,n}(q_{n+1}) \leq q_n$.  
     Assume inductively that $n < j$ and we have chosen $q_i \in E_i$ for $n\leq i\leq j$ such that
     $$p(i) \geq q_i \geq \pi_{i+1,i}(q_{i+1}) \geq \pi_{i+2,i}(q_{i+2}) \geq\dots\geq \pi_{ji}(q_j).$$
     Then pick $q_{j+1} \in E_{j+1}$ such that $q_{j+1} \leq \vec p(j+1)$ and $\pi_{j+1,j}(q_{j+1}) \leq q_j$.
    The induction carries forward, and after $\omega$ steps, we have that for all $i<\omega$ $\la \pi_{ji}(q_j) : \max\{n,i\}<j<\omega \ra$ is a descending sequence in $E_i$ with infimum $q_i^*$.  If we define $\vec q = \la q^*_i : i < \omega \ra$, then $\vec q \in \varprojlim(\vec\bbP,\vec\pi)$ by the continuity of each $\pi_{ji}$, $\vec q \leq \vec p$, and $\vec q(n) = q^*_n \leq r$. 
\end{proof}

To make a similar claim about inverse systems of longer lengths, we need to assume that they behave nicely at limit points.

\begin{lemma}
\label{invlim_general}
    Assume the following:
    \begin{enumerate}
        \item  $\kappa$ is a regular uncountable cardinal and $\delta<\kappa$.
        \item $\la\vec\bbP,\vec\pi\ra$ is a $\kappa$-good inverse system of length $\delta$.
        \item 
        For all limit $\gamma<\delta$, the map $\sigma_\gamma : E_\gamma^{\la\vec\bbP,\vec\pi\ra} \to \varprojlim(\vec\bbP{\rest}\gamma,\vec\pi {\rest}\gamma)$ defined by $\sigma_\gamma(p) = \la \pi_{\gamma\alpha}(p) : \alpha<\gamma \ra$ is a projection.
    \end{enumerate}  
    Then $\varprojlim(\vec\bbP,\vec\pi)$ is completely $\kappa$-closed, and for each $i<\delta$, the map $\vec p \mapsto \vec p(i)$ is a $\kappa$-continuous projection from $\varprojlim(\vec\bbP,\vec\pi)$ to $\bbP_i$.
\end{lemma}

\begin{proof}
    Complete $\kappa$-closure and $\kappa$-continuity follow by the same argument as in the case $\delta = \omega$.

    Let us show that the map $\vec p \mapsto \vec p(i)$ is a projection.  Suppose $\vec p \in \varprojlim(\vec\bbP,\vec\pi)$, $\alpha<\delta$, and $r \leq \vec p(\alpha)$.    Let $E_i = E_i^{\la\vec\bbP,\vec\pi\ra}$ for $i<\delta$.  Note that by Lemmas~\ref{int_denseclosed} and \ref{inv_denseclosed}, each $E_i$ is a $\kappa$-closed dense subset of $\bbP_i$.   Let $q_\alpha \in E_\alpha$ be below $r$. 

    Assume inductively that $\alpha < \gamma<\delta$ and we have chosen $q_i \in E_i$ for $\alpha\leq i<\gamma$ such that $q_i \leq \vec p(i)$, and for all $\beta<\gamma$, $\la \pi_{i\beta}(q_i) : \max\{\alpha,\beta\}\leq i< \gamma \ra$ is a decreasing sequence in $E_\beta$.  If $\gamma = \gamma'+1$,  pick $q_\gamma \in E_\gamma$ such that $q_\gamma \leq \vec p(\gamma)$ and $\pi_{\gamma\gamma'}(q_\gamma) \leq q_{\gamma'}$.
     
    If $\gamma$ is a limit, first define a $\gamma$-sequence $\vec r_\gamma$ by \[\vec r_\gamma(\beta) =  \inf \{ \pi_{i\beta}(q_i) : \max\{\alpha,\beta\} \leq i < \gamma \ra \} \in E_\beta.\]  Using $\kappa$-continuity, it is easy to check that $\vec r_\gamma \in \varprojlim(\vec\bbP{\rest}\gamma,\vec\pi {\rest}\gamma)$ and $\vec r_\gamma \leq \vec p\rest\gamma$.  Then use our extra assumption to find $q_\gamma \in E_\gamma$ such that $\sigma_\gamma(q_\gamma)\leq \vec r_\gamma$.  This allows the induction to continue.  In the end, define $\vec q \in \prod_i E_i$ by $\vec q(i) = \inf \{ \pi_{ji}(q_j) : \max\{\alpha,i\} \leq j < \delta \ra \}$.  Just as before, we have that $\vec q \in \varprojlim(\vec\bbP,\vec\pi)$, $\vec q \leq \vec p$, and $\vec q(\alpha)\leq r$.
\end{proof}

\begin{lemma}
\label{invlim_dense}
    Under the same hypotheses as in Lemma~\ref{invlim_general}, assume additionally that for each $\alpha<\delta$, we are given a dense $\kappa$-closed $D_\alpha \subseteq \bbP_\alpha$.  Then $\varprojlim(\vec\bbP,\vec\pi) \cap \prod_\alpha D_\alpha$ is a dense $\kappa$-closed subset of $\varprojlim(\vec\bbP,\vec\pi)$.
\end{lemma}

\begin{proof}
    For $i<j<\delta$, replace $\pi_{ji}$ with $\pi'_{ji} = \pi_{ji} \rest D_j$.  It is easy to see that \[\varprojlim(\vec\bbP,\vec\pi) \cap \prod_\alpha D_\alpha = \varprojlim(\vec\bbP,\vec\pi').\]  $\varprojlim(\vec\bbP,\vec\pi')$ is completely $\kappa$-closed by Lemma~\ref{invlim_general}, and since glb's are computed pointwise, $\varprojlim(\vec\bbP,\vec\pi')$ is a $\kappa$-closed subset of $\varprojlim(\vec\bbP,\vec\pi)$.
    
    To show that $\varprojlim(\vec\bbP,\vec\pi')$ is dense in $\varprojlim(\vec\bbP,\vec\pi)$, let $\vec p \in \varprojlim(\vec\bbP,\vec\pi)$. 
    Let $E_0 = D_0$, and
    for $0<\alpha<\delta$, define $E_\alpha = E^{\la\vec\bbP,\vec\pi'\ra}_\alpha$. 
    Let $q_0 \leq \vec p(0)$ be in $E_0$.

    Assume inductively that $\gamma<\delta$ and we have chosen $q_{i} \in E_{i}$ for $i<\gamma$ such that $q_{i} \leq \vec p(i)$, and $\la \pi_{ji}(q_j) : i \leq j< \gamma \ra$ is a decreasing sequence in $E_{i}$.  If $\gamma = \gamma'+1$,  pick $q_{\gamma} \in E_{\gamma}$ such that $q_{\gamma} \leq \vec p(\gamma)$ and $\pi_{\gamma\gamma'}(q_\gamma) \leq q_{\gamma'}$.
     If $\gamma$ is a limit, first define $\vec r \in \varprojlim(\vec\bbP{\rest} \gamma,\vec\pi {\rest} \gamma)$
    by $\vec r(i) = \inf \{ \pi_{ji}(q_j) : i \leq j <\gamma \ra \} \in E_i$. 
    We have that $\vec r \leq \vec p \rest \gamma$.  
    Then find $q_{\gamma} \in E_{\gamma}$ below $\vec p(\gamma)$ such that $\sigma_{\gamma}(q_{\gamma}) \leq \vec r$.

    After $\delta$-many steps, define $q_i^* =  \inf \{ \pi_{ji}(q_j) : i \leq j <\delta \ra \}  \in E_i$, and let $\vec q = \la q_i^* : i <\delta \ra$.  We have that $\vec q \in \varprojlim(\vec\bbP,\vec\pi')$ and $\vec q \leq \vec p$ in $\varprojlim(\vec\bbP,\vec\pi)$.
\end{proof}

\begin{lemma}
\label{invlim_cof1}
Assume the following:
    \begin{enumerate}
        \item  $\kappa$ is a regular uncountable cardinal and $\delta<\kappa$.
        \item $\la\vec\bbP,\vec\pi\ra$ is a $\kappa$-good inverse system of length $\delta$.
        \item $X \subseteq \delta$ is cofinal.
        \item For all $\gamma \in X$, the map $\sigma_{\gamma} : E_\gamma^{\la\vec\bbP{\rest} X,\vec\pi{\rest} X \ra} \to \varprojlim(\vec\bbP\rest (X\cap\gamma),\vec\pi \rest (X\cap\gamma))$ defined by $\sigma_\gamma(p) = \la \pi_{\gamma\alpha}(p) : \alpha\in X \cap\gamma \ra$ is a projection.
    \end{enumerate}
    Then the map $\vec p \mapsto \vec p \rest X$ is a $\kappa$-continuous isomorphism of $\varprojlim(\vec\bbP,\vec\pi)$ with a dense $\kappa$-closed subset of $\varprojlim(\vec\bbP \rest X,\vec\pi \rest X)$.
\end{lemma}

\begin{proof}
    The map $\vec p \mapsto \vec p \rest X$ is clearly order-preserving.
    It is injective since if $\vec p_0 \not= \vec p_1$, then $\vec p_0(\alpha) \not= \vec p_1(\alpha)$ for some $\alpha<\delta$, and thus $\vec p_0(\beta) \not= \vec p_1(\beta)$ for all $\beta\in [\alpha,\delta)$, and hence $\vec p_0 \rest X \not= \vec p_1 \rest X$ since $X$ is cofinal.  If $\vec p_0 \rest X \leq \vec p_1 \rest X$, then $\vec p_0 \leq \vec p_1$ since $X$ is cofinal and each $\pi_{ji}$ is order-preserving, so the map is an order-isomorphism with its range.  The map is clearly $\kappa$-continuous since glb's are computed pointwise.

    To show that the range is dense, note that it is easy to see by induction that for all $\alpha<\delta$, $E_\alpha = E_\alpha^{\la\vec\bbP,\vec\pi\ra}$ is a $\kappa$-closed dense subset of $E_\alpha^{\la\vec\bbP\rest X,\vec\pi\rest X\ra}$.  By Lemma~\ref{invlim_dense}, $E^* = \varprojlim(\vec\bbP{\rest} X,\vec\pi{\rest} X) \cap \prod_{\alpha \in X} E_\alpha$ is a $\kappa$-closed dense subset of $\varprojlim(\vec\bbP{\rest} X,\vec\pi{\rest} X)$.  We claim that the map $\vec p \mapsto \vec p \rest X$ is a surjection from $\varprojlim(\vec\bbP,\vec\pi)$ to $E^*$.
    For $\vec q \in E^*$, define $\vec p \in \varprojlim(\vec\bbP,\vec\pi)$ by $\vec p(\alpha) = \vec q(\alpha)$ for $\alpha \in X$ and for $\alpha \in \delta\setminus X$, $\vec p(\alpha) = \pi_{\beta\alpha}(\vec q(\beta))$ for $\beta = \min(X\setminus\alpha)$.  We have $\vec p(\alpha) \in E_\alpha$ since $\pi_{\beta\alpha}[E_\beta]\subseteq E_\alpha$.
\end{proof}

\begin{prop}
\label{invlim_cof2}
Assume the following:
    \begin{enumerate}
        \item  $\kappa$ is a regular uncountable cardinal and $\delta$ is an ordinal such that $\cf(\delta)<\kappa$.
        \item $\la\vec\bbP,\vec\pi\ra$ is a $\kappa$-good inverse system of length $\delta$.
        \item For all $X \subseteq \delta$ of size $<\cf(\delta)$ with a maximum element $\gamma$, the map $\sigma_\gamma : E_\gamma^{\la\vec\bbP{\rest} X,\vec\pi{\rest} X \ra} \to \varprojlim(\vec\bbP\rest (X\cap\gamma),\vec\pi \rest (X\cap\gamma))$ defined by $\sigma_\gamma(p) = \la \pi_{\gamma\alpha}(p) : \alpha\in X \cap\gamma \ra$ is a projection.
    \end{enumerate}
    Then for all cofinal $A,B \subseteq \delta$ of size $<\kappa$, 
    $\varprojlim(\vec\bbP{\rest} A,\vec\pi{\rest} A)$ and 
    $\varprojlim(\vec\bbP{\rest} B,\vec\pi{\rest} B)$
    have isomorphic $\kappa$-closed dense subsets.
\end{prop}

\begin{proof}
    Let $A, B \subseteq \delta$ be as hypothesized, and let $A' \subseteq A$ and $B' \subseteq B$ be cofinal and of order type $\cf(\delta)$.  Let $C = A \cup B$.  By Lemma~\ref{invlim_cof1}, the restriction maps to $A'$ and $B'$ are $\kappa$-continuous dense embeddings of $\varprojlim(\vec\bbP{\rest} C,\vec\pi{\rest} C)$ into $\varprojlim(\vec\bbP{\rest} A',\vec\pi{\rest} A')$ and $\varprojlim(\vec\bbP{\rest} B',\vec\pi{\rest} B')$ respectively.  The restriction maps to $A$ and $B$ are also $\kappa$-continuous dense embeddings, because for example if $\vec p \in \varprojlim(\vec\bbP{\rest} A,\vec\pi{\rest} A)$, then $\vec p \rest A' \in \varprojlim(\vec\bbP{\rest} A',\vec\pi{\rest} A')$, and there is $\vec q \in \varprojlim(\vec\bbP{\rest} C,\vec\pi{\rest} C)$ such that $\vec q \rest A' \leq \vec p \rest A'$.  Hence $\vec q \rest A \leq \vec p$, since $A'$ is cofinal in $A$.
\end{proof}

\subsection{Reasonable collapses}

\begin{prop}
Suppose $\bbP,\bbQ$ are separative posets, $\bbP$ is a regular suborder of $\bbQ$, and $\pi : \bbQ \to \bbP$ is a projection such that $\pi \rest \bbP = \mathrm{id}$.  Then for all $q \in \bbQ$, $q \leq \pi(q)$, and for all $p \in \bbP$ and $q \in \bbQ$, $p \perp q$ if and only if $p \perp \pi(q)$.
\end{prop}

\begin{proof}

    For $q \in\bbQ$, $q \Vdash \pi(q) \in \pi(\dot G) = \dot G \cap\bbP$, so by separativity, $q\leq \pi(q)$.

    Suppose $p \in \bbP$ and $q \in \bbQ$.  If $r\leq p,q$, then $\pi(r) \leq p,\pi(q)$ since $\pi(p) = p$.  If $r \leq p,\pi(q)$, then there is $q' \leq q$ such that $\pi(q') \leq r$, and then by the first claim, $q' \leq \pi(q') \leq p$, so $q' \leq p,q$.
\end{proof}

The following is meant to abstract some general features of 
$\col(\kappa,{<}\lambda)$ and similar collapsing posets.

\begin{definition}
Suppose $\kappa$ is regular and $\lambda>\kappa$ is inaccessible.  A poset $\bbP$ is called a \emph{reasonable $(\kappa,\lambda)$-collapse} when:
\begin{enumerate}
    \item $\bbP$ is $\lambda$-c.c.\ and completely $\kappa$-closed.
    \item There is a $\subseteq$-increasing sequence of regular suborders $\la \bbP_\alpha : \alpha<\lambda \ra$.
    \item There is a sequence of maps $\la \pi_\alpha : \alpha < \lambda \ra$ such that for all $\alpha$, $\pi_\alpha : \bbP \to \bbP_\alpha$ is a $\kappa$-continuous projection, $\pi_\alpha \rest \bbP_\alpha = \mathrm{id}$, and for $\alpha<\beta$, $\pi_\alpha = \pi_\alpha \circ \pi_\beta$.
    \item For unboundedly many regular $\alpha<\lambda$, $|\bbP_{\alpha}| = \alpha$ and $\Vdash_{\bbP_{\alpha}} |\alpha| = \kappa$.
\end{enumerate}
\end{definition}

\begin{definition}
    A poset $\bbP$ is called a \emph{quite reasonable} $(\kappa,\lambda)$-collapse when there are $\la \bbP_\alpha : \alpha<\lambda\ra$ and $\la\pi_\alpha : \alpha<\lambda\ra$ witnessing that $\bbP$ is reasonable, with the following additional property:
    
    Whenever $\cf(\alpha)<\kappa$, the map $p \mapsto \la \pi_\beta(p) : \beta<\alpha\ra$ is an isomorphism from $\bbP_\alpha$ to $\varprojlim(\vec\bbP \rest \alpha,\vec\pi\rest\alpha)$, where $\vec\pi\rest\alpha = \la \pi_\gamma \rest \bbP_{\beta} : \gamma<\beta<\alpha \ra$.
   
\end{definition}

Let us remark that the Levy, Silver, Laver, and Easton collapses are all quite reasonable.

\section{Dual Shioya forcing}\label{sec: Dual-shioya-collapse}
In this section, we will present our main building block. 

In \cite{ShioyaDense}, Shioya defined a forcing notion that generically glues together copies of the Levy collapse $\Col(\kappa, \mu)$ for $\mu < \lambda$, using complete embeddings with $\kappa$-closed quotients,
and showed how to use it in order to get dense ideals. Here we will use a variant of this forcing, using projections instead of embeddings. Those two definitions yield equivalent forcings,\footnote{Using Lemma \ref{lem:embedding-from-projection}, one can obtain a complete embedding with a $\kappa$-closed quotient from a $\kappa$-continuous projection. On a dense subset of conditions $\langle \bbP, p, \dot\tau\ra$ of $\Sh(\kappa,\lambda)$, such that $\mathbb{P}$ is isomorphic to $\Col(\kappa,\mu)$, $\mu$ is a regular cardinal, and $\tau$ codes the generic of $\mathbb{P}$, this modification is order-preserving from our poset to Shioya's poset.} but using projections seems to be more natural with respect to the uniformization forcing in Section \ref{sec: uniformization}. We will not use the equivalence of those forcings in the proofs.
\begin{definition}
Let $\omega < \kappa < \lambda$ be regular cardinals, $\lambda$ inaccessible. $\Sh(\kappa,\lambda)$\footnote{$\Sh$ is the Japanese hiragana character representing the first syllable of the name, ``Shioya''.  We caution that the associated sound is also the reading of the Japanese word for ``death'', so before Japanese audiences, it may be advisable to pronounce the full name.} 
consists of all triples of the form $\langle \mathbb{P}, p, \tau\rangle$ where:
\begin{enumerate}
    \item $\mathbb{P}\in V_\lambda$ is a separative completely $\kappa$-closed poset.
    \item $p\in \mathbb{P}$.
    \item For some $\gamma<\lambda$, $\tau$ is a $\mathbb{P}$-name for a function from $\gamma$ to $2$.
\end{enumerate}
Let us define the order of  $\Sh(\kappa,\lambda)$. Let $\langle \mathbb{P}_0, p_0,\tau_0\rangle \leq \langle \mathbb{P}_1, p_1,\tau_1\rangle$ if there is a $\kappa$-continuous projection $\pi$ from a $\kappa$-closed dense subset of $\mathbb{P}_0$ to $\mathbb{P}_1$ such that $p_0 \Vdash p_1 \in \pi(\dot{G})$, where $\dot{G}$ is the canonical name for the $\bbP_0$-generic filter, and $p_0 \Vdash \pi^*(\tau_1) \trianglelefteq \tau_0$, where $\pi^*$ is the canonical translation of $\bbP_1$-names into $\bbP_0$-names via the projection $\pi$.
\end{definition}

\begin{lemma}
    The ordering on $\Sh(\kappa,\lambda)$ is transitive.
\end{lemma}

\begin{proof}
    Suppose $\langle \mathbb{P}_0, p_0,\tau_0\rangle \leq \langle \mathbb{P}_1, p_1,\tau_1\rangle$ and $\langle \mathbb{P}_1, p_1,\tau_1\rangle \leq \langle \mathbb{P}_2, p_2,\tau_2\rangle$.
    Let $D_0,D_1$ and $\pi_{0}$,$\pi_{1}$ be such that for $i = 0,1$, $D_i$ is a $\kappa$-closed dense subset of $\bbP_i$, and $\pi_{i} : D_i \to \bbP_{i+1}$ is a $\kappa$-continuous projection.  By Lemma~\ref{composition}, there is a $\kappa$-closed dense $E \subseteq \bbP_0$ on which $\pi_1\circ\pi_0 : E \to \bbP_2$ is a $\kappa$-continuous projection.  Let $\sigma = \pi_1\circ\pi_0$.


    The $\sigma$-image of a generic $G_0$ for $\bbP_0$ induces a generic $G_2$ for $\bbP_2$, and this is the same as taking the filter $G_1$ induced by the $\pi_{0}$-image of $G_0$ and then taking the filter induced by the $\pi_{1}$-image of $G_1$.  Since $p_0$ forces that $p_1 \in \dot G_1$ and $p_1$ forces that $p_2 \in \dot G_2$, $p_0$ forces that $p_2 \in \dot G_2$.  Since $p_0$ forces that 
    $\pi_{0}^*(\tau_1) \trianglelefteq \tau_0$ 
    and $p_1$ forces that 
    $\pi_{1}^*(\tau_2) \trianglelefteq \tau_1$,
    it follows that $p_0$ forces that
    $\sigma^*(\tau_2) \trianglelefteq \tau_0$.
    \end{proof}

The following lemma shows the universality of $\Sh(\kappa,\lambda)$. A similar property was shown by Shioya \cite{ShioyaDense}, using a very similar proof. The same property holds for the Anonymous collapse of \cite{Eskew2016}.

\begin{lemma}
\label{mainproj}
    Suppose $\bbR$ is a reasonable $(\kappa,\lambda)$-collapse. Then, there is a dense subset $D \subseteq \bbR * \dot\Add(\lambda)$ and a projection $\psi  :  D \to \Sh(\kappa,\lambda)$.
\end{lemma}
\begin{proof}
Suppose $\la \bbR_\alpha : \alpha<\lambda \ra$ and $\la \rho_{\alpha} : \alpha<\lambda \ra$ witness that $\bbR$ is a reasonable $(\kappa,\lambda)$-collapse.  We can take $\bbR_0$ to be $\{ 1_\bbR\}$. 
Let $D$ be the set of $\la r,\dot s \ra$ such that for some $\alpha<\lambda$, $r \in \bbR_\alpha$, $\dot s$ is a $\bbR_\alpha$-name, and $\dot s$ is forced (by $1_\bbR$) to have domain $\alpha$.  Using the $\lambda$-c.c., it is easy to see that $D$ is dense.  For $\la r,\dot s \ra \in D$ with (unique) witnessing ordinal $\alpha$, let us call $\alpha$ the \emph{length} of $\la r,\dot s \ra$.  For $\la r,\dot s \ra \in D$ of length $\alpha$, let $\psi(r,\dot s) = \la \bbR_\alpha,r,\dot s\ra$.  $\psi$ is clearly order-preserving, with the maps $\rho_\alpha \rest \bbR_\beta$ for $\alpha\leq\beta$ witnessing the order relations in $\ran\psi$.  Note that $\psi(1_\bbR,\check\emptyset) = \la \{ 1\},1,\check\emptyset \ra$ is a maximal element of $\Sh(\kappa,\lambda)$.

Suppose $\la r_0,\dot s_0 \ra \in D$ has length $\alpha_0$ and $\la\bbP,\bar p,\tau\ra \leq \la\bbR_{\alpha_0},r_0,\dot{s}_0\ra$.  Let $\pi_0 : \bbP \to \bbR_{\alpha_0}$ be a $\kappa$-continuous projection defined on a dense $\kappa$-closed subset of $\bbP$ witnessing the ordering.  Strengthening the condition if necessary, we may assume $\bar p \in \dom\pi_0$.   Let $\alpha_1$ be a regular cardinal such that $\alpha_0 \leq \alpha_1 < \lambda$, $|\bbP| < \alpha_1$,  $\Vdash_\bbP \dom\tau < \check{\alpha}_1$, $|\bbR_{\alpha_1}| = \alpha_1$ and $\Vdash_{\bbR_{\alpha_1}} |\alpha_1| = \kappa$.  Extend $\la\bbP,\bar p,\tau\ra$ to $\la\bbQ,\bar q,\sigma\ra$, where $\bbQ = \bbP \times \col(\kappa,\alpha_1)$, $\bar q = \la \bar p,1 \ra$, and $\sigma$ is forced to be an end-extension of $\mathrm{pr}_0^*(\tau)$ of length $\alpha_1$, where $\mathrm{pr}_0$ is the projection $\la a,b\ra \mapsto a$.

For notational convenience, let $\bbQ_0 = \bbQ$, let $\bbQ_1 = \bbR_{\alpha_1}$, let $\chi_0 = \pi_0 \circ \mathrm{pr}_0$, and let $\chi_1 = \rho_{\alpha_0} \rest \bbR_{\alpha_1}$.



\begin{claim}
\label{treefactor}
    There are $e_0,e_1$ such that on some $\kappa$-closed dense sets, the following diagram commutes. 

\[
\begin{tikzcd}
    & \mathbb{Q}_1=\mathbb{R}_{\alpha_1}\arrow{rrd}{\chi_1} & & \\
    \mathbb{R}_{\alpha_0} \times \Col(\kappa,\alpha_1) \arrow[dashed]{ur}{e_1} \arrow[dashed]{dr}{e_0} \arrow{rrr}{\mathrm{pr}_0}& &  & \mathbb{R}_{\alpha_0} \\
    & \mathbb{Q}_0=\mathbb{P} \times \Col(\kappa,\alpha_1) \arrow{urr}{\chi_0} 
 \arrow{r}{\mathrm{pr}_0} & \mathbb{P}\arrow{ur}{\pi_0} & 
\end{tikzcd}
\]
    Namely, for $n<2$, there is a $\kappa$-closed dense $E_n \subseteq \bbR_{\alpha_0} \times \col(\kappa,\alpha_1)$ and a $\kappa$-continuous dense embedding $e_n : E_n \to \bbQ_n$, with the property that if $\la r,c \ra \in E_n$, then $\chi_n(e_n(r,c)) = r$.
\end{claim}
The claim follows from Lemma~\ref{lem:embedding-from-projection}.

Now let $E = E_0 \cap E_1$, which is $\kappa$-closed and dense.  For $n<2$, let $F_n = e_n[E]$.  There is a $\kappa$-continuous isomorphism $\iota : F_1 \to F_0$ defined by $\iota = e_0 \circ e_1^{-1}$.  $\mathrm{pr}_0 \circ \iota$ is a $\kappa$-continuous projection from $F_1$ to $\bbP$.
If $q \in F_1$ and $q = e_1(r,c)$, then 
$$\pi_0 \circ \mathrm{pr_0} \circ \iota(q) = \chi_0 \circ e_0(r,c) = r = \chi_1 \circ e_1(r,c) = \chi_1(q) = \rho_{\alpha_0}(q).$$
Take $r_1 \in F_1$ such that $\iota(r_1) \leq \bar q$.
Let $\dot s_1$ be any $\bbR_{\alpha_1}$-name such that $r_1 \Vdash \dot s_1 = \iota^*(\sigma)$.  Then we have that $\la \bbR_{\alpha_1},r_1,\dot s_1 \ra \leq \la \bbP,\bar p,\tau\ra$, as witnessed by the projection $\mathrm{pr}_0 \circ \iota$ defined on the $\kappa$-closed dense set $F_1$.  
Furthermore, $\la r_1,\dot s_1 \ra \leq \la r_0,\dot s_0\ra$ in $\bbR * \dot\add(\lambda)$ since $r_1 \leq \rho_{\alpha_0}(r_1) = \pi_0 \circ \mathrm{pr_0} \circ \iota(r_1) \leq \pi_0(\bar p) \leq r_0$, and 
\[r_1 \Vdash \dot s_1 \trianglerighteq (\mathrm{pr}_0 \circ \iota)^*(\tau) \trianglerighteq (\pi_0\circ\mathrm{pr}_0 \circ \iota)^*(\dot s_0) = \rho_{\alpha_0}^*(\dot s_0) = \dot s_0. \qedhere \]
\end{proof}

\begin{remark}
\label{termremark}
Suppose $\la \bbR_\alpha : \alpha<\lambda \ra$ and $\la \rho_{\alpha} : \alpha<\lambda \ra$ witness that $\bbR$ is a reasonable $(\kappa,\lambda)$-collapse.  Let $\psi : D \to \Sh(\kappa,\lambda)$ be as in the above lemma.  For $\la r,\dot s \ra \in D$, $\psi(r,\dot s) = \la\bbR_\alpha,r,\dot s\ra$, where $\alpha$ is the length of $\la r,\dot s \ra$.

It follows from the construction in the above lemma that
whenever $\la r_0,\dot s_0 \ra \in D$ has length $\alpha_0$, $\la\bbP,p,\tau\ra \leq \la\bbR_{\alpha_0},r_0,\dot s_0\ra$, and $\pi_0 : \bbP \to \bbR_{\alpha_0}$ is a projection defined on a dense $\kappa$-closed set witnessing the ordering, then there is $\la r_1,\dot s_1 \ra \in D$ of length $\alpha_1 \geq \alpha_0$ such that:
    \begin{itemize}
        \item $r_1 \leq r_0$.
        \item $1_\bbR \Vdash \dot s_1 \leq \dot s_0$.
        \item $\la\bbR_{\alpha_1},r_1,\dot s_1\ra \leq \la\bbP,p,\tau\ra$.
        \item If $r_0 = 1_\bbR$ and $p = 1_\bbP$, then we can take $r_1 = 1_\bbR$.
        \item There is a projection $\pi_1 : \bbR_{\alpha_1} \to \bbP$ defined on a dense $\kappa$-closed set witnessing the ordering such that $\pi_0 \circ \pi_1 = \rho_{\alpha_0} \rest \bbR_{\alpha_1}$.
    \end{itemize}
\end{remark}

\begin{lemma}
\label{distributive}
    Let $G \subseteq \Sh(\kappa,\lambda)$ be generic, $\bbR$ be a reasonable $(\kappa,\lambda)$-collapse, and $(\bbR * \dot\add(\lambda))/G$ be a quotient forcing via the projection $\psi$ given by Lemma~\ref{mainproj}.  
    
    Then, $(\bbR * \dot\add(\lambda))/G$ is $\lambda$-distributive in $V[G]$.
\end{lemma}

\begin{proof}
    Let $\la \bbR_\alpha : \alpha<\lambda \ra \subseteq \bbR$ witness that $\bbR$ is reasonable.
    Let $H * K \subseteq \bbR * \dot\add(\lambda)$ be generic such that $\psi(H{*}K) = G$.  Let $k: \lambda \to 2$ be the function corresponding to $K$.  
    
    First, we claim that $k \in V[G]$.  Let $\alpha<\lambda$ and let $\la\bbP_0,p_0,\tau_0\ra,\la\bbP_1,p_1,\tau_1\ra \in G$ be such that $p_0 \Vdash_{\bbP_0} \tau(\alpha) = n_0$ and $p_1 \Vdash_{\bbP_1} \tau(\alpha) = n_1$. There is $\la\bbQ,q,\sigma\ra \in G$ below both $\la\bbP_0,p_0,\tau_0\ra$ and $\la\bbP_1,p_1,\tau_1\ra$, and by the definition of the ordering, $q \Vdash \sigma(\alpha) = n_0 = n_1$, so $n_0 = n_1$.  Thus $G$ determines a function $g : \lambda \to 2$ by $g(\alpha) = n$ iff for some $\la\bbP,p,\tau\ra \in G$, $p \Vdash_\bbP \tau(\alpha) = n$.  Now if $\la r,\dot s\ra \in H*K$ is a condition of length $\beta$ and $r \Vdash_{\bbR_\beta} \dot s(\alpha) = n$, then $\psi(r,\dot s) = \la\bbR_\beta,r,\dot s\ra \in G$.  Thus $g = k$.

    For any set of ordinals $x \in V[H{*}K]$ of size $\leq\kappa$, there is some $\alpha<\lambda$ such that $x \in V[H \cap \bbR_\alpha]$.  In $V[G]$, the set $H \cap \bbR_\alpha$ is coded by some $y \subseteq \kappa$, and by a density argument there will be some $\beta<\lambda$ such that $y = \{ \gamma<\kappa : k(\beta +\gamma) = 1\}$.  Thus $y = \{ \gamma<\kappa : g(\beta+\gamma) = 1 \} \in V[G]$, and so $x \in V[G]$.
\end{proof}

In the previous lemma, for $\langle \bbP, p, \tau\rangle\in G$, we could not necessarily identify a generic for $\mathbb{P}$ in the generic extension. Namely, it is possible that there are incompatible $p_0,p_1\in\bbP$ such that $\langle \bbP, p_0, \tau\rangle, \langle \bbP, p_1, \tau\rangle\in G$. The ordering of $\Sh(\kappa,\lambda)$ ensures that such incompatible conditions must still decide the same information on $\tau$. In the next definition we identify a dense subset of conditions in which such a phenomenon cannot occur.  
\begin{definition}[Coding conditions]
    We say $\la \bbP,p,\tau \ra \in \Sh(\kappa,\lambda)$ is a \emph{coding condition} when $\tau$ is forced to have domain $\delta<\lambda$, and whenever $G \subseteq \bbP$ is generic,
    \begin{itemize}
        \item $V[G] = V[\tau^G]$;
        \item in all generic extensions of $V[G]$, $G$ is the unique filter $F$ that is $\bbP$-generic over $V$ such that $\tau^F = \tau^G$.
    \end{itemize}
  We say that $\la \bbP,p,\tau \ra$ is a \emph{strong coding condition} when $\tau$ is forced to have domain $\delta<\lambda$, and there is a 
  dense $D \subseteq \bbP$, a set $X \subseteq \delta$, and a function $f : X \to D$ such that for all $\alpha \in X$,
    $$1_\bbP \Vdash f(\alpha) \in \dot G \leftrightarrow \tau(\alpha) = 1.$$
\end{definition}


\begin{lemma}
\label{strongcoding}
    Strong coding conditions are coding.
\end{lemma}

\begin{proof}
    Suppose $\la\bbP,p,\tau\ra \in \Sh(\kappa,\lambda)$ is strong coding, and let $\delta<\lambda$, $D \subseteq \bbP$, $X\subseteq\delta$, and $f : X \to D$ witness.  Let $G \subseteq \bbP$ be generic and let $W$ be any generic extension of $V[G]$.  $D \cap G$ can be read off from $\tau^G$, so $V[G] = V[\tau^G]$.  If $H \in W$ is a different filter that is $\bbP$-generic over $V$, then there is $p \in D \cap G$ such that $p \notin H$.  There is $\alpha\in X$ such that $\tau^G(\alpha) = 1$ and $\tau^H(\alpha) = 0$.
\end{proof}
Note that $\la\bbP,p,\tau\ra$ being a coding or strong coding condition depends only on $\bbP$ and $\tau$ and not on $p$. 
It is easy to see that strong coding conditions are dense in $\Sh(\kappa,\lambda)$.
We also note that the property of being a coding condition is relatively local; a standard Löwenheim-Skolem argument shows that $\la\bbP,p,\tau\ra$ is a coding condition iff this is true in $H_\theta$, where $\theta>2^{|\bbP|}$ is regular and $\la\bbP,p,\tau\ra\in H_\theta$.

Suppose $\la \la\bbP_\alpha,1_\alpha,\tau_\alpha\ra : \alpha < \delta\ra \subseteq \Sh(\kappa,\lambda)$ is a descending sequence of coding conditions.  
Suppose $\la\bbP_\delta,1_\delta,\tau_\delta\ra$ is a lower bound, $\tau_\delta$ is forced to the the union of the $\tau_\alpha$, and $\bbP_\delta$ is in some sense a ``limit'' of the $\bbP_\alpha$, in that there is a dense set $D \subseteq \bbP_\delta$ and there are projections $\pi_{\delta\alpha} : D \to \bbP_\alpha$, such that whenever $G \subseteq \bbP_\delta$ is generic, then $G \cap D = \{ p \in D : (\forall \alpha<\delta) \pi_{\delta\alpha}(p) \in \pi_{\delta\alpha}(G) \}$.  Then $\la\bbP_\delta,1_\delta,\tau_\delta\ra$ is also coding, since $G$ can be recovered from $\la \pi_{\delta\alpha}(G) : \alpha<\delta\ra$, which can be recovered from $\tau_\delta^G = \bigcup_{\alpha<\delta} \tau_\alpha^{\pi_{\delta\alpha}(G)}$, and if $G' \not= G$ is another $\bbP_\delta$-generic filter, then for some $\alpha<\delta$, $\pi_{\delta\alpha}(G) \not= \pi_{\delta\alpha}(G')$, and thus $\tau_\delta^G \not= \tau_\delta^{G'}$.

\begin{lemma}[Projection freezing]
\label{freeze}
    Let $\langle \mathbb{P}, p, \tau\rangle$ be a coding condition and let $\langle \mathbb{Q},q,\sigma\rangle \leq \langle \mathbb{P}, p, \tau\rangle$.   For any dense $D \subseteq \bbQ$, if $\pi_0,\pi_1$ are projections from $D$ to $\bbP$ such that $q \Vdash \pi_i^*(\tau) \trianglelefteq \sigma$ for $i=0,1$, then $\pi_0(r) = \pi_1(r)$ for all $r \leq q$ in $D$.
\end{lemma}

\begin{proof}
    Let $D,\pi_0,\pi_1$ be as hypothesized, and let $\delta<\lambda$ be the forced domain of $\tau$.
    Suppose $r \leq q$, $r \in D$, and $\pi_0(r) \not= \pi_1(r)$; let us assume $\pi_0(r) \nleq \pi_1(r)$.  Let $p' \leq \pi_0(r)$ be such that $p' \perp \pi_1(r)$, and let $r' \leq r$ be such that $\pi_0(r') \leq p'$.  Thus we have $\pi_0(r') \perp \pi_1(r')$ since $\pi_1(r') \leq \pi_1(r) \perp p'$. 

    Let $H \subseteq \bbQ$ be generic with $r' \in H$.  Let $G_0,G_1$ be the $\bbP$-generics induced by $\pi_0,\pi_1$ respectively; note $G_0 \not= G_1$.  Since $q \in H$, $\tau^{G_0} = \tau^{G_1} = \sigma^H \rest \delta$.  But this is a contradiction since there is only one $\bbP$-generic filter $F \in V[H]$ with $\tau^F = \sigma^H \rest \delta$.
\end{proof}

\begin{lemma}
\label{G_P}
    Suppose $\la\bbP,1_\bbP,\tau\ra\in \Sh(\kappa,\lambda)$ is a coding condition, and let $G \subseteq \Sh(\kappa,\lambda)$ be generic over $V$ with $\la\bbP,1_\bbP,\tau\ra\in G$.  Then 
    $$G_{\bbP}^\tau = \{ p \in \bbP : \la\bbP,p,\tau\ra \in G\}$$
    is a $\bbP$-generic filter over $V$.
\end{lemma}

\begin{proof}
    First, let us check that that $G_\bbP^\tau$ is a filter.  Suppose $\la\bbP,p_0,\tau\ra,\la\bbP,p_1,\tau\ra\in G$.  If $\la\bbQ,q,\sigma\ra \in \Sh(\kappa,\lambda)$ is below both $\la\bbP,p_0,\tau\ra$ and $\la\bbP,p_1,\tau\ra$, with projections $\pi_0,\pi_1 : \bbQ \to \bbP$ witnessing each order relation respectively, then we can find $q'\leq q$ such that $q' \in \dom \pi_0 \cap \dom \pi_1$.  
    Therefore there is a dense set of conditions $\la\bbQ,q,\sigma\ra\in\Sh(\kappa,\lambda)$ that are either incompatible with one of $\la\bbP,p_0,\tau\ra,\la\bbP,p_1,\tau\ra$, or below both with $q$ in the domain of each of a pair of projections witnessing the order relations.  Since $\la\bbP,p_0,\tau\ra,\la\bbP,p_1,\tau\ra \in G$, the latter option appears in $G$.  If $\la\bbQ,q,\sigma\ra \in G$ is below both, with witnessing projections $\pi_0,\pi_1 : \bbQ \to \bbP$, with $q \in \dom \pi_0 \cap \dom \pi_1$, then $\pi_0(q) = \pi_1(q)$ by Lemma~\ref{freeze}, since $\la\bbP,1,\tau\ra$ is coding.  Thus $\la\bbP,\pi_0(q),\tau\ra\in G$ and $\pi_0(q) \leq p_0,p_1$.

    To show that $G_\bbP^\tau$ is generic, let $D \subseteq \bbP$ be an open and dense set in $V$.  If $\la\bbQ,q,\sigma\ra \in \Sh(\kappa,\lambda)$ is below $\la\bbP,1,\tau\ra$, then if $\pi : \bbQ \to \bbP$ is a witnessing projection, there is $q' \leq q$ such that $\pi(q') \in D$.  Therefore there is a dense set of conditions $\la\bbQ,q,\sigma\ra$ that are either incompatible with $\la\bbP,1,\tau\ra$, or below it and such that $q \in \pi^{-1}[D]$ for a witnessing projection $\pi$.  If $\la\bbQ,q,\sigma\ra\in G$ is in this dense set, the the latter option must occur, so for some projection $\pi : \bbQ \to \bbP$, $\pi(q) \in D \cap G_\bbP^\tau$.
\end{proof}

\begin{remark}
    If $\la\bbP,1,\tau\ra$ is a coding condition, then there is a projection $\pi$ from a dense subset of $\Sh(\kappa,\lambda) \rest \la\bbP,1,\tau\ra$ to $\bbP$.  For any $\la\bbQ,q,\sigma\ra\leq\la\bbP,1,\tau\ra$, if there is  a projection $\rho$ witnessing the ordering with $q\in\dom\rho$, let $\pi(\la\bbQ,q,\sigma\ra)=\rho(q)$.  This is well-defined because by Lemma~\ref{freeze}, any two $\rho_0,\rho_1$ with these properties must have $\rho_0(q)=\rho_1(q)$.  It is easy to check that $\pi$ is a projection.
\end{remark}

\begin{remark} 
    $\Sh(\kappa,\lambda)$ is never countably closed.  
    
    To see this, let $\bbP$ be any nontrivial completely $\kappa$-closed separative poset in $V_\lambda$.  $\mathcal{B}(\bbP)$ contains a countable maximal antichain, $\la a_i : i < \omega \ra$.  If $b_i = \sum_{j\geq i} a_j$, then $\la b_i : i < \omega \ra$ is a descending sequence without a lower bound.  Viewing $\bbP$ as a dense subset of $\mathcal{B}(\bbP)$, let $\bbP_i = \bbP \cup \{ b_i \}$ for each $i<\omega$.  Note that each $\bbP_i$ is completely $\kappa$-closed, as the single extra point $b_i$ does not affect this.  
    
    Now, let $\tau$ be a name such that $\la \bbP,1,\tau \ra$ is a coding condition.  The sequence $\la\la\bbP_i,b_i,\tau \ra : i <\omega \ra$ is a descending sequence in $\Sh(\kappa,\lambda)$, with $\id \rest \bbP$ witnessing each order relation.  A lower bound to this sequence would force that for each $i$, $G^\tau_{\bbP_i}$ is a generic filter for $\bbP_i$ possessing $b_i$.  We must also have that $G^\tau_{\bbP_i} \rest \bbP = G^\tau_\bbP$ for all $i<\omega$.  Thus, if $H$ is the filter on $\mathcal{B}(\bbP)$ generated by $G^\tau_\bbP$, then $H$ contains the descending sequence $\la b_i : i < \omega \ra$.  But this contradicts the fact that $H$ must possess one of the $a_i$ by genericity.
 \end{remark}



As a corollary of the Freezing Lemma, we conclude that (at least locally) the projections that witness the order of the conditions in $\Sh(\kappa,\lambda)$ commute.
\begin{lemma}
\label{commute}
    Suppose $\la\bbP_0,p_0,\tau_0\ra, \la\bbP_1,p_1,\tau_1\ra,\la\bbP_2,p_2,\tau_2\ra$ are elements of $\Sh(\kappa,\lambda)$ such that:
    \begin{enumerate}
        \item $\la\bbP_2,p_2,\tau_2\ra \leq \la\bbP_1,p_1,\tau_1\ra\leq\la\bbP_0,p_0,\tau_0\ra$.
        \item $\la\bbP_0,p_0,\tau_0\ra$ is a coding condition.
        \item For $i < j$, $i,j\in\{0,1,2\}$, $\pi_{j,i}$ is a projection witnessing $\la\bbP_j,p_j,\tau_j\ra \leq \la\bbP_i,p_i,\tau_i\ra$.
    \end{enumerate}
    Then for densely many $q \leq p_2$, $\pi_{2,0}(q) = \pi_{1,0} \circ \pi_{2,1}(q)$. 
\end{lemma}

In Lemmas \ref{strongstrategic} and \ref{lambdastratclosed} we will prove strengthenings of the following lemmas:
\begin{lemma}
    $\Sh(\kappa,\lambda)$ is $\kappa$-strategically closed.
\end{lemma}

\begin{lemma}
    $\Sh(\kappa,\lambda)$ forces $\lambda = \kappa^{+}$.
\end{lemma}

\section{Uniformization}\label{sec: uniformization}
Before diving into the details, let us motivate our construction. In the standard constructions of saturated ideals, we lift an elementary embedding (typically, an almost-huge one) through a two-stage collapse $\bbP * \dot\bbR$.  We start with an embedding $j \colon V \to M$ such that $\crit(j) = \kappa$, $j(\kappa)=\lambda$, and $M^{<\lambda}\subseteq M$. $\bbP \subseteq V_\kappa$ turns $\kappa$ into a successor cardinal $\mu^+$, and $\bbR$ is a reasonable $(\kappa,\lambda)$-collapse in $V^\bbP$.

Typically, we will get that $j(\mathbb{P})$ projects to $\bbP*\dot\bbR$.  So if $G' \subseteq j(\bbP)$ is generic, then we obtain a generic $G * H \subseteq \bbP*\dot\bbR$, and we can lift the embedding first to $j \colon V[G] \to M[G']$.  As any initial segment of $\mathbb{R}$ is smaller than $\lambda$, the set $j[H {\restriction} \alpha]$ belongs to $M[G']$ for $\alpha<\lambda$, and using some closure property of $j(\mathbb R)$, we obtain a lower bound $m_\alpha \in j(\mathbb{R})\restriction j(\alpha)$ for this set. 



In order to extend the \emph{master sequence} $\la m_\alpha:  \alpha < \lambda\ra$ to an $M$-generic filter for $j(\mathbb{R})$, we must extend each $m_\alpha$ in a way that will be compatible with later extensions of $m_\beta$, $\alpha < \beta$. In other words, we are using the fact that the \emph{restriction map is a projection}. So, we can extend each $m_\alpha$ to some condition in $j(\mathbb{R}\restriction \alpha)$ meeting some antichain which is bounded by coordinate $j(\alpha)$, knowing that we will not obtain a disagreement with $m_\beta$ for larger $\beta$.

At limit points $\delta$ of this process, we must verify that $m_\delta$ is compatible with the \emph{sequence} of steps that we already made. This is not a consequence of the existence of projections from $j(\mathbb{R})\restriction\delta$ to the lower parts. Instead, we need to verify that this projection \emph{factors through the inverse limit} of the system of forcing notions and projections that we considered so far.\footnote{This is not automatic, even for very well-behaved forcing notions and projections. Indeed, let $\mathbb{Q}_n = \Add(\omega_1,n)$, namely all partial functions from $\omega_1\times n$ to $2$ with countable domain. Let $\pi_{n,m}\colon \mathbb{Q}_n \to \mathbb{Q}_m$ be the restriction map. Let $\mathbb{Q}_\omega$ be the collection of all functions $f \colon \alpha\times \omega \to 2$, with $\alpha<\omega_1$, such that for all $\beta<\alpha$, $\{n < \omega  :  f(\beta,n)=1\}$ is finite.  Let $\pi_{\omega,n} : \bbQ_\omega \to \bbQ_n$ be the restriction map $f \mapsto f \rest (\omega_1 \times n)$.  The system $\la\pi_{\beta,\alpha} : \alpha<\beta\leq \omega \ra$ is a commuting system of continuous projections, but $\bbQ_\omega$ is not the inverse limit of the system $\langle \mathbb Q_n, \pi_{n,m}  :  m < n < \omega\rangle$.  

It is not difficult to arrange that $\mathbb{Q}_\omega$ appears in a generic $G \subseteq \Sh(\omega_1,\lambda)$, as well as $\mathbb Q_n$ for all $n$,  with the above projections witnessing the ordering of their conditions, by picking the corresponding names $\tau$ in a suitable way. In this case, no other forcing that appears in $G$ can factor through the inverse limit of $\la \mathbb Q_n, \pi_{n,m}: n < m < \omega\ra$.}

In our forcing, we are using the Dual Shioya collapse $\Sh(\kappa,\lambda)$ after some first stage $\bbP \subseteq V_\kappa$.  Suppose that in some forcing extension, we have a lifted embedding $j \colon V[G] \to M[G']$, where $G,G'$ are generic for $\bbP,j(\bbP)$ respectively.  Suppose we also have a generic $H \subseteq \Sh(\kappa,\lambda)^{V[G]}$ such that $V_\alpha \cap H \in M[G']$ for all $\alpha<\lambda$.  In order to be able to lift the  embedding through $V[G][H]$, we need to add structure to $H$. We isolate a sequence of forcing notions and (almost) commuting projections between them. Moreover, we want the projections to respect small inverse limits. 

This can be done by adding the quotient relative to $\col(\kappa,{<}\lambda)*\add(\lambda)$, but this would be too aggressive. We want to maintain the universality of the forcing notion, in the sense of Lemma \ref{iteratedshioyastrat}. So, we add a more complicated, but weaker structure.

\begin{definition}
    Let $\bbU(G)$ be the set of triples $u = \la\vec\bbQ^u,\vec\pi^u, \vec q^u\ra$ such that:
\begin{enumerate}
\item For some $\delta<\lambda$, $\vec\bbQ^u = \la\la \bbQ_\alpha,1_\alpha,\tau_\alpha \ra : \alpha \leq \delta \ra$, $\vec\pi^u =  \la \pi_{\beta\alpha} : \alpha<\beta\leq\delta \ra$, and $\vec q^u = \la q_{\beta\alpha}  :  \alpha < \beta \leq \delta\rangle$.
\item For each $\alpha\leq\delta$, $\la \bbQ_\alpha,1_\alpha,\tau_\alpha \ra \in G$ and $\la\bbQ_\alpha,1_\alpha,\tau_\alpha \ra$ is a coding condition.
\item\label{codecommute} For each $\alpha<\beta\leq\delta$, $\la\bbQ_\beta,q_{\beta\alpha},\tau_\beta\ra \in G$, and $\pi_{\beta\alpha}$  is a $\kappa$-continuous projection from a $\kappa$-closed dense subset of $\bbQ_\beta$ to $\bbQ_\alpha$ that witnesses the order relation $\la\bbQ_\beta,q_{\beta\alpha},\tau_\beta\ra \leq \la\bbQ_\alpha,1_\alpha,\tau_\alpha\ra$. We denote $\pi_{\beta\beta}:=id_{\mathbb{Q}_\beta}$.
\item\label{invlimU} Suppose $x \in [\delta+1]^{<\kappa}$ has a top element $\gamma$.  If $q \in \bbQ_\gamma$ is such that $q \in E_\gamma^{\la\vec\bbQ \rest x,\vec\pi\rest x\ra}$
and $\pi_{\gamma\beta}(q) \leq q_{\beta\alpha}$ for all $\{\alpha,\beta\} \in [x]^2$, then below $q$, the natural map into the inverse limit along $x \cap \gamma$ is a projection.\footnote{Recall that $E_\gamma^{\la\vec\bbQ \rest x,\vec\pi\rest x\ra}$ is the set of all $p \in \mathbb Q_\gamma$ such that for all finite sequences $\alpha_0<\alpha_1<\dots<\alpha_n = \gamma$ contained in $x$, $\pi_{\alpha_1,\alpha_0}\circ\dots\circ\pi_{\alpha_n,\alpha_{n-1}}$ is defined at $p$. It is a $\kappa$-closed dense set.}   
More specifically, for all such $q$ and all
$$\vec r \in \varprojlim\left(\vec\bbQ \rest (x \cap \gamma),\vec\pi\rest(x\cap\gamma)\right)$$
pointwise below $\la \pi_{\gamma\alpha}(q) : \alpha \in x \cap \gamma \ra$, there is $q' \leq q$ such that
$$\la \pi_{\gamma\alpha}(q') : \alpha \in x \cap \gamma \ra \leq \vec r.$$

\end{enumerate}
We define $u' = \la \vec \bbQ^{u'}, \pi^{u'}, \vec q^{u'}\ra \leq u = \la \vec \bbQ^u, \pi^u, \vec q^u\rangle$  when $\vec\bbQ^{u'} \supseteq \vec\bbQ^u$, $\vec\pi^{u'}\supseteq\vec\pi^u$, and ${\vec q\ } ^{u'}\supseteq {\vec q\ }^u$.\footnote{We will omit the superscript $u$ when it is clear from the context, and add it to the components of the forcing for emphasis, when needed.}
\end{definition}

One way to think about the forcing $\bbU(G)$ is as a threading forcing, similar to the forcing adding a thread to an anti-compactness principle, as in \cite{CummingsForemanMagidor} and many subsequent works. The forcing $\Sh(\kappa,\lambda)$ adds a collection of collapsing generics connected via unspecified projections. In a larger generic extension (for example, in the generic extension by $\Col(\kappa,{<}\lambda) * \Add(\lambda)$), one can find those projections explicitly. Since the quotient map between $\Col(\kappa,{<}\lambda)\ast \Add(\lambda)$ and $\Sh(\kappa,\lambda)$ is $\lambda$-distributive, we expect the forcing $\bbU(G)$ to be $\lambda$-distributive as well.  

Let us remark on the connection between (\ref{codecommute}) and (\ref{invlimU}).  Suppose $\la\vec\bbQ,\vec\pi,\vec q\ra$ is a condition of length $\delta+1$.  Suppose $\alpha<\beta<\gamma\leq\delta$, and $q \in \bbQ_\gamma$ is such that the maps $\pi_{\gamma\alpha}$ and $\pi_{\beta\alpha}\circ\pi_{\gamma\beta}$ are defined at $q$, $q \leq q_{\gamma\beta},q_{\gamma\alpha}$, and $\pi_{\gamma\beta}(q) \leq q_{\beta\alpha}$.  
Then $\la \bbQ_\gamma,q,\tau_\gamma\ra \leq \la\bbQ_\alpha,1_\alpha,\tau_\alpha\ra$, and the maps $\pi_{\gamma\alpha}$ and $\pi_{\beta\alpha} \circ \pi_{\gamma\beta}$ both witness this relation.  Therefore, since $\la\bbQ_\alpha,1_\alpha,\tau_\alpha\ra$ is a coding condition, Lemma~\ref{freeze} implies that $\pi_{\gamma\alpha}(q) = \pi_{\beta\alpha} \circ \pi_{\gamma\beta}(q)$.

Thus if $x,\gamma,q$ satisfy the hypotheses of (\ref{invlimU}), then $\la\pi_{\gamma\alpha}(q) : \alpha\in x\cap\gamma \ra \in \varprojlim(\vec\bbQ \rest (x\cap\gamma),\vec\pi\rest(x\cap\gamma))$, and for all $\beta \in x$, the maps down to $\bbQ_\alpha$ for $\alpha \in x \cap\beta$ commute below $\pi_{\gamma\beta}(q)$.  So below $\la\pi_{\gamma\alpha}(q) : \alpha\in x\cap\gamma \ra$, $\la\vec\bbQ\rest(x\cap\gamma),\vec\pi\rest(x\cap\gamma)\ra$ is a $\kappa$-good inverse system.  Condition (\ref{invlimU}) requires that, below $q$, $\bbQ_\gamma$ projects to the inverse limit in the natural way.

We will use the flexibility in choosing the conditions $q_{\beta\alpha}$ to obtain a strategic closure. This will play a crucial role in the proof of Lemma \ref{iteratedshioyastrat}, in which we use the $\bbU$ forcing, which is defined in a smaller forcing extension, as a termspace forcing for $\bbU$ as defined in a larger extension. In this case, the commutativity of the projections and the existence of inverse limits might occur only below a sufficiently strong condition. 

Let $U\subseteq \mathbb{U}(G)$ be generic and let $\vec{\mathbb{Q}}_U = \bigcup_{u \in U} \vec{\mathbb{Q}^u}$.
We say that a condition $\langle \mathbb{Q}, 1, \sigma\rangle$ appears in $U$ if it belongs to the range of $\vec{\mathbb{Q}}_{U}$.
\begin{lemma}
\label{1ptext}
    If $G * U \subseteq \Sh(\kappa,\lambda) * \dot\bbU(\dot G)$ is generic, then for all $\la\bbP,p,\sigma\ra \in G$, there is a coding condition $\la\bbQ,1,\tau\ra$ appearing in $U$ such that for some $q \in G^\tau_\bbQ$, $\la\bbQ,q,\tau\ra\leq\la\bbP,p,\sigma\ra$.
\end{lemma}
\begin{proof}
    Suppose $G \subseteq \Sh(\kappa,\lambda)$ is generic.  Let $u=\la\vec\bbQ^u,\vec\pi^u,\vec q^u\ra\in\bbU(G)$ and $\la\bbP,p,\sigma\ra \in G$ be arbitrary. Suppose $\vec\bbQ^u$ has length $\delta+1$, and let $\la\bbQ_\delta,1_\delta,\tau_\delta\ra$ be the top element of $\vec\bbQ^u$.  Let $\la\bbQ_{\delta+1},q^*,\tau_{\delta+1}\ra \in G$ be a coding condition below both $\la\bbQ_\delta,1_\delta,\tau_\delta\ra$ and $\la\bbP,p,\sigma\ra$. Let $\rho$ be a projection witnessing $\la\bbQ_{\delta+1},q^*,\tau_{\delta+1}\ra\leq\la\bbQ_\delta,1_\delta,\tau_\delta\ra$.  
    Let $\pi_{\delta+1,\delta}^{u'} = \rho$, and for $\alpha<\delta$, let  $\pi_{\delta+1,\alpha}^{u'} = \pi_{\delta\alpha}^{u} \circ\rho$.

    Let $q^{u'}_{\delta+1,\delta} = q^*$, and for $\alpha<\delta$, let $q^{u'}_{\delta+1,\alpha}$ be some $r \in G_\bbQ^\tau$ below $q^*$ such that $\rho(r) \leq q^u_{\delta\alpha}$.
    Let $u'$ be the associated extensions of $u$, with the $\delta+1^{st}$ element of $\vec\bbQ^{u'}$ being $\la\bbQ_{\delta+1},1_{\bbQ_{\delta+1}},\tau_{\delta+1}\ra$.  It suffices to show that $u'\in\bbU(G)$.


    Let us verify condition~(\ref{invlimU}).  Suppose $x \in [\delta+1]^{<\kappa}$, and let $x' = x \cup\{\delta+1\}$.
    Suppose $q \in \bbQ_{\delta+1}$ and $\pi_{\delta+1,\beta}(q) \leq q^{u'}_{\beta\alpha}$ for all $\{\alpha,\beta\} \in [x']^2$.  By the definition of $\vec\pi^{u'}$ and $\vec q^{u'}$, $\rho(q)$ and $\pi_{\delta\alpha}(\rho(q))$ are defined for each $\alpha \in x \cap \delta$, and $\pi^u_{\delta\beta}(\rho(q)) \leq q^u_{\beta\alpha}$ for all $\{\alpha,\beta\} \in [x \cup \{\delta\}]^2$.
    
    Suppose $\vec r \in \varprojlim(\vec\bbQ^{u'
    }\rest x,\vec\pi^{u'}\rest x)$ and $\vec r \leq\la\pi_{\delta+1,\alpha}(q) : \alpha \in x \ra$.  If $\delta \in x$, put $q' = \vec r(\delta)$, and otherwise use the assumption that $u \in \bbU(G)$ to find $q' \leq \rho(q)$ in $\bbQ_\delta$ such that $\la\pi_{\delta\alpha}(q') : \alpha \in x \ra \leq \vec r \rest\delta$.  Then find $q'' \leq q$ in $E^{\la\vec\bbQ^{u'},\vec\pi^{u'}\ra}_{\delta+1}$ such that $\rho(q'') \leq q'$, 
    so that $\la\pi_{\delta+1,\alpha}(q'') : \alpha \in x \ra \leq \vec r$.
\end{proof}

\begin{lemma}
\label{strongstrategic}
    $\Sh(\kappa,\lambda) * \dot\bbU(\dot G)$ is strongly $\kappa$-strategically closed.
\end{lemma}

\begin{proof}
    Let us play the game $\mathcal{G}^\mathrm{I}_\kappa(\Sh(\kappa,\lambda) * \dot\bbU(\dot G))$.  We will use the quite reasonable $(\kappa,\lambda)$-collapse $\col(\kappa,{<}\lambda)$ to guide our strategy for Player II.  Later, we will refer to the strategy we describe as $\Sigma$ and invoke some of its specific properties.

    Let us denote the conditions in the play as $\langle p_i, \dot{u}_i\rangle$, where $p_i = \la \bbP_i, r_i, \tau_i\rangle$.  

    Suppose Player I plays $\la\la\bbP_0,r_0,\tau_0\ra,\dot u_0\ra$.  By extending the first coordinate if necessary, we may assume it decides the top $\Sh(\kappa,\lambda)$-condition in $\dot u_0$ to be some $\la\bbQ,1,\sigma\ra$, and then since this is compatible with $\la\bbP_0,r_0,\tau_0\ra$, we can find a common coding extension $\la\bbP_0',r_0',\tau_0'\ra$.  Then, using the proof of Lemma~\ref{1ptext}, we can extend $\dot u_0$ to $\dot u_0'$, such that $\la\bbP_0',r_0',\tau_0'\ra$ forces $\la\bbP_0',1,\tau_0'\ra$  to be at the top of $\dot u_0'$.  
    Let $\nu$ be a cardinal such that $|\bbP_0'|<\nu<\lambda$ and $\nu^{<\kappa} = \nu$, and let $\alpha_1 = \nu+1$.  By Lemma~\ref{mcaloon}, both $\bbP_0' \times \col(\kappa,\nu)$ and $\col(\kappa,{<}\alpha_1)$ have $\kappa$-closed dense sets isomorphic to $\col(\kappa,\nu)$, so let $\rho$ be a $\kappa$-continuous projection from a dense $\kappa$-closed subset of $\col(\kappa,{<}\alpha_1)$ to $\bbP_0'$.
    Let $r_1$ and $\tau_1$ be such that $\la\col(\kappa,{<}\alpha_1),r_1,\tau_1\ra \leq \la\bbP_0',r_0',\tau_0'\ra$, with the former also a coding condition, and with the ordering witnessed by $\rho$.
    Again appealing to Lemma~\ref{1ptext}, let $\dot u_1$ be a one-point extension of $\dot u_0'$ with $\la\col(\kappa,{<}\alpha_1),1,\tau_1\ra$ at the top. We set the projections of $\dot u_1$ to be the composition of the projection $\rho$ with the previous projections $\vec\pi^{\dot u_0'}$. Letting $\dot\delta+1$ be the length of the condition $\dot{u}_0'$, we extend its sequence $\vec q$ by taking $q^{\dot{u}_1}_{\dot\delta + 1,\alpha}$, for $\alpha<\dot\delta$, to be a name for a condition in $G^{\tau_1}_{\col(\kappa,{<}\alpha_1)}$ below $r_1$ projected by $\rho$ below $q^{\dot{u}_0'}_{\dot{\delta},\alpha}$, and we set $q^{\dot{u}_1}_{\dot \delta+1,\dot\delta} = r_1$.  
    
    Let $\mathbb{P}_1 = \col(\kappa,{<}\alpha_1)$, and let $\la\la\bbP_1,r_1,\tau_1\ra,\dot u_1\ra$ be Player II's first move.

    Suppose Player I responds with $\la\la\bbP_2,r_2,\tau_2\ra,\dot u_2\ra$.  We first strengthen it as above, so that we can assume $\la\bbP_2,r_2,\tau_2\ra$ forces $\la\bbP_2,1,\tau_2\ra$ to be at the top of $\dot u_2$, and also so that $\la\bbP_2,q_2,\tau_2\ra$ decides the projection that $\dot u_2$ assigns from $\bbP_2$ to $\col(\kappa,{<}\alpha_1)$, say as $\pi^*$.  Then appealing to Lemma~\ref{mainproj} and Remark~\ref{termremark}, we can find $\alpha_3$, $r_3$, and $\tau_3$ such that $r_3 \leq r_1$ in $\col(\kappa,{<}\lambda)$, $1 \Vdash \tau_3 \trianglerighteq \tau_1$,   
     $\la\col(\kappa,{<}\alpha_3),r_3,\tau_3\ra$ is a coding condition below $\la\bbP_2,r_2,\tau_2\ra$, say with witnessing projection $\rho$,
     and on a $\kappa$-closed dense set, $\pi^* \circ \rho$ equals the restriction map $p \mapsto p \rest \alpha_1$. Note that here we ask this commutativity to hold on a dense set and not just below a certain condition.
     
     Let $\dot u_3$ be a name for a one-point extension of $\dot u_2$ with $\la\col(\kappa,{<}\alpha_3),1,\tau_3\ra$ at the top, but with the projections chosen in the following way:  
    \begin{itemize}
        \item The projection from $\col(\kappa,{<}\alpha_3)$ to $\col(\kappa,{<}\alpha_1)$ is the restriction map, $r \mapsto r \rest\alpha_1$. 
        \item The projection from $\Col(\kappa,{<}\alpha_3)$ to posets $\bbQ^{\dot{u}_2}_\zeta$, appearing in $\dot u_2$ between $\col(\kappa,{<}\alpha_1)$ and $\col(\kappa,{<}\alpha_3)$ is, as in Lemma~\ref{1ptext}, the composition $\pi \circ \rho$, where $\pi$ is the projection that $\dot u_2$ chooses from $\bbP_2$ to $\bbQ^{\dot u_2}_\zeta$.
        \item For each $\zeta$ smaller than the index of $\col(\kappa,{<}\alpha_3)$, the projection from $\col(\kappa,{<}\alpha_3)$ to $\bbQ^{\dot u_2}_\zeta$ is the composition $\pi \circ (\cdot\rest\alpha_1)$ where $\pi$ is the projection that $\dot u_2$ chooses from $\Col(\kappa,{<}\alpha_1)$ to $\bbQ^{\dot u_2}_\zeta$.
    \end{itemize}
    The sequence of conditions $\vec q^{u_2}$ is extended as follows.  Working in an extension by $\Sh(\kappa,\lambda)$ obtained by forcing below $\la\col(\kappa,{<}\alpha_3),r_3,\tau_3\ra$, suppose the posets played by Player I occupy places $\xi_0 < \xi_1$ in $u_2$, so that those played by II will occupy places $\xi_0+1$ and $\xi_1+1$ in $u_3$. 
    The conditions $q^{u_3}_{\xi_1+1,\alpha}$ for $\alpha\leq\xi_1$ are chosen from $G^{\tau_3}_{\col(\kappa,{<}\alpha_3)}$ with the following properties:
    \begin{itemize}
        \item $q^{u_3}_{\xi_1+1,\xi_0+1} =1$.
        \item For $\alpha\leq\xi_0$, $q^{u_3}_{\xi_1+1,\alpha} \rest \alpha_1 \leq q^{u_2}_{\xi_0+1,\alpha}$.
        \item For $\xi_0+1 < \alpha \leq \xi_1$, $q^{u_3}_{\xi_1+1,\alpha}\leq r_3$ and $\rho(q^{u_3}_{\xi_1+1,\alpha}) \leq q^{u_2}_{\xi_1,\alpha}.$
    \end{itemize}

    To check clause (\ref{invlimU}), suppose $x \in [\xi_1+1]^{<\kappa}$, $q \in \col(\kappa,{<}\alpha_3)$ is such that $\pi_{\xi_1+1,\beta}(q) \leq q^{u_3}_{\beta\alpha}$ for $\{\alpha,\beta\} \in [x \cup \{\xi_1+1\} ]^2$.  Suppose $\vec r \leq \la \pi_{\xi_1+1,\alpha}(q) : \alpha \in x \ra$ in $\varprojlim(\vec \bbQ^{u_3} \rest x,\vec \pi^{u_3} \rest x)$.  If $\sup(x) > \xi_0$, then we find the desired $q \leq q'$ as in the proof of Lemma~\ref{1ptext}.  If $\sup(x) \leq \xi_0$, then there is $q' \leq q \rest \alpha_1$ with the desired property by induction, and we can take $q'' = q \cup q'$. 
    
    Let us assume inductively that $\delta<\kappa$ and the game has gone on for $\delta$ rounds, and Player II's moves satisfy the following for odd $i<\delta$:
    \begin{itemize}
        \item Each of II's moves is of the form $\la\la\col(\kappa,{<}\alpha_i),r_i,\tau_i\ra,\dot u_i\ra$, where \\
        $\la\col(\kappa,{<}\alpha_i),1,\tau_i\ra$ is forced to be at the top of $\dot u_i$.
        \item For odd $j<i$, $r_i \rest \alpha_j \supseteq r_j$, $1 \Vdash \tau_i \trianglerighteq \tau_j$,
        it is forced that the projection chosen by $\dot u_i$ from $\col(\kappa,{<}\alpha_i)$ to $\col(\kappa,{<}\alpha_j)$ is the restriction map, $r \mapsto r\rest\alpha_j$, and the condition $q_{\beta_i,\beta_j} \in \col(\kappa,{<}\alpha_i)$, where $\beta_j<\beta_i$ are the indices of these posets in $u_i$, is just the trivial condition.
        \item It is forced that for any $\bbQ$ appearing in $\dot u_i$, if $j$ is such that $\col(\kappa,{<}\alpha_j)$ appears after $\bbQ$, and it is the poset in the first coordinate of a previous play by II, and the indices in $\dot u_i$ of $\bbQ,\col(\kappa,{<}\alpha_j),\col(\kappa,{<}\alpha_i)$ are $\gamma,\beta_j,\beta_i$ respectively, then $\pi^{u_i}_{\beta_i,\gamma} = \pi^{u_i}_{\beta_j,\gamma} \circ (\cdot\rest\alpha_j)$, restricted to some $\kappa$-closed dense set.
        Furthermore, it is forced that $q_{\beta_i,\gamma} \rest \alpha_j \leq q_{\beta_j,\gamma}$. 
    \end{itemize}
    We must show two things: first, that a lower bound at stage $\delta$ exists when $\delta$ is a limit, and second, that Player II can continue to play moves that satisfy these assumptions.  If $\delta$ is not a limit, Player II responds just as in round 3.

    Suppose $\delta$ is a limit ordinal. 
    Let us first describe the $\Sh(\kappa,\lambda)$-condition. Let $\alpha_\delta = \sup \{ \alpha_i : i < \delta$ is odd$\}$.  Let $r_\delta = \bigcup \{ r_i : i < \delta$ is odd$\}$.  Let $\tau_\delta$ be a $\col(\kappa,{<}\alpha_\delta)$-name for the union of $\{ \tau_i : i<\delta$ is odd$\}$ (formally, pulled to be $\col(\kappa,{<}\alpha_\delta)$-names).  By the remarks just after Lemma~\ref{strongcoding}, $\la\col(\kappa,{<}\alpha_\delta),r_\delta,\tau_\delta\ra$ is a coding condition, and it is a lower bound to the first coordinates of the previously played conditions.

    To define $\dot u_\delta$, temporarily work in an extension $V[G]$ by $\Sh(\kappa,\lambda)$ obtained by forcing below $p_\delta=\la\col(\kappa,{<}\alpha_\delta),r_\delta,\tau_\delta\ra$.  Suppose $\la\la\bbQ_\beta,1,\sigma_\beta\ra: \beta<\xi\ra$, $\la\pi_{\beta\beta'},q_{\beta\beta'} : \beta'<\beta<\xi\ra$ 
    are the sequences given by amalgamating the $\dot u_i^G$ for $i<\delta$. 
    For an odd ordinal $i<\delta$, let $\beta_i$ be the index below $\xi$ of the poset $\mathbb{P}_i=\col(\kappa,{<}\alpha_i)$ played by II. 

    Let us define the condition $u_\delta$, extending $u_i$ for all $i < \delta$.
    Let $\la\bbQ^{u_\delta}_\xi,1,\sigma_\xi\ra = \la\col(\kappa,{<}\alpha_\delta),1,\tau_\delta\ra$.
    To define the projections, first if $\gamma = \beta_i$ for some odd $i<\delta$, let $\pi^{u_\delta}_{\xi\gamma} = \cdot\rest\alpha_i$, and let $q^{u_\delta}_{\xi\gamma} = 1$.  For other $\gamma<\xi$, define
    \begin{align*}
        D_\gamma = \{ r \in \col(\kappa,{<}\alpha_\delta) : & \text{ }\pi_{\gamma,\beta_j} \circ \pi_{\beta_{i},\gamma}(r \rest \alpha_i) \text{ is defined} \\ 
       &\text{ whenever } j<i<\delta \text{ are odd and } \beta_j<\gamma<\beta_i \}.
    \end{align*}
    It is easy to check that $D_\gamma$ is a $\kappa$-closed dense subset of $\col(\kappa,{<}\alpha_\delta)$.
    Define $\pi^{u_\delta}_{\xi\gamma} : D_\gamma \to \bbQ_\gamma$ as $\pi_{\beta_{i},\gamma}\circ (\cdot\rest\alpha_{i})$, where $i<\delta$ is such that $\gamma<\beta_i$.  (By the induction hypothesis, it does not matter which such $i$ we pick.)
    To define $\vec{q}^{u_\delta}$, let $q^{u_\delta}_{\xi\gamma}$ be a condition $q$ in the generic filter of $\col(\kappa,{<}\alpha_\delta)$, such that whenever $\beta_j<\gamma<\beta_i$, $q \rest \alpha_{i} \leq q_{\beta_{i},\gamma}$, and $\pi_{\xi\gamma}(q) \leq q_{\gamma,\beta_j}$.
    As those are ${<}\kappa$-many conditions below which to project, such a condition $q$ exists.  We have that for all $\gamma<\xi$, $\pi^{u_\delta}_{\xi\gamma}$ witnesses the ordering
    $\la \col(\kappa,{<}\alpha_\delta),q_{\xi\gamma}^{u_\delta},\tau_\delta\ra \leq \la\bbQ_\gamma,1,\sigma_\gamma\ra$.
    Working in $V[G]$, let us show that $u_\delta$ satisfies condition (\ref{invlimU}). Let $x \in [\xi]^{<\kappa}$, and let $x' = x \cup \{\xi\}$.  Let $y = \{\beta_i  :  i < \delta\}$.  Note that if $r \in \bbQ_\xi^{u_\delta}$ and $\pi^{u_\delta}_{\xi\gamma}(r)$ is defined for all $\gamma \in x$, then also $\pi^{u_\delta}_{\gamma,\beta_j}\circ\pi^{u_\delta}_{\beta_i,\gamma}\circ\pi^{u_\delta}_{\xi\beta_i}$ is defined at $r$ whenever $\gamma \in x$ and  $\beta_j<\gamma<\beta_i$.
    Furthermore, if $q \in \bbQ_\xi^{u_\delta}$ has the property that all compositions of the maps $\vec\pi^{u_\delta}$ indexed by elements of $x'$ are defined at $q$, and $\pi^{u_\delta}_{\xi\beta}(q) \leq q^{u_\delta}_{\beta\alpha}$ for all $\{\alpha,\beta\} \in [x']^2$, then by construction, this also holds for all $\{\alpha,\beta\} \in [x'\cup y]^2$.  
    It follows that for all $\{\alpha,\beta\} \in [x'\cup y]^2$, 
    $\la\bbQ_\beta,\pi_{\xi\beta}(q),\sigma_\beta\ra \leq \la\bbQ_\alpha,1,\sigma_\alpha\ra$,
    and the map $\pi_{\beta\alpha}$ witnesses this.  By the remarks after the definition of $\bbU(G)$, the system of maps
    $\la\pi_{\beta\alpha} : \{\alpha,\beta\} \in [x' \cup y]^2\ra$ commutes below $\la \pi_{\xi\gamma}(q) : \gamma \in x' \cup y \ra$.

    Suppose $q \in E^{\la\vec\bbQ^{u_\delta} \rest x', \vec\pi^{u_\delta} \rest x' \ra}$, $\pi^{u_\delta}_{\xi\beta}(q) \leq q^{u_\delta}_{\beta\alpha}$ for all $\{\alpha,\beta\} \in [x']^2$, and 
    $\vec r \in \varprojlim(\vec\bbQ^{u_\delta} \rest x, \vec\pi^{u_\delta} \rest x)$ is below $\la \pi^{u_\delta}_{\xi\gamma}(q) : \gamma \in x \ra$.  If $x$ is not cofinal in $\xi$, let $i<\delta$ be such that $\beta_i > \sup x$.  Then $\vec r$ is below $\la\pi_{\beta_i,\gamma}(q \rest \alpha_i) : \gamma \in x\ra$.  Since $u_i$ is a condition and the projections commute below $q$, there is $q' \leq q \rest \alpha_i$ such that $\la \pi_{\beta_i,\gamma}(q') : \gamma \in x\ra \leq \vec r$.  Then let $q'' = q \cup q'$.

    If $x$ is cofinal in $\xi$, then by the above observations, \[q \in E^{\la\vec\bbQ^{u_\delta} \rest (x'\cup y), \vec\pi^{u_\delta} \rest (x' \cup y) \ra}.\]
    By Lemma~\ref{invlim_cof1}, below $\la \pi_{\xi\gamma}(q) : \gamma \in x \cup y \ra$,
    $\varprojlim(\vec\bbQ^{u_\delta} \rest (x \cup y), \vec\pi^{u_\delta} \rest (x \cup y))$
    is canonically isomorphic to a dense subset of 
    $\varprojlim(\vec\bbQ^{u_\delta} \rest x, \vec\pi^{u_\delta} \rest x)$ via the restriction map.  Thus we can find $\vec s \in \varprojlim(\vec\bbQ^{u_\delta} \rest (x \cup y), \vec\pi^{u_\delta} \rest (x \cup y))$ such that $\vec s \rest x \leq \vec r$, and take $q' \leq q$ as $\bigcup_{i<\delta} \vec s(\beta_i)$, witnessing (\ref{invlimU}).
    So, at the $\delta^{th}$ step, Player I has a possible move. Let us show that Player II can continue similarly after a limit-stage move of Player I, $\la p_\delta,\dot u_\delta\ra$, where $p_\delta=\la\bbP_\delta,r_\delta,\tau_\delta\ra$.  
    The main idea is that Player II can play almost as if Player I had played the lower bound constructed above.
    
    Working temporarily in an extension $V[G]$ by $\Sh(\kappa,\lambda)$ obtained by forcing below $p_\delta$, assume $\dot u_\delta^G$ takes the form $\la\la\la\bbQ_i,1,\sigma_i\ra : i\leq\xi \ra,\la\pi_{ji} : i<j\leq\xi\ra, \la q_{ji} : i < j\leq\xi \ra\ra$.
    Let $\la\la\col(\kappa,{<}\alpha_i),r_i,\tau_i\ra,\dot u_i\ra$ be Player II's moves for odd $i<\delta$, and let $\beta_i$ be the index in $u_\delta$ where $\col(\kappa,{<}\alpha_i)$ appears.  Let $\alpha_\delta = \sup_i \alpha_i$, $\beta_\delta = \sup_i \beta_i$, and $y = \{ \beta_i : i<\delta$ is odd$\}$.
    By strengthening 
$p_\delta$ if necessary, we may assume that it decides the values of the ordinals $\xi,\alpha_i,\beta_i$, the projections $\pi_{ji}$ for $\{i,j\} \in [y\cup\{\xi\}]^2$, and that $|\bbP_\delta| = \nu > \alpha_\delta$, and $\bbP_\delta$ collapses $\nu$ to $\kappa$.  As in Lemma~\ref{1ptext}, we may also assume
    that $\la\bbP_\delta,1,\tau_\delta\ra$ is forced to be at the top of $\dot u_\delta$ at position $\xi$.  
    
    Since $u_\delta$ satisfies condition (\ref{invlimU}), there will be $q^*$ in the generic for $\bbQ_{\xi} = \bbP_\delta$ such that the maps in the system $\la\pi_{ji} : \{i,j\} \in [y \cup \{\xi\}]^2\ra$ commute below $q^*$, and
the natural map from $\bbQ_{\xi}$ to $\varprojlim(\vec\bbQ\rest y,\vec\pi\rest y)$ is a projection.  Note that this inverse limit is canonically isomorphic to $\col(\kappa,{<}\alpha_\delta)$.  Using Theorem~\ref{mcaloon}, and strengthening $q^*$ if necessary, we can construct a projection $\varphi \colon \bbQ_\xi \to \col(\kappa,{<}\alpha_\delta)$, defined on a $\kappa$-closed dense set, such that below $q^*$, $\varphi$ agrees with the natural map to $\varprojlim(\vec\bbQ\rest y,\vec\pi\rest y)$.\footnote{This is done simply by selecting a dense tree as in Theorem~\ref{mcaloon} with $q^*$ in the first level, noting the tree below any given node is isomorphic to the whole tree, noting that $\col(\kappa,<\nu+1)$ also contains a dense copy of that tree, and gluing partial projections.}  Without loss of generality, $q^* = r_\delta$.
    
    
    
    If $\tau_\delta'$ is a name for the union of $\tau_i$ for odd $i<\delta$, then we have that $\varphi$ witnesses the ordering $\la\bbP_\delta,r_\delta,\tau_\delta\ra\leq \la\col(\kappa,{<}\alpha_\delta),1,\tau_\delta'\ra$.
    By Lemma~\ref{mainproj} and Remark~\ref{termremark}, we can find a coding condition of the form $\la\col(\kappa,{<}\alpha_{\delta+1}),r_{\delta+1},\tau_{\delta+1}\ra \leq \la\bbP_\delta,r_\delta,\tau_\delta\ra$ with witnessing projection $\rho$ such that $\varphi\circ\rho = \cdot\rest\alpha_\delta$ on some $\kappa$-closed dense set, and $\tau_{\delta+1}$ is forced (by 1) to end-extend $\tau_\delta'$.
    
    We extend $\dot u_\delta$ to $\dot u_{\delta+1}$ that is forced to have $\bbQ^{u_{\delta+1}}_{\xi+1}=\col(\kappa,{<}\alpha_{\delta+1})$ and $\sigma_{\xi+1}=\tau_{\delta+1}$. 
    For the projections, $\pi^{u_{\delta+1}}_{\xi+1,\gamma}$ is defined as $\pi^{u_{\delta}}_{\xi\gamma}\circ\rho$ for $\beta_\delta \leq\gamma\leq\xi$, $\pi^{u_{\delta+1}}_{\xi+1,\gamma}$ is the restriction map when $\gamma \in y$, and for $\gamma \in \beta_\delta\setminus y$, 
    $\pi^{u_{\delta+1}}_{\xi+1,\gamma} = \pi^{u_\delta}_{\beta_i,\gamma} \circ (\cdot\rest\beta_i)$, where $\gamma<\beta_i<\beta_\delta$, restricted to the set of points $p$ such that $\pi^{u_{\delta}}_{\gamma,\beta_j}\circ\pi^{u_{\delta}}_{\beta_k,\gamma}(p \rest \beta_k)$ is defined whenever $\beta_j<\gamma<\beta_k$.
    For the conditions $q_{\xi+1,\gamma}$, we take $q_{\xi+1,\gamma} = 1$ for $\gamma \in y$,
    $\rho(q) \leq q_{\xi,\gamma}^{u_\delta}$ for $\beta_\delta\leq\gamma\leq\xi$,
    and for $\gamma \in\beta_\delta\setminus y$, we take $q_{\xi+1,\gamma}$ to be some $q$ in the generic below $r_{\delta+1}$ 
    such that whenever $\beta_j<\gamma<\beta_i$, $q \rest \alpha_{i} \leq q_{\beta_{i},\gamma}$, and $\pi_{\xi\gamma}(q) \leq q_{\gamma,\beta_j}$.


    Let us check that condition (\ref{invlimU}) holds for $\dot u_{\delta+1}$.  Suppose $x \in [\xi+1]^{<\kappa}$ and let $x' = x \cup \{\xi+1\}$; we must show $p \mapsto \la \pi_{\xi+1,\beta}(p) : \beta \in x\cap(\xi+1)\ra$ is a projection from $E^{\la\vec\bbQ\rest x',\vec\pi\rest x'\ra}_{\xi+1}$ to $\varprojlim(\vec\bbQ \rest x,\vec\pi\rest x)$ below the suitable condition.  If $\sup x > \beta_\delta$, or $\sup(x) < \beta_i$ for some $i<\delta$, then we argue as in Lemma~\ref{1ptext}.  So suppose $\sup x = \beta_\delta$, $q \in E^{\la\vec\bbQ\rest x',\vec\pi\rest x'\ra}_{\xi+1}$, and $\pi_{\xi+1,\beta}(q) \leq q_{\beta\alpha}$ for all $\{\alpha,\beta\} \in [x']^2$.  By construction, $q \in E^{\la\vec\bbQ\rest (x'\cup y),\vec\pi\rest (x'\cup y)\ra}_{\xi+1}$, and the system of maps
    $\la\pi_{\beta\alpha} : \{\alpha,\beta\} \in [x' \cup y]^2\ra$ commutes below $\la \pi_{\xi+1,\gamma}(q) : \gamma \in x' \cup y \}$.  If $\vec r \leq \la \pi_{\xi+1,\gamma}(q) : \gamma \in x \ra$ is in $\varprojlim(\vec\bbQ \rest x,\vec\pi\rest x)$, then as before, we apply Lemma~\ref{invlim_cof1} to find $q' \leq q \rest \alpha_\delta$ that projects below $\vec r$, and then take $q'' = q' \cup q$.
    \end{proof}

In the previous lemma, we used a reasonable collapse in order to guide our strategy for $\Sh(\kappa,\lambda)\ast \dot{\mathbb{U}(G)}$. As the first component of this iteration is only $\kappa$-closed, we could not hope for a much stronger closure than what we got. In the next lemma, we consider a case in which the generic for $\Sh(\kappa,\lambda)$ was obtained using a well-behaving forcing. In this case, we have better control over the properties of $\mathbb{U}(G)$.

\begin{lemma}
\label{lambdastratclosed}
    If $\bbR$ is a quite reasonable $(\kappa,\lambda)$-collapse, $H*K$ is $\bbR * \dot\add(\lambda)$-generic, and $G = \psi(H*K)$, where $\psi$ is the projection given by Lemma~\ref{mainproj}, then in $V[H*K]$, $\bbU(G)$ is $\lambda$-strategically closed.  
\end{lemma}

\begin{proof}
    Let $\la\bbR_\alpha : \alpha<\lambda\ra$ and $\la\rho_\alpha: \alpha<\lambda\ra$ witness that $\bbR$ is quite reasonable.  Suppose Player I plays $u_0 = \la\la\la\bbQ_i,1,\tau_i\ra : i\leq \xi_0\ra,\la\pi_{ji} : i<j\leq\xi_0\ra,\la q_{ji} : i<j\leq\xi_0\ra\ra$.  There are $\alpha_1 < \lambda$, $r_1 \in \bbR_{\alpha_1}$, and $\sigma_1$ such that $\la r_1,\sigma_1\ra \in H*K$ and $\psi(r_1,\sigma_1)= \la\bbR_{\alpha_1},r_1,\sigma_1\ra$ is a coding condition below $\la\bbQ_{\xi_0},1,\tau_{\xi_0}\ra$.  Using the argument for Lemma~\ref{1ptext}, $u_0$ can be extended to $u_1$ of length $\xi_0+2$ with $\la\bbR_{\alpha_1},1,\sigma_1\ra$ at position $\xi_0+1$.  This is Player II's first move.

    Suppose Player I responds with $u_2 = \la\la\la\bbQ^{u_2}_i,1,\tau^{u_2}_i\ra : i\leq \xi_2\ra,\la\pi^{u_2}_{ji} : i<j\leq\xi_2\ra, \la q^{u_2}_{ji} : i < j \leq \xi_2\ra\ra$.  There are $\alpha_3 < \lambda$, $r_3 \in \bbR_{\alpha_3}$, and $\sigma_3$ such that 
    $\la r_3,\sigma_3\ra \in H*K$, $r_3 \Vdash \sigma_1 \trianglelefteq \sigma_3$, and $\psi(r_3,\sigma_3)= \la\bbR_{\alpha_3},r_3,\sigma_3\ra$ is a coding condition below $\la\bbQ^{u_2}_{\xi_2},1,\tau^{u_2}_{\xi_2}\ra$.
    Let $\chi : \bbR_{\alpha_3} \to \bbQ_{\xi_2}$ witness the ordering.
    
    
    Let Player II's next move be $u_3 = \la\la\la\bbQ^{u_3}_i,1,\tau^{u_3}_i\ra : i\leq \xi_2+1\ra,\la\pi^{u_3}_{ji} : i<j\leq\xi_2+1\ra, \la q^{u_3}_{ji} : i < j \leq \xi_2 + 1\ra\ra$, where:
    \begin{itemize}
        \item $\la\bbQ^{u_3}_{\xi_2+1},1,\tau^{u_3}_{\xi_2+1}\ra = \la\bbR_{\alpha_3},1,\sigma_3\ra$.
        \item $\pi^{u_3}_{\xi_2+1,\beta} =  \pi^{u_2}_{\xi_2\beta}\circ\chi$ for $\xi_0+1<\beta\leq\xi_2$, $\pi^{u_3}_{\xi_2+1,\xi_0+1} = \rho_{\alpha_1} \rest \bbR_{\alpha_3}$, and $\pi^{u_3}_{\xi_2+1,\beta} = \pi^{u_2}_{\xi_0+1,\beta} \circ \rho_{\alpha_1}$ for $0\leq\beta\leq\xi_0$.
        \item For $\xi_0+1<\beta\leq\xi_1$, $q^{u_3}_{\xi_2+1,\beta}$ is a condition $q$ in the generic below $r_3$ such that $\chi(q) \leq q^{u_2}_{\xi_2,\beta}$, and for $0\leq\beta\leq\xi_0+1$, $q^{u_3}_{\xi_2+1,\beta}$ is a condition $q$ in the generic below $r_3$ such that $\rho_{\alpha_1}(q) \leq q^{u_2}_{\xi_0+1,\beta}$.
    \end{itemize}  

    It is easy to check that condition (\ref{codecommute}) holds for membership in $\bbU(G)$.  The argument that this condition (\ref{invlimU}) is satisfied is the same as for Player II's second move in Lemma~\ref{strongstrategic}.

    Assume inductively that $\delta<\lambda$, and the game has gone on for $\delta$-many rounds, with the following properties:
    \begin{itemize}
        \item The moves take the form $u_\beta = \la\la\la\bbQ^{u_\beta}_i,1,\tau^{u_\beta}_i\ra: i \leq \xi_\beta\ra,\la\pi^{u_\beta}_{ji}: i<j\leq\xi_\beta\ra, \la q^{u_\beta}_{ji}: i < j \leq \xi_\beta\ra\ra$ for $\beta<\delta$, with Player II having played in odd finite rounds and even infinite rounds.
        \item There is an increasing sequence of ordinals $\la\alpha_\beta : \beta<\delta\ra$ such that for odd finite and even infinite $\beta$, $\bbQ^{u_\beta}_{\xi_\beta} = \bbR_{\alpha_\beta}$.
        \item For $\beta$ odd finite or even infinite, $\la 1,\tau_{\xi_\beta}\ra \in H*K$, and $\tau_{\xi_\beta}$ is an $\bbR_{\alpha_\beta}$-name for a binary sequence of length $\alpha_\beta$ (so that $\psi(1,\tau_{\xi_\beta})=\la \bbR_{\alpha_\beta},1,\tau_{\xi_\beta}\ra$).
        \item For $\beta<\gamma$ both rounds where II played, $\pi^{u_\beta}_{\xi_\gamma\xi_\beta} = \rho_{\alpha_\beta} \rest \bbR_{\alpha_\gamma}$.
    \end{itemize}
    If $\delta$ is not a limit and the last round was a move of Player I, then II responds just as in round 3. 

    Let $\delta$ be a limit ordinal. The cofinality of $\delta$ in $V[G]$ is at most $\kappa$.  Let $\nu = \cf(\delta)$, and let $\la\gamma_i: i< \nu\ra$ be an increasing cofinal sequence in $\delta$ consisting of odd finite or even infinite ordinals.  Let $\alpha_\delta' = \sup_{i<\nu} \alpha_{\gamma_i}$.  There is $\alpha_\delta \geq \alpha_\delta'$ such that for some $\bbR_{\alpha_\delta}$-name $\tau$ for an $\alpha_\delta$-length binary sequence, $\la 1,\tau\ra \in H*K$, and $\psi(1,\tau)$ is a coding condition.  
    Put $\xi_{\delta} = \sup_{\beta<\delta} \xi_\beta$.
    Let us define the condition $u_\delta$ extending the amalgamation of $u_i$ for $i < \delta$. So $\vec\bbQ^{u_\delta}\restriction \xi_\delta$ is determined, and similarly the restrictions $\vec\pi^{u_\delta} \restriction \xi_{\delta}$ and $\vec q^{u_\delta} \restriction \xi_{\delta}$.  
    \begin{itemize}
        \item $\bbQ^{u_\delta}_{\xi_\delta} = \bbR_{\alpha_\delta}$ and $\tau_{\xi_\delta}^{u_\delta} = \tau$.
        \item For $\beta<\delta$ odd finite or even infinite, put $\pi^{u_\delta}_{\xi_\delta,\xi_\beta} = \rho_{\alpha_\beta} \rest \bbR_{\alpha_\delta}$, and let $q^{u_\delta}_{\xi_\delta,\xi_\beta}$ be a condition in $H\cap \bbR_{\alpha_\delta}$ forcing $\tau \trianglerighteq \tau^{u_\delta}_{\xi_\beta}$.
        \item For $\beta<\xi_\delta$ not of the form $\xi_\beta$ for $\beta$ a round where II played, let $i<\nu$ be least such that $\xi_{\gamma_i} > \beta$, and let $\pi^{u_\delta}_{\xi_\delta,\beta} = \pi^{u_{\gamma_i}}_{\xi_{\gamma_i},\beta} \circ\pi^{u_\delta}_{\xi_\delta,\xi_{\gamma_i}}$, restricted to the set of points $p$ such that $\pi^{u_\delta}_{\beta,\xi_{\gamma_k}}\circ \pi^{u_\delta}_{\xi_{\gamma_i},\beta} \circ\pi^{u_\delta}_{\xi_\delta,\xi_{\gamma_i}}(p)$ is defined for all $k<i$.  Since $\la\xi_{\gamma_k} : k \leq i \ra\in V$, $\pi^{u_\delta}_{\xi_\delta,j}$ is an object in $V$.
        \item Suppose $\beta < \xi_\delta$ is not an index of a Player II round. Let $i$ be the least index of a Player II round such that $\gamma_i > \beta$. 
        Let $q^{u_\delta}_{\xi_\delta,\beta}$ be a condition in the generic such that its projection under $\pi^{u_\delta}_{\xi_\delta,\eta}$ is stronger than $q^{u_{\gamma_i}}_{\eta,\zeta}$ for all $\{\zeta,\eta\} \subseteq \{\beta\} \cup \{\xi_{\gamma_j}:  j \leq i\} \cup \{\xi_\delta\}$. 
    \end{itemize}

    Let us verify that under this construction, $u_\delta$ is a condition in $\bbU(G)$. 
    To show (\ref{invlimU}), suppose $x \in [\xi_\delta]^{<\kappa}$, and let $x' = x \cup \{\xi_\delta\}$.
    Assume $q \in E^{\la\vec\bbQ\rest x',\vec\pi\rest x'\ra}_{\xi_\delta}$,
    $\pi_{\xi_\delta,\beta}(q) \leq q_{\beta\alpha}$ for all $\{\alpha,\beta\} \in [x']^2$.
    Let $y = \{ i<\nu : (\exists \beta \in x) i$ is least such that $\xi_{\gamma_i} \geq \beta \}$.
    By construction, we also have $q \in E^{\la\vec\bbQ\rest (x'\cup y),\vec\pi\rest (x'\cup y)\ra}_{\xi_\delta}$, and
    $\pi_{\xi_\delta,\beta}(q) \leq q_{\beta\alpha}$ for all $\{\alpha,\beta\} \in [x' \cup y]^2$.
    
    Suppose $\vec r \leq \la \pi_{\xi_\delta,\beta}(q) : \beta \in x\ra$ in $\varprojlim(\vec\bbQ\rest x,\vec\pi\rest x)$.
    If $y$ has a maximum element $j$, then because $u_{\gamma_{j}}$ is a condition, there is $q' \leq \pi_{\xi_\delta,\xi_{\gamma_j}}(q)$ such that $\la\pi_{\xi_{\gamma_j},\beta}(q') : \beta \in x \ra \leq \vec r$.  Then we take $q'' \leq q$ in $E^{\la\vec\bbQ\rest x',\vec\pi\rest x'\ra}_{\xi_\delta}$ such that $\pi_{\xi_\delta,\xi_{\gamma_j}}(q'')\leq q'$.

    If $y$ does not have a maximum element, let $\eta = \sup \{ \xi_{\gamma_i} : i \in y\}$.  By Lemma~\ref{invlim_cof1}, there is $\vec s \in \varprojlim(\vec\bbQ\rest(x\cup y),\vec\pi\rest(x\cup y))$ such that $\vec s \rest x \leq \vec r$.  Since $\bbR$ is a \emph{quite} reasonable collapse, $\bbR_\zeta$ is canonically isomorphic to $\varprojlim(\vec\bbQ \rest y,\vec\pi\rest y)$, where $\zeta = \sup \{\alpha_{\gamma_i} : i \in y\}$.  Thus we can find $q' \in \bbR_\zeta$ such that $q' \leq \rho_\zeta(q)$ and $\la \rho_{\alpha_{\gamma_i}}(q') : i \in y \} = \vec s \rest y$, and then find $q'' \leq q$ in $\bbR_{\alpha_\delta}$ such that $\rho_\zeta(q'') = q'$.  Thus $\la\pi_{\xi_\delta,\beta}(q'') : \beta \in x \ra \leq \vec r$.
    \end{proof}
    
    

\begin{corollary}
\label{Udist}
    $\Sh(\kappa,\lambda)$ forces that $\dot\bbU(\dot G)$ is $\lambda$-distributive.
\end{corollary}

\begin{proof}
    Suppose $G \subseteq \Sh(\kappa,\lambda)$ is generic, and towards a contradiction, suppose $\alpha<\lambda$ and $u \in \bbU(G)$ forces $\dot x$ to be a new $\alpha$-sequence of ordinals over $V[G]$.  Force a generic $H*K \subseteq \col(\kappa,{<}\lambda) * \dot\add(\lambda)$ such that $\psi(H{*}K) = G$.  By Lemma~\ref{distributive}, $V[H{*}K]$ has the same ${<}\lambda$-sequences of ordinals as $V[G]$.  In $V[H{*}K]$, $\bbU(G)$ is $\lambda$-distributive, so if $U \subseteq \bbU(G)$ is a generic filter over $V[H{*}K]$ with $u \in U$, then $V[H{*}K][U]$ has the same ${<}\lambda$-sequences of ordinals as $V[G]$.  Thus $\dot x^U \in V[G]$, but since $U$ is also generic over $V[G]$ and $u \Vdash \dot x \notin V[G]$, this is a contradiction.
\end{proof}
\begin{lemma}[Master conditions]
\label{3step}
    $\Sh(\kappa,\lambda) * \dot\bbU(\dot G) * \dot\col(\kappa,\lambda)$ is equivalent to $\col(\kappa,\lambda)$.  Moreover, whenever $\lambda'>\lambda$ is inaccessible, there are $\col(\kappa,\lambda)$-names $\tau^*$ and $\dot u^*$ such that $\la\col(\kappa,\lambda),1,\tau^*\ra$ is a strong coding condition in $\Sh(\kappa,\lambda')$ and $\dot u^*$ is forced to be an element of $\dot\bbU(\dot G')$, where $\dot G'$ is the canonical name for the $\Sh(\kappa,\lambda')$-generic, with the following property:  
    
    Whenever $G' * U' \subseteq \Sh(\kappa,\lambda') * \dot\bbU(\dot G')$ is a generic with $\la\la\col(\kappa,\lambda),1,\tau^*\ra,\dot u^* \ra \in G' * U'$, and $G*U*K$ is the $\Sh(\kappa,\lambda) * \dot\bbU(\dot G) * \dot\col(\kappa,\lambda)$-generic induced by $(G')^{\tau^*}_{\col(\kappa,\lambda)}$, then $G * U \subseteq G' * U'$.
    
\end{lemma}

\begin{proof}

    We combine the ideas of Lemmas~\ref{mcaloon} and \ref{strongstrategic}.  The forcing $\Sh(\kappa,\lambda) * \dot\bbU(\dot G) * \dot\col(\kappa,\lambda)$  collapses $\lambda$ to $\kappa$, and it has a dense subset of size $\lambda$.\footnote{Because of the lack of $\lambda$-c.c., the iteration does not have size $\lambda$, but the claim holds since each iterand is forced to have size $\lambda$.}  Let $\dot f$ be a name for a surjection from $\kappa$ to the generic filter intersected with this dense set.

    We use the strategy $\Sigma$ for Player II from Lemma~\ref{strongstrategic} to build a dense tree $T$ in the Boolean completion of the three-step iteration that is isomorphic to $\col(\kappa,\lambda)$.  For the first level of the tree, note that there is a dense set of conditions $\la s,u,c \ra \in \Sh(\kappa,\lambda) * \dot\bbU(\dot G) * \dot\col(\kappa,\lambda)$ such that $\la s,u \ra = \Sigma(\la\la s',u'\ra\ra)$ for some $\la s',u'\ra$, and $\la s,u,c \ra\leq q$ for some $q$ such that $\la s,u,c\ra \Vdash \dot f(0) = \check q$.  Take the first level $T_1$ to be a maximal antichain of size $\lambda$ of such conditions.

    Assume inductively that we have constructed the levels $\la T_\beta : \beta<\alpha\ra$ with the following properties:
    \begin{itemize}
        \item For $\gamma<\beta<\alpha$, $T_\beta$ is a maximal antichain refining $T_{\gamma}$.
        \item If $\beta<\alpha$ is a successor, then the elements of $T_\beta$ are members of $\Sh(\kappa,\lambda) * \dot\bbU(\dot G) * \dot\col(\kappa,\lambda)$, and if $\beta<\alpha$ is a limit, then each $p \in T_\beta$ is of the form $\inf \{ q \in \bigcup_{\gamma<\beta} T_\gamma : q \geq p\}$.
        \item For each successor $\beta<\alpha$ and each $p \in T_\beta$, we have assigned a descending sequence $\la \la s_i^p,u_i^p,c_i^p \ra : i < 2\beta \ra \subseteq \Sh(\kappa,\lambda) * \dot\bbU(\dot G) * \dot\col(\kappa,\lambda)$ with minimum element $p = \la s_{2 \beta  - 1}^p, u_{2 \beta - 1}^p, c_{2 \beta - 1}^p\ra$, such that $\la\la s_i^p,u_i^p \ra : i < 2\beta\ra$ is a partial play of $\mathcal{G}_\kappa^{\mathrm{I}}(\Sh(\kappa,\lambda) * \dot\bbU(\dot G))$ following $\Sigma$, with the property that $\{\la s_i^p,u_i^p,c_i^p\ra : i < 2\beta$ is odd$\}$ is exactly the set of nodes $q \in T_\gamma$ for successor $\gamma<\beta$ that are above $p$.  Furthermore, we require that for successor $\gamma<\beta<\alpha$, $q \in T_\gamma$, and $p \in T_\beta$ such that $q \geq p$, $\la \la s_i^p,u_i^p,c_i^p \ra : i < 2\gamma \ra$ is an initial segment of $\la \la s_i^p,u_i^p,c_i^p \ra : i < 2\beta \ra$.
        \item If $\beta<\alpha$ and $\beta = \gamma+1$, then for each $p \in T_\beta$, there is some $q$ such that $p \leq q$ and $p \Vdash \dot f(\gamma) = \check q$.
    \end{itemize}
    If $\alpha$ is a limit, let $T_\alpha$ be the set of infima to the branches through $\bigcup_{\beta<\alpha} T_\beta$.  Since the first two coordinates along each such branch conform to a play of the game according to $\Sigma$, and the third coordinate is in a $\kappa$-closed forcing, these infima are nonzero.  $T_\alpha$ is a maximal antichain just as in Lemma~\ref{mcaloon}.

    If $\alpha = \beta+1$, then for each $p \in T_\beta$, we can find $\la s,u,c\ra = p' \leq p$ such that for some $q$, $p'\leq q$ and $p' \Vdash \dot f(\beta) = \check q$.  We can put $\la s,u \ra$ as Player I's next move in the game.  Thus there is a dense set of $r \leq p$ whose first two coordinates come as Player II's next move in the game according to $\Sigma$ with initial sequence of plays $\la\la s_i^p,u_i^p,c_i^p\ra: i < 2\beta \ra$.  Let $T^p_\alpha$ be a maximal antichain of such conditions, and for each $r \in T^p_\alpha$, choose $\la s^r_{2\beta},u^r_{2\beta},c^r_{2\beta}\ra$ to give moves for Player I at round $2\beta$, continuing the sequence of plays up to $r$.  Let $T_\alpha = \bigcup \{ T^p_\alpha : p \in T_\beta \}$.  This completes the induction.

    $T$ is dense because for any $p \in \Sh(\kappa,\lambda) * \dot\bbU(\dot G) * \dot\col(\kappa,\lambda)$, there are conditions $q$ and $r$ and some $\alpha<\kappa$ such that $r \leq p$ and $q \Vdash \dot f(\alpha) = \check r$.  There is some $t \in T_{\alpha+1}$ compatible with $q$, and by construction we must have $t \leq r \leq p$.

    Let us now define a system of projections from dense open subsets of $T$ to $\col(\kappa,{<}\alpha)$ for $\alpha<\lambda$.  Recall that all of Player II's moves according to $\Sigma$ take the form $\la\la\col(\kappa,{<}\alpha),p,\tau\ra,\dot u\ra$.  For $\alpha<\lambda$, let $D_\alpha$ be the set of nodes $t \in T$ such that either $t$ is a successor node with the first coordinate of the form $\la\col(\kappa,{<}\beta),p,\tau\ra$ for some $\beta\geq\alpha$, or $t$ is a limit node, and $\alpha \leq \sup z$, where $z$ is the set of ordinals $\beta$ such that there is a successor node above $t$ with first coordinate of the form $\la\col(\kappa,{<}\beta),p,\tau\ra$.
    If $t \in D_\alpha$ is a successor node of the form $\la\la\col(\kappa,{<}\beta),p,\tau\ra,u,c\ra$, let $\sigma_\alpha(t) = p \rest \alpha$. 
    To show this is order-preserving, suppose $t_0>t_1$ are successor nodes of $T$, with the first coordinate of $t_i$ being $\la\col(\kappa,{<}\beta_i),q_i,\tau_i\ra$.  Since these nodes conform to a play of the game by $\Sigma$, the restriction map $x \mapsto x \rest \beta_0$ witnesses the ordering $\la\col(\kappa,{<}\beta_1),q_1,\tau_1\ra \leq \la\col(\kappa,{<}\beta_0),q_0,\tau_0\ra$.  Thus $q_1 \supseteq q_0$, and so $q_1 \rest \alpha \supseteq q_0 \rest \alpha$.
    
    If $t \in D_\alpha$ is a limit node, let $\delta$ be the ordertype of the set of nodes above $t$, and let $\la\la\la\col(\kappa,{<}\beta_i),q_i,\tau_i\ra,u_i,c_i\ra : i < \delta \ra$ list the successor nodes above $t$.  Put $\sigma_\alpha(t) = \bigcup_i q_i \rest \alpha$.
    By definition, $\sigma_\alpha$ is $\kappa$-continuous, and order preservation also holds for pairs that may include limit nodes.  If $\beta<\alpha$ and $t \in D_\alpha$, then $t \in D_\beta$ also, and $\sigma_{\beta}(t) = \sigma_{\alpha}(t) \rest \beta$.

    
    To show $\sigma_\alpha$ is a projection, suppose $\alpha<\lambda$, $t \in D_\alpha$, and $p \leq \sigma_{\alpha}(t)$.  If $t$ is a successor node, then $t$ has the form $\la\la\col(\kappa,{<}\beta),q,\tau\ra,u,c\ra$.  We get a stronger condition by replacing $q$ with $p\cup q$, and since $T$ is dense, there is a successor node in $T$ below it of the form $t' = \la\la\col(\kappa,{<}\gamma),r,\tau'\ra,u',c'\ra$, for some $\gamma\geq\beta$.  Let $\pi$ witness the ordering $\la\col(\kappa,{<}\gamma),r,\tau'\ra \leq \col(\kappa,{<}\beta),p\cup q,\tau\ra$.  Since the branch up to $t'$ conforms to $\Sigma$,
    the restriction map $x \mapsto x \rest \beta$ witnesses the ordering $\la\col(\kappa,{<}\gamma),r,\tau'\ra\leq \la\col(\kappa,{<}\beta),q,\tau\ra$.  Thus both $(\cdot\rest\beta)$ and $\id\circ\pi = \pi$ witness this ordering, so by Lemma~\ref{freeze}, $\pi(r) = r \rest \beta$.
    Thus $r \supseteq p \cup q$, and so $\sigma_{\alpha}(t') \leq p$.  
    
    If $t$ is a limit node, let $\la\la\la\col(\kappa,{<}\beta_i),q_i,\tau_i\ra,u_i,c_i\ra : i < \delta \ra$ list the successor nodes above $t$. 
    Let $\gamma = \sup_i\beta_i$.  As in the limit case of the proof of Lemma~\ref{strongstrategic}, one lower bound to the branch above $t$ takes the form $\la\la\col(\kappa,{<}\gamma),q,\tau\ra,u,c\ra$, where $q = \bigcup_i q_i$.  Thus $\la\la\col(\kappa,{<}\gamma),p \cup q,\tau\ra,u,c\ra$ is another lower bound, which is a condition stronger than $t$.  Since $T$ is dense, there is a successor node in $T$ below it of the form $t' = \la\la\col(\kappa,{<}\eta),r,\tau'\ra,u',c'\ra$.
    Let $\pi$ be a projection witnessing $\la\col(\kappa,{<}\eta,r,\tau'\ra \leq \la\col(\kappa,{<}\gamma),p\cup q,\tau\ra$.
    Since the branch up to $t'$ conforms to $\Sigma$, the restriction map witnesses the ordering $\la\col(\kappa,{<}\eta),r,\tau'\ra \leq \la\col(\kappa,{<}\beta_i),q_i,\tau_i\ra$ for $i <\delta$.
    For such $i$, the projections $(\cdot\rest\beta_{i}) \circ \pi$ and $(\cdot\rest\beta_{i})$ both witness $\la\col(\kappa,{<}\eta),r,\tau'\ra \leq \la\col(\kappa,{<}\beta_{i}),q_{i},\tau_{i}\ra$, so by Lemma~\ref{freeze}, they are the same below $r$.  Thus $\pi(r) \rest \gamma = r \rest \gamma$, so $r \supseteq p \cup q$, and $\sigma_{\alpha}(t')\leq p$.

    Let $\tau_0$ be a $\Sh(\kappa,\lambda)$-name for the generic function $g: \lambda \to 2$ as in the proof of Lemma~\ref{distributive}.  Then let $\tau^*$ be a $T$-name for a binary function of length $\lambda+\lambda$, where $\tau^* \rest \lambda = \tau_0$, and $\tau^* \rest [\lambda,\lambda+\lambda)$ codes the generic for $T$ (as in the definition of strong coding).  It is clear that for any inaccessible $\lambda'>\lambda$, $\la T,1,\tau^*\ra$ is a coding condition in $\Sh(\kappa,\lambda')$.  We claim that for any $t = \la\la\col(\kappa,{<}\alpha),p,\tau\ra,u,c\ra \in T$, the projection $\sigma_\alpha$ witnesses that $\la T,t,\tau^*\ra \leq \la\col(\kappa,{<}\alpha),p,\tau\ra$. 
    Suppose $B \subseteq T$ is generic with $t \in B$, $H = \sigma_\alpha(B)$, and $G*U*K \subseteq \Sh(\kappa,\lambda) * \dot\bbU(G') * \dot\col(\kappa,\lambda)$ is the filter corresponding to $B$.
    Since $\sigma_\alpha(t) = p$, $p \in H$.
    Further, for all $q \in H$, there is $s = \la\la\col(\kappa,{<}\beta),p',\tau'\ra,u',c'\ra \in B$ below $t$ such that $\sigma_\alpha(s) = p' \rest \alpha \leq q$.  Then $\la\col(\kappa,{<}\beta),p',\tau'\ra \in G$, and $p' \rest \alpha \in G^\tau_{\col(\kappa,<\alpha)}$.  Thus $H \subseteq G^\tau_{\col(\kappa,<\alpha)}$, and by the maximality of generic filters, $H = G^\tau_{\col(\kappa,<\alpha)}$.  If $\delta$ is the forced domain of $\tau$, then $\tau^H = \tau_0^G \rest \delta = (\tau^*)^B \rest \delta$.
    
    Suppose $G' \subseteq \Sh(\kappa,\lambda')$ is generic with $\la T,1,\tau^*\ra \in G'$, $B$ is the induced generic branch for $T$, and $G * U * K\subseteq \Sh(\kappa,\lambda) * \dot\bbU(G') * \dot\col(\kappa,\lambda)$ is the corresponding filter.  For any $\la\bbP,p,\tau\ra \in G$, there is a successor node $t = \la\la\col(\kappa,{<}\alpha),p',\tau'\ra,u,c\ra \in B$ such that $\la\col(\kappa,{<}\alpha),p',\tau'\ra \leq \la\bbP,p,\tau\ra$.  Since $\la T,t,\tau^*\ra \leq \la\col(\kappa,{<}\alpha),p',\tau'\ra$, we have $\la\bbP,p,\tau\ra \in G'$ as well.  Hence, $G \subseteq G'$.
    
    Let $\dot u^* = \la\vec\bbQ^{u^*},\vec\pi^{u^*}, \vec q^{u^*}\ra$ be a $T$-name for 
    \[\la\la\la\bbQ_i,1,\tau_i\ra : i \leq \lambda\ra,\la\pi_{ji} : i<j\leq\lambda\ra, \la q_{ji} : i < j \leq\lambda\ra,\] defined as follows.  Let  $\vec\bbQ\rest\lambda$,  $\vec\pi\rest\lambda$ and $\vec q \rest \lambda$ be the unions of the corresponding coordinates of conditions appearing in the $\dot\bbU(\dot G)$-generic.  Let $\bbQ_\lambda = T$ and $\tau_\lambda = \tau^*$.  Let $\la\alpha_i: i < \kappa\ra$ enumerate the indices in $\vec\bbQ$ at which the posets in the first coordinates of the successor nodes of $T$ appear in the generic branch.  For 
    $i<\kappa$, if $\bbQ_{\alpha_i} = \col(\kappa,{<}\beta)$, let $\pi_{\lambda\alpha_i} = \sigma_\beta$. 
    For $\gamma<\lambda$ such that $\gamma \notin \{ \alpha_i : i<\kappa \}$, let $i<\kappa$ be least such that $\gamma<\alpha_i$, and let $\pi_{\lambda\gamma} = \pi_{\alpha_i\gamma} \circ \pi_{\lambda\alpha_i}$, restricted to the set of points $t$ such that $\pi_{\gamma\alpha_j}\circ\pi_{\alpha_i\gamma} \circ \pi_{\lambda\alpha_i}(t)$ is defined for all $j<i$.

    To define the conditions $q_{\lambda\gamma}$ for $\gamma<\lambda$, first for $i<\kappa$, let $q_{\lambda\alpha_i}$ be the $i^{th}$ successor node $t \in B$, so $t$ has the form $\la\la\col(\kappa,{<}\beta),p,\tau\ra,u,c\ra$.  We established above that $\sigma_\beta$ witnesses $\la T,t,\tau^* \ra \leq \la\col(\kappa,{<}\beta),p,\tau\ra$.
    For $\gamma<\lambda$ such that $\gamma \notin \{\alpha_i : i <\kappa\}$, let $i$ be least such that $\alpha_i>\gamma$.  Let $q_{\lambda\gamma}$ be some $t \in B$ such that $t \leq q_{\lambda\alpha_i}$, and $\pi_{\lambda\beta}(t) \leq q_{\beta\alpha}$ for all $\{\alpha,\beta\} \subseteq \{ \gamma \} \cup \{ \alpha_j : j \leq i \}$.

    
    As $\Sh(\kappa,\lambda') \rest \la T,1,\tau^*\ra$ induces a generic for $T$, $\dot u^*$ can be interpreted as a $\Sh(\kappa,\lambda')$-name.  We must check that $\dot u^*$ is forced to be a member of $\dot\bbU(\dot G)$.  The satisfaction of the first three clauses of the definition is clear.  
    For (\ref{invlimU}), suppose $B$ and $G*U*K$ are generics as above, and let $u^* = (\dot u^*)^B$.  Suppose $x \in [\lambda]^{<\kappa}$, and let $x' = x \cup \{\lambda\}$.  Let $\la \alpha_i : i<\kappa\ra$ be as above, and let $y = \{ \alpha_i : (\exists \beta\in x) \alpha_i$ is least such that $\alpha_i\geq\beta\}$.  Suppose $t \in E_\lambda^{\la\vec\bbQ\rest x',\vec\pi\rest x' \ra}$ is such that $\pi_{\lambda\beta}(t) \leq q_{\beta\alpha}$ for all $\{\alpha,\beta\} \in [x']^2$.  By construction, we also have that $\pi_{\lambda\beta}(t) \leq q_{\beta\alpha}$ for all $\{\alpha,\beta\} \in [x' \cup y]^2$.
    Assume $\vec r \in \varprojlim(\vec\bbQ \rest x,\vec\pi\rest x)$ and $\vec r \leq \la \pi_{\lambda i}(t) : i \in x \ra$. 
    
    Suppose first that $y$ has a maximum element $\alpha_j$.  Then there is $u \in U$ such that $u^* \rest (\alpha_j+1) = u$, and we have that, putting $q = \pi_{\lambda\alpha_j}(t)$, $\pi_{\alpha_j\gamma}(q) \leq q_{\gamma\beta}$ for all $\{\beta,\gamma \} \in [x\cup y]^2$.  Thus there is $q' \leq q$ such that 
    $\pi_{\alpha_j\gamma}(q') \leq \vec r$, and there is $t' \leq t$ in $E_\lambda^{\la\vec\bbQ\rest x',\vec\pi\rest x' \ra}$ such that $\pi_{\lambda\alpha_j}(t') \leq q'$.

    Suppose next that $y$ does not have a maximal element.  By Lemma~\ref{invlim_cof1}, there is $\vec s \in \varprojlim(\vec\bbQ \rest (x \cup y),\vec \pi\rest (x \cup y))$ such that $\vec s \rest x \leq \vec r$.
    Let $z$ be the set of ordinals $\beta$ such that $\col(\kappa,{<}\beta)$ appears as $\bbQ_{\alpha_i}$ for $\alpha_i \in y$, and let $\gamma = \sup z$.
    Then $\vec s \rest y$ can be represented as $p \in \col(\kappa,{<}\gamma)$ via the canonical isomorphism.
    We have that $t \in \dom\sigma_\gamma$, and $p \leq \sigma_\gamma(t)$.  By the fact that $\sigma_\gamma$ is a projection, there is $t' \leq t$ in $E_\lambda^{\la\vec\bbQ\rest x',\vec\pi\rest x' \ra}$ such that $\sigma_\gamma(t') \leq p$, which implies $\la \pi_{\lambda i}(t') : i \in x \ra \leq \vec r$.

    Finally, if we force with $\Sh(\kappa,\lambda')* \dot\bbU(\dot G')$ below $\la\la T,1,\tau^*\ra,\dot u^*\ra$, obtaining a generic $G'*U'$, then if $G*U$ is the generic for $\Sh(\kappa,\lambda) * \dot\bbU(\dot G)$ induced from the $T$-generic, then all elements of $U$ are initial segments of $(\dot u^*)^{G'} \rest \lambda$, so $U \subseteq U'$.
\end{proof}


\section{Dense ideals on $\omega_1$ and $\omega_2$}\label{sec: dense-ideals-1-2}

Let us first show how the posets thus far introduced can force the existence of dense ideals on successors of regular uncountable cardinals.  Recall that a cardinal $\kappa$ is \emph{almost-huge} if there is an elementary embedding $j: V \to M$ with critical point $\kappa$ such that $M^{<j(\kappa)}\subseteq M$.  We say, ``$\kappa$ is almost-huge with target $\lambda$,'' when there is an embedding $j$ witnessing the almost-hugeness of $\kappa$ with $j(\kappa) = \lambda$.  If $j: V \to M$ is such an embedding, then by \cite[Theorem 24.11]{Kanamori}, there is another embedding $i: V \to N$ witnessing the almost-hugeness of $\kappa$, with $i(\kappa) =\lambda$, given via the direct limit of supercompactness measures derived from $j$.  Such an embedding $i$ has the additional property that $\lambda<i(\lambda)<\lambda^+$ and $i[\lambda]$ is cofinal in $i(\lambda)$.

\begin{theorem}
\label{anysuccreg}
    Suppose $\kappa$ is almost-huge with target $\lambda$ and $\mu<\kappa$ is regular and uncountable.  Then $\Sh(\mu,\kappa) * \dot\bbU(\dot G) * \dot\col(\kappa,{<}\lambda)$ forces that there is a normal ideal $I$ on $\kappa$ such that $\p(\kappa)/I \cong \mathcal{B}(\col(\mu,\kappa))$.
\end{theorem}

\begin{proof}
    Let $j: V \to M$ be an embedding witnessing that $\kappa$ is almost-huge with target $\lambda$.  Let us assume that $j(\lambda)<\lambda^+$ and $j[\lambda]$ is cofinal in $j(\lambda)$.

    By Lemma~\ref{3step}, $\Sh(\mu,\kappa) * \dot\bbU(\dot G) * \dot\col(\mu,\kappa)$ is equivalent to $\col(\mu,\kappa)$, so $\Sh(\mu,\kappa)$ can be seen as a subforcing of $\col(\mu,\kappa)$.  The  iteration $\col(\mu,\kappa) * \dot\col(\kappa,{<}\lambda)^{\Sh(\mu,\kappa)}$ is clearly a quite reasonable $(\mu,\lambda)$-collapse.  Let us choose an increasing sequence of regular suborders witnessing this, $\la\bbR_\alpha: \alpha<\lambda\ra$, with the property that $\bbR_0 = \{ 1\}$ and $\bbR_{\alpha} = \col(\mu,\kappa)$ for $0<\alpha\leq\kappa+\kappa$, for the map $\psi$ from Lemma~\ref{mainproj} to satisfy that $\psi(1,\tau^*) = \la\col(\mu,\kappa),1,\tau^*\ra$, where $\tau^*$ is the name for the $(\kappa+\kappa)$-length binary sequence given by Lemma~\ref{3step}.  The later $\bbR_\alpha$ are of the form $\col(\mu,\kappa) * \dot\col(\kappa,{<}\beta_\alpha)^{\Sh(\mu,\kappa)}$ for some appropriate increasing sequence of ordinals $\beta_\alpha$, and the witnessing system of projection maps is given by the canonical restriction maps.

    Let $G * U * H \subseteq  \Sh(\mu,\kappa) * \dot\bbU(\dot G) * \dot\col(\kappa,{<}\lambda)$ be generic.  Let $g \subseteq \col(\mu,\kappa)$ be generic over $V[G*U*H]$, which implies that $H$ is generic over $V[G*U*g]$. By Lemma \ref{3step}, $G * U * g$ is equivalent to a generic $K\subseteq \col(\mu,\kappa)$ so $G * U * g * H$ induces a $V$-generic $K * H \subseteq \col(\mu,\kappa) * \dot\col(\kappa,{<}\lambda)^{\Sh(\mu,\kappa)}$.
    
    Now consider the model $M[K{*}H]$.  Since in $M$, $j(\lambda)$ is an inaccessible cardinal  below $(\lambda^+)^V$, $\p(\lambda)^{M[K*H]}$ has size $\lambda$ in $V[K{*}H]$.  Also, since $K*H$ is generic for a forcing that is $\lambda$-c.c.\ in $M$, $M[K{*}H]$ is ${<}\lambda$-closed in $V[K{*}H]$.  Working in $V[K{*}H]$, inductively build a filter $h \subseteq \add(\lambda)^{M[K*H]}$ that is generic over $M[K{*}H]$ with $(\tau^*)^K \in h$.
    
    Let $\psi : [\col(\mu,\kappa) * \dot\col(\kappa,{<}\lambda)^{\Sh(\mu,\kappa)}] * \dot\add(\lambda) \to \Sh(\mu,\lambda)$ be the projection of Lemma~\ref{mainproj}, and let $G' = \psi(K{*}H{*}h)$.  We have $\la\col(\mu,\kappa),1,\tau^*\ra \in G'$, $K = (G')^{\tau^*}_{\col(\mu,\kappa)}$, and by Lemma~\ref{3step}, $G \subseteq G'$.  Therefore, in $V[K{*}H]$, we can lift the embedding to $j : V[G] \to M[G']$.  By Lemma~\ref{distributive}, $M[G']$ is ${<}\lambda$-closed in $M[K{*}H]$, so it is ${<}\lambda$-closed in $V[K{*}H]$.

    Let $u^*$ be the condition in $\bbU(G')$ given by Lemma~\ref{3step}.  Using the closure of $M[G']$, the fact that $\p(\bbU(G'))^{M[G']}$ has size $\lambda$ in $V[K{*}H]$, and the $\lambda$-strategic closure given by Lemma~\ref{lambdastratclosed}, build a filter $U' \subseteq \bbU(G')$ that is generic over $M[G']$ with $u^* \in U'$.  By Lemma~\ref{3step}, $U \subseteq U'$.  Thus the embedding can be further lifted to $j: V[G{*}U] \to M[G'{*}U']$.

    For each $\alpha<\lambda$, let $H_\alpha = H \cap \col(\kappa,{<}\alpha)^{V[G*U]}$.  Each such $H_\alpha$ is a member of $M[G']$.  Let $m_\alpha = \bigcup j[H_\alpha] \in M[G']$, which is a condition in $\col(\lambda,{<}j(\lambda))^{M[G']}$ since $|H_\alpha| < \lambda$ and the Levy collapse is $\lambda$-directed closed.  In $V[K{*}H]$, let $\la A_\alpha : \alpha < \lambda \ra$ be an enumeration of the maximal antichains contained in $\col(\lambda,{<}j(\lambda))^{M[G']}$ that are members of $M[G'{*}U']$, which is possible again because $j(\lambda)<(\lambda^+)^V$.  
    
    Recursively, build a descending sequence $\la q_\alpha : \alpha < \lambda \ra$ such that $q_\alpha \leq m_\alpha$ and $q_\alpha \in D_\alpha$.  In order to do this, also recursively choose an increasing sequence of ordinals $\la \gamma_\alpha: \alpha < \lambda \ra \subseteq \lambda$ such that $\gamma_\alpha \geq \alpha$ and $A_\alpha \subseteq \col(\kappa,{<}j(\gamma_\alpha))^{M[G']}$.  We inductively assume that $q_\alpha$ is compatible with $j(p)$ for all $p \in H$ and that $q_\alpha \in \col(\lambda,{<}j(\gamma_\alpha))$. Given $q_\alpha$, first take $q_{\alpha+1}' = q_\alpha \cup m_{\gamma_{\alpha+1}}$, and then find $q_{\alpha+1} \leq q_{\alpha+1}'$ in $\col(\lambda,{<}j(\gamma_{\alpha+1}))^{M[G']}$ such that $q_{\alpha+1} \leq a$ for some $a \in A_{\alpha}$.  $q_{\alpha+1}$ is compatible with $j(p)$ for all $p \in H$ because $q_{\alpha+1} \supseteq j(p)$ for all $p \in H_{\gamma_{\alpha+1}}$.
    At a limit stage $\alpha$, given $\la q_i : i <\alpha\ra$, let $q_\alpha = \bigcup_{i<\alpha} q_i$, which is compatible with $j(p)$ for all $p \in H$ because each $q_i$ is.

    The sequence $\la q_\alpha : \alpha<\lambda\ra$ generates a filter $H' \subseteq \col(\lambda,j(\lambda))^{M[G']}$ that is generic over $M[G']$ and has the property that $j[H] \subseteq H'$.  
    
    Thus, we may lift the embedding one more time to $j : V[G{*}U{*}H] \to M[G'{*}U'{*}H']$.  To perform this lifting, we only needed to adjoin the generic $g \subseteq \col(\mu,\kappa)$ to $V[G{*}U{*}H]$.  In $V[G{*}U{*}H]$, define 
    \[I = \{ X \subseteq \kappa : 1 \Vdash_{\col(\mu,\kappa)} \check\kappa \notin j(X)\}.\]  By standard arguments, $I$ is $\kappa$-complete and normal.  

    \begin{claim}\label{claim:dense-embedding-to-collapse}
        There is a dense embedding $e$ from $\p(\kappa)/I$ to $\mathcal{B}(\col(\mu,\kappa))$.
    \end{claim}
    \begin{proof}
        Define $e : \p(\kappa)/I \to \mathcal{B}(\col(\mu,\kappa))$ by $e([X]_I) = \| \check\kappa \in j(X) \|$.  It is clear that $e$ preserves order and antichains.  
    
        Since $|\col(\mu,\kappa)| = \kappa$, $\p(X)/I$ is $\lambda$-c.c.  If $\la [X_\alpha]_I : \alpha<\kappa \ra$ is a maximal antichain, then by normality, $\nabla_\alpha X_\alpha := \{ \beta : (\exists \alpha<\beta) \beta \in X_\alpha \}$ is equivalent to $\kappa$ modulo $I$, and thus $1_{\col(\mu,\kappa)}$ forces that $\kappa \in j(X_\alpha)$ for some $\alpha <\kappa$; in other words $\{ e([X_\alpha]_I) : \alpha < \kappa \}$ is maximal in $\mathcal{B}(\col(\mu,\kappa))$.  Thus $e$ is a complete embedding.

        In order to show that $e$ is dense, it suffices to show that there is a name $\sigma$ in the forcing language of $\p(\kappa)/I$ such that whenever $g \subseteq \col(\mu,\kappa)$ is generic and $F = e^{-1}[g]$, then $g = \sigma^F$. Indeed, for any $p \in \col(\mu,\kappa)$, there is $X \subseteq \kappa$ such that $[X]_I \Vdash \check p \in \sigma$, and so $e([X]_I) \leq p$.  
        
        To show the desired claim, let $g \subseteq \col(\mu,\kappa)$ be generic, let $F = e^{-1}[g]$, and let $G',U',H'$ be as above.  Let $N = \Ult(V[G{*}U{*}H],F)$, let $i : V[G{*}U{*}H] \to N$ be the ultrapower embedding, and let $k : N \to M[G'{*}U'{*}H']$ be defined by $k([h]_F) = j(h)(\kappa)$.  Then $k$ is elementary and $j = k \circ i$.  
    \[\begin{tikzcd}
            V[G{*}U{*}H]\ar[d, "i"]\ar[r, "j"] & M[G'{*}U'{*}H'] \\
            N\ar[ur, "k"]
        \end{tikzcd}\]      

        Since $I$ is normal and $\kappa$-complete, $\crit(i) = \kappa$. In particular, for every $X\subseteq \kappa$ in $V[G{*}U{*}H]$, $X \in N$ and thus $(\kappa^+)^N \geq \lambda$. Since $\crit(k) > \kappa$ (as $\kappa = k([\id]_F)$),  and $\crit(k)$ must be a cardinal in $N$, $\crit k \geq \lambda$. Since $k(i(\kappa)) = j(\kappa) = \lambda$, in fact $\crit(k) > \lambda$. This implies that $i(G) = j(G) = G'$. 
        
        Thus, $g$ can be defined as the filter over the third coordinate given by the translating $i(G)^{\tau^*}_{\col(\mu,\kappa)}$ into a filter over $\Sh(\mu,\kappa) * \dot\bbU(\dot G) * \dot\col(\mu,\kappa)$ via the embedding of Lemma~\ref{3step}.
       \end{proof}
       So, $e^{-1}(\col(\mu,\kappa))$ is dense in the $I$-positive subsets of $\kappa$, as wanted.
\end{proof}

The basic arguments about $\Sh(\mu,\kappa)$, even the transitivity of the order, only work when $\mu$ is regular and uncountable. Thus, a different approach is needed for dense ideals on $\omega_1$.  Fortunately, suitable technology has already been presented in \cite{Eskew2016}.

For regular cardinals $\mu<\kappa$, the \emph{anonymous collapse} $\bbA(\mu,\kappa)$ is defined as follows.  Consider the complete Boolean algebra $\bbB = \mathcal{B}(\col(\mu,{<}\kappa)*\dot\add(\kappa))$.  Let $\dot X$ be the canonical name for the generic subset of $\kappa$ added by the second step. Then let $\bbA(\mu,\kappa)$ be the complete subalgebra generated by $\{ \| \check\alpha \in \dot X \| : \alpha<\kappa \}$.  When $\mu$ is uncountable, $\Sh(\mu,\kappa)$ corresponds to a complete subalgebra $\bbS \subseteq \bbB$, and by the argument for Lemma~\ref{distributive}, $\bbA(\mu,\kappa) \subseteq \bbS$.  But, unlike our version of the Shioya forcing,  $\bbA(\mu,\kappa)$ makes sense also for $\mu = \omega$.  

The following lemma has been known in the folklore for a long time---see for example \cite[p.\ 27]{FMS}.
A proof for a more general class of forcing notions is given in \cite{Eskew2016}.

\begin{lemma}
\label{woodin}
    If $\bbP$ is a reasonable $(\mu,\kappa)$-collapse, then whenever $G \subseteq \bbP$ is generic, there is some further forcing extension $W \supseteq V[G]$ in which there exists a $V$-generic $H \subseteq \col(\mu,{<}\kappa)$ such that $\p(\mu)^{V[H]} = \p(\mu)^{V[G]}$.
\end{lemma}




\begin{theorem}
\label{omega1}
    If $\kappa$ is almost-huge with target $\lambda$, then $\bbA(\omega,\kappa) * \dot\Sh(\kappa,\lambda)*\dot\bbU(\dot G)$ forces that there is a normal $\omega_1$-dense ideal on $\omega_1$.
\end{theorem}

\begin{proof}
    Let $j: V \to M$ be an embedding witnessing that $\kappa$ is almost-huge with target $\lambda$.  Let us assume that $j(\lambda)<\lambda^+$ and $j[\lambda]$ is cofinal in $j(\lambda)$.

    Let us fix some generic filters: 
    \begin{itemize}
        \item Let $A * G * U \subseteq \bbA(\omega,\kappa) * \dot\Sh(\kappa,\lambda)*\dot\bbU(\dot G)$ be generic.  
        \item Let $\bbB =  \mathcal{B}(\col(\omega,{<}\kappa)*\dot\add(\kappa))/A$, and take a generic $g \times B \subseteq \col(\omega,\kappa) \times \bbB$ over $V[ A * G * U]$.  
        \item Let $H * X \subseteq \col(\omega,{<}\kappa)*\dot\add(\kappa)$ be the filter corresponding to $A*B$. 
    \end{itemize}
    Since $B$ is generic over $V[A{*}G{*}U{*}g]$, $G*U*g$ is generic over $V[A{*}B] = V[H{*}X]$.  Since $g \times B$ is generic for a forcing of size $\kappa$ from $V[A]$, Lemmas~\ref{easton} and \ref{Udist} imply that adjoining $U$ to $V[A{*}G{*}g{*}B] = V[H{*}X{*}g{*}G]$ adds no ${<}\lambda$-sequences.  Let $\bbR_0 = \bbR^{V[A]}$, the reals of $V[A]$, and let $\bbR_1 = \bbR^{V[A{*}G{*}g{*}B]}$.
    
    By Lemma~\ref{distributive}, in $V[A]$ there is a $\Sh(\kappa,\lambda)$-name for a $\lambda$-distributive forcing $\dot\bbQ$ such that $\Sh(\kappa,\lambda)*\dot\bbQ$ is equivalent to $\col(\kappa,{<}\lambda)*\dot\add(\lambda)$.  
    Again by Lemma~\ref{easton}, $\bbQ$ remains $\lambda$-distributive in $V[A{*}G{*}g{*}B]$.  Thus in some forcing extension of $V[H{*}X{*}g{*}G]$, there is a filter $K \subseteq \col(\kappa,{<}\lambda)^{V[H]}$ that is generic over $V[H{*}X{*}g]$ and such that $\bbR^{V[H{*}X{*}g{*}K]} = \bbR_1$.  Since $\add(\kappa) \times \col(\omega,\kappa) \times \col(\kappa,{<}\lambda)$ is a reasonable $(\omega,\lambda)$-collapse in $V[H]$, Lemma~\ref{woodin} implies that in some further forcing extension, there is a filter $H'' \subseteq \col(\omega,{<}\lambda)$ that is generic over $V[H]$ with the property that $\bbR^{V[H][X*g*K]} = \bbR^{V[H][H'']}$.  Rearranging $H \times H''$ into $H' \subseteq \col(\omega,{<}\lambda)$ such that $H' \rest \kappa = H$, we have that the embedding lifts to $j : V[H] \to M[H']$.  
    
    This embedding restricts to $j : V(\bbR^{V[H]}) \to M(\bbR^{V[H']})$, which is also elementary, and we have $\bbR^{V[H]} = \bbR_0$ and $\bbR^{V[H']} = \bbR_1$.  Thus the restricted embedding can be defined in $V[A{*}G{*}g{*}B]$.

    In $V[A{*}G{*}g{*}B]$, $\p(\lambda)^{M[A{*}G{*}g{*}B]}$ has size $\lambda$.  Furthermore, for every set of ordinals $x \in V[A{*}G{*}g{*}B]$ of size ${<}\lambda$, there is a poset $\bbP \in V_\lambda$, a filter $h \subseteq \bbP$, and a $\bbP$-name $\tau \in M$ such that $x = \tau^h$. Thus $M(\bbR_1)$ is ${<}\lambda$-closed in $V[A{*}G{*}g{*}B]$.  
    
    Using these observations, we can build a set $X' \subseteq \lambda$ that is $\add(\lambda)$-generic over $M[A{*}G{*}g{*}B]$ and such that $X' \cap \kappa = X$.  In particular we can further lift the embedding to $j : V(\bbR_0)[X] \to M(\bbR_1)[X']$.   
    $X$ generates the generic $A$, so by elementarity, $X'$ generates an $M$-generic filter $j(A)=A' \subseteq \bbA(\omega,\lambda)$.

    We will construct a descending sequence $\la q_\alpha : \alpha < \lambda\ra \subseteq \Sh(\lambda,j(\lambda))^{M[A']}$ such that:
    \begin{itemize}
         \item for every $p \in G$, there is $\alpha <\lambda$ such that $q_\alpha \leq j(p)$, and
        \item for every $Y \subseteq \kappa$ in $V[A{*}G]$ and every name $\tau \in V[A]$ such that $\tau^G = Y$, there is $\alpha<\lambda$ such that $q_\alpha$ decides whether $\kappa \in j(\tau)$.
    \end{itemize}
    Then we define an ultrafilter $F$ on $\p(\kappa)^{V[A*G]}$ by $Y \in F$ iff for some name $\tau$ such that $\tau^G = Y$ and some $\alpha<\lambda$, $q_\alpha \Vdash \kappa \in j(\tau)$.  $F$ is well-defined because if $\tau_0^G = \tau_1^G$, then there is some $p \in G$ such that $p \Vdash \tau_0 = \tau_1$, and then there is some $\alpha<\lambda$ such that $q_\alpha \Vdash j(\tau_0) = j(\tau_1)$.  $F$ is upwards-closed because if $\tau_0^G\subseteq\tau_1^G$, and $\tau_0^G \in F$, then there is some $p \in G$ such that $p \Vdash \tau_0 \subseteq\tau_1$ and some large enough $\alpha$ such that $q_\alpha \Vdash \kappa \in j(\tau_1)$.  $F$ is closed under diagonal intersections from $V[A{*}G]$ because if $\la Y_i : i <\kappa \ra \in V[A{*}G]$ then there is a sequence of names $\la \tau_i : i<\kappa \ra \in V[A]$ such that $\tau_i^G = Y_i$ for $i<\kappa$. If each $Y_i \in F$, then there is some $\alpha<\lambda$ such that $q_\alpha \Vdash \kappa \in j(\tau_i)$ for $i<\kappa$, and it follows that $q_\alpha \Vdash \kappa \in j(\triangle_i \tau_i)$.
    $F$ is ``ultra'' because for every relevant name $\tau$, there is some $\alpha<\lambda$ such that either $q_\alpha \Vdash \kappa \in j(\tau)$ or $q_\alpha \Vdash \kappa \in \lambda \setminus j(\tau)$.

    Let $\la\la\la \bbP_\alpha,1,\tau_\alpha \ra : \alpha <\lambda\ra,\la\pi_{\beta\alpha} : \alpha<\beta<\lambda \ra, \la q_{\beta\alpha} : \alpha<\beta<\lambda\ra\ra$ be the uniformizing sequence obtained from the generic filter $U$.
    For each $\alpha<\lambda$, let $m_\alpha \in j(\bbP_\alpha)$ be $\inf \{ j(p) : p \in G^{\tau_\alpha}_{\bbP_\alpha} \}$.  For $\alpha<\beta$, the genericity of $G^{\tau_\beta}_{\bbP_\beta}$ and the density of $\dom\pi_{\beta\alpha}$ imply that $\dom \pi_{\beta\alpha} \cap G^{\tau_\beta}_{\bbP_\beta}$ is a downwards-dominating subset of $G^{\tau_\beta}_{\bbP_\beta}$.  It follows by the $\kappa$-closedness of $\dom \pi_{\beta\alpha}$ that $m_\beta \in \dom j(\pi_{\beta\alpha})$, and by continuity that $m_{\alpha} = j(\pi_{\beta\alpha})(m_\beta)$.
    Since $m_\beta \leq j(q_{\beta\alpha})$, $\la j(\bbP_\beta),m_\beta,\tau_\beta\ra\leq\la j(\bbP_\alpha),m_\alpha,\tau_\alpha\ra$.  It follows that for all $\alpha<\beta<\gamma<\lambda$,  $j(\pi_{\beta\alpha})\circ j(\pi_{\gamma\beta}) = j(\pi_{\gamma\alpha})$ below $m_\gamma$, whenever these are both defined.
    Thus for each $\gamma<\lambda$, $\la m_\alpha : \alpha <\gamma \ra$ is an element of 
    $\varprojlim(\la j(\bbP_\alpha) : \alpha < \gamma\ra,\la j(\pi_{\beta\alpha}) : \alpha<\beta<\gamma\ra)$, which can be computed in  $M[A']$ since $U \rest \gamma \in M(\bbR_1)$.

    Let $\vec\bbP = \la j(\bbP_\alpha) : \alpha < \lambda\ra$ and $\vec\pi = \la j(\pi_{\beta\alpha}) : \alpha<\beta<\lambda\ra$.
    We claim that for all $\gamma < \lambda$, the map
    $$\sigma_\gamma : E^{\la\vec\bbP,\vec\pi\ra}_\gamma \to \varprojlim(\vec\bbP\rest\gamma,\vec\pi\rest\gamma),$$
    defined by $p \mapsto \la j(\pi_{\gamma\alpha})(p) : \alpha<\gamma \ra$, is a projection below $m_\gamma$.
    To see this, let $\gamma<\lambda$ and use the distributivity of $\bbU(G)$ to find some $\xi<\lambda$ such that $U \rest (\gamma+1) \in V[A][G^{\tau_\xi}_{\bbP_\xi}]$.  In this model, $U \rest (\gamma+1)$ has the property that for all $x \in [\gamma+1]^{<\kappa}$ with top element $\gamma$, if $\vec\bbQ = \la \bbP_\alpha: \alpha \in x \ra$ and $\vec\chi = \la \pi_{\beta\alpha}: \{\alpha,\beta\} \in [x]^2 \ra$, then the map $\sigma_x: E^{\la\vec\bbQ,\vec\chi\ra}_\gamma \to \varprojlim(\vec\bbQ,\vec\chi)$ defined by $p \mapsto \la \pi_{\gamma\alpha}(p): \alpha \in x \cap \gamma\ra$ is a projection, below any condition that projects as in clause (\ref{invlimU}) of the definition of $\bbU(G)$.  
    
    If we force below $m_\xi$ over $M[A']$, we get a lifted elementary embedding $j  :  V[A][G^{\tau_\xi}_{\bbP_\xi}] \to M[A'][K]$ for some $K \subseteq j(\bbP_\xi)$.  By elementarity, the property holds of $U' = j(U \rest (\gamma+1))$ in $M[A'][K]$.  In particular, it holds for $x = j[\gamma+1]$.  By re-indexing $U' \rest j[\gamma+1]$ by $\gamma+1$, we see that the desired claim holds in $M[A'][K]$. Since $j(q)_{j(\beta),j(\alpha)} = j(q_{\beta\alpha})\geq m_{\beta}$ for all $\{\alpha,\beta \} \in [\gamma+1]^2$, $\sigma_\gamma$ is a projection below $m_{\gamma}$. But $U' \rest j[\gamma+1] \in M[A']$ by the $\lambda$-closure of $j(\bbP_\xi)$, so the desired claim already holds in $M[A']$.

    To construct the sequence $\la q_\alpha : \alpha<\lambda\ra$, first enumerate all ordinals below $j(\lambda)$ as $\la\xi_\alpha : \alpha<\lambda\ra$.
    Let $\alpha_0$ be such that it is forced that $j(\dom \tau_{\alpha_0}) > \xi_0$.  Let $p_0 \in E^{\la\vec\bbP,\vec\pi\ra}_{\alpha_0}$ be below $m_{\alpha_0}$ and decide $j(\tau_{\alpha_0})(\xi_0)$.  Then let $\alpha_1 \geq \alpha_0$ be such that it is forced that $j(\dom \tau_{\alpha_1}) > \xi_1$, and let $p_1 \leq m_{\alpha_1}$ be in $E^{\la\vec\bbP,\vec\pi\ra}_{\alpha_1}$ and such that $j(\pi_{\alpha_1\alpha_0})(p_1) \leq p_0$ and $p_1$ decides $j(\tau_{\alpha_1})(\xi_1)$.  Continuing in this fashion, suppose we reach a limit $\zeta<\lambda$ such that we have an increasing sequence of ordinals $\la \alpha_i : i < \eta \ra$ and a sequence $\la p_i : i < \eta \ra$ such that for $i<\eta$,
    \begin{itemize}
        \item $p_i \in E^{\la\vec\bbP,\vec\pi\ra}_{\alpha_i}$,
        \item $p_i$ decides $j(\tau_{\alpha_i})(\xi_i)$, and
        \item $\la j(\pi_{\alpha_k\alpha_i})(p_k) : i \leq k < \eta \ra$ is a descending sequence in $E^{\la\vec\bbP,\vec\pi\ra}_{\alpha_i}$ below $m_{\alpha_i}$.
    \end{itemize}
    For $i<\eta$, let  $p^*_i = \inf \{ j(\pi_{\alpha_k\alpha_i})(p_k) : i \leq k < \eta\}$.   Let $\gamma = \sup_{i<\eta} \alpha_i$, and for $\beta <\gamma$, let $r_\beta = j(\pi_{\alpha_i\beta})(p_i^*)$ for any $i<\eta$ such that $\alpha_i \geq\beta$.  Then $\la r_\beta : \beta < \gamma \ra$ is a member of $\varprojlim(\vec\bbP\rest\gamma,\vec\pi\rest\gamma)$ below $\la m_\beta : \beta < \gamma\ra$.  Thus we may select some $\alpha_\eta \geq \gamma$ such that it is forced that  $j(\dom \tau_{\alpha_\eta})>\xi_\eta$, and we may select some $p_\eta \in E^{\la\vec\bbP,\vec\pi\ra}_{\alpha_\eta}$ below $m_{\alpha_\eta}$ such that  $\la j(\pi_{\alpha_\eta\beta})(p_{\alpha_\eta}) : \beta < \gamma \ra \leq \la r_\beta : \beta<\gamma\ra$ and $p_{\alpha_\eta}$ decides $j(\tau_{\alpha_\eta})(\xi_\eta)$.  Thus the induction can continue.

    Put $q_i = \la j(\bbP_{\alpha_i}),p_i,j(\tau_{\alpha_i})\ra$ for $i<\lambda$.  For each $p\in G$, there is some sufficiently large $i$ and some $r \in G^{\tau_{\alpha_i}}_{\bbP_{\alpha_i}}$ such that $\la\bbP_{\alpha_i},r,\tau_{\alpha_i}\ra \leq p$, and so $q_i \leq  \la j(\bbP_{\alpha_i}),m_{\alpha_i},j(\tau_{\alpha_i})\ra \leq j(p)$.
    Recall the $\Sh(\kappa,\lambda)$-generic function $f : \lambda \to 2$ from Lemma~\ref{distributive}, defined as $f(\beta) = n$ iff there is $\la\bbP,p,\tau\ra \in G$ such that $p \Vdash \tau(\beta) = n$. For every $Y \subseteq \kappa$ in $V[A{*}G]$, there is $\beta<\lambda$ such that $Y = \{ i <\kappa : f(\beta+i) = 1\}$.  If $\dot x_\beta$ is a name for this set and $\tau$ is any name such that $\tau^G = Y$, then there is $p \in G$ such that $p \Vdash \tau = \dot x_\beta$.
    If $i<\lambda$ is sufficiently large, then $q_i \leq j(p)$ and $q_i$ decides $j(\tau_{\alpha_i})(j(\beta)+\kappa)$, meaning it decides whether $\kappa \in j(\tau)$.  Therefore, we have our $V[A{*}G]$-normal ultrafilter $F$.

    Let $I = \{ Y \subseteq \kappa : 1 \Vdash_{\col(\omega,\kappa) \times \bbB} Y \notin \dot F \}$.  As in the proof of Theorem~\ref{anysuccreg}, $I$ is a normal ideal on $\kappa$, and the map $[Y]_I \mapsto \| \check Y \in \dot F \|$ is a complete embedding from $\p(\kappa)/I$ into $\mathcal{B}(\col(\omega,\kappa) \times \bbB)$.  Since the latter has a dense set of size $\omega_1$, $I$ is $\omega_1$-dense.
    \end{proof}

\begin{corollary}
    Suppose $\kappa_1<\kappa_2<\kappa_3$ are such that for $n\in\{1,2\}$,  $\kappa_n$ is almost-huge with target $\kappa_{n+1}$.  Then $\bbA(\omega,\kappa_1) * \dot\Sh(\kappa_1,\kappa_2) * \dot\bbU(\dot G) * \dot\col(\kappa_2,{<}\kappa_3)$
    forces that for $n\in\{1,2\}$, there is a normal ideal $I$ on $\omega_n$ such that 
    \[\p(\omega_n)/ I\cong \mathcal{B}(\col(\omega_{n-1},\omega_n)).\]
\end{corollary}

\section{More dense ideals}\label{sec: more-dense-ideals}

Towards producing longer intervals of cardinals carrying dense ideals, we investigate the effect of iterating $\Sh(\kappa,\lambda)*\dot\bbU(\dot G)$.  Recall that by Lemma~\ref{3step}, $\Sh(\kappa,\lambda)*\dot\bbU(\dot G) * \dot\col(\kappa,\lambda)$ is equivalent to $\col(\kappa,\lambda)$.

\begin{lemma}
\label{shproj}
    Suppose $\mu<\kappa<\lambda$ are regular uncountable with $\kappa,\lambda$ inaccessible. Let $\tau^*$ be the $\col(\mu,\kappa)$-name given by Lemma~\ref{3step}.  Then there is a projection $\psi$ defined on a dense subset,
    $$\psi : \col(\mu,\kappa) * \dot\Sh(\kappa,\lambda)^{\Sh(\mu,\kappa)*\dot\bbU(\dot G)} \to\Sh(\mu,\lambda) \rest \la\col(\mu,\kappa),1,\tau^*\ra.$$
\end{lemma}

\begin{proof}
    To avoid confusion, we will refer to the tree forcing that is equivalent to the three-step iteration $\Sh(\mu,\kappa)*\dot\bbU(\dot G) * \dot\col(\mu,\kappa)$ as $\bbT$.
    We will define $\psi$ on the dense set of conditions $\la t,\la\dot\bbP,\dot p,\dot\tau\ra\ra$ such that for some cardinal $\delta \in (\kappa,\lambda)$, $\dot\bbP$ is forced to have size $\delta$ and  collapse $\delta$ to $\kappa$, and $\dot\tau$ is forced to have domain $\delta$. 
    $\psi$ will map such a condition $\la t,\la\dot\bbP,\dot p,\dot\tau\ra\ra$  to one of the form $\la \bbT*\dot\bbP,\la t,\dot p\ra,\sigma\ra$.
    To define $\sigma$, we interpret $\dot\tau$ in $V^{\Sh(\mu,\kappa)*\dot\bbU(\dot G)*\dot\bbP}$ as coding a $\col(\mu,\kappa)$-name in the following way.  In $V$, fix some enumeration $\la c_\alpha : \alpha < \kappa \ra$ of $\col(\mu,\kappa)$.  Let $f$ be the interpretation of $\dot\tau$ in a $(\Sh(\mu,\kappa)*\dot\bbU(\dot G)*\dot\bbP)$-generic extension.  In this extension, for $\alpha<\delta$, let $A_\alpha = \{ c_\beta : f(\kappa\cdot\alpha+\beta) = 1 \}$, and let $\sigma_0$ be the $\col(\mu,\kappa)$-name $\bigcup_{\alpha<\delta} A_\alpha \times \{\check\alpha\}$.  Then let $\sigma_1$ be a name for the characteristic function of $\sigma_0$.
    In $V$, let $\sigma_2$ be a $(\bbT{*}\dot\bbP)$-name that is forced to evaluate to $\sigma_1^K$, where $K$ is the $\col(\mu,\kappa)$-generic on the third coordinate in the decomposition of $\bbT$ given by Lemma~\ref{3step}.  Let $\sigma$ be a name for the concatenation $\tau^* \,^{\frown} \sigma_2$.

    $\psi$ is order-preserving because if $\la t_1,\la\dot\bbP_1,\dot p_1,\dot\tau_1\ra\ra \leq \la t_0,\la\dot\bbP_0,\dot p_0,\dot\tau_0\ra\ra$ are in the domain of $\psi$, then $\Sh(\mu,\kappa)*\dot\bbU(\dot G)$ forces that there is a $\kappa$-continuous projection from a dense $\kappa$-closed subset of $\dot\bbP_1$ to $\dot\bbP_0$.  If $\dot\pi$ is a name for such a projection such that $t_1$ forces that $\dot\pi$ witnesses the ordering of the second coordinates, then $\la t,\dot p \ra \mapsto \la t,\dot\pi(\dot p)\ra$ is a $\mu$-continuous projection from a dense subset of $\bbT* \dot\bbP_1$ to $\bbT * \dot\bbP_0$, and it clearly witnesses $\psi(\la t_1,\la\dot\bbP_1,\dot p_1,\dot\tau_1\ra\ra)\leq \psi(\la t_0,\la\dot\bbP_0,\dot p_0,\dot\tau_0\ra\ra)$.
    To show $\psi$ is a projection, suppose 
    $$\la\bbQ,q,\dot x\ra \leq \psi(t_0,\la\dot\bbP_0,\dot p_0,\dot\tau_0\ra) = \la\bbT*\dot\bbP_0,\la c_0,\dot p_0\ra,\sigma_0\ra.$$
    We may assume that for some regular cardinal $\delta > |\bbT * \dot\bbP_0|$, $\bbQ$ has size $\delta$, it  collapses $\delta$ to $\mu$, and $\dot x$ is forced to have length $\delta$.  Let $\bbQ_0 = \bbQ$ and $\bbQ_1 = \bbT*\dot\bbP_1 = \bbT * (\dot\bbP_0 \times \dot\col(\kappa,\delta)^{\Sh(\mu,\kappa)})$.   Let $\chi_0$ be a projection witnessing the ordering $\la\bbQ,q,\dot x\ra \leq \la\bbT*\dot\bbP_0,\la t_0,\dot p_0\ra,\sigma_0\ra$, and let $\chi_1 : \bbQ_1 \to \bbT*\dot\bbP_0$ be $\la a,\la b,c\ra\ra \mapsto \la a,b \ra$.  

    By Lemma~\ref{lem:embedding-from-projection}, for $n<2$ there is a dense $\mu$-closed $E_n \subseteq (\bbT*\dot\bbP_0) \times \col(\mu,\delta)$ and a $\mu$-continuous dense embedding $e_n : E_n \to \bbQ_n$, with the property that if $\la \la a,b \ra, c\ra \in E_n$, then $\chi_n(e_n(\la \la a,b\ra,c \ra )) = \la a, b\ra$.
    

    Let $E = E_0 \cap E_1$, and let $F_n = e_n[E]$.  There is a $\mu$-continuous isomorphism $\iota : F_1 \to F_0$ defined by $\iota = e_0 \circ e_1^{-1}$.
    If $r \in F_1$ and $r = e_1(a,b)$, then 
    $$\chi_0 \circ \iota(r) = \chi_0 \circ e_0(a,b) = a = \chi_1 \circ e_1(a,b) = \chi_1(r).$$
    If $\la t_1,\la\dot p_1,\dot d_1\ra\ra = r \in F_1$ is such that $\iota(r) \leq q$, then since $q \Vdash \la t_0,\dot p_0\ra \in \chi_0(\dot G_\bbQ)$, we must have $\la t_1,\dot p_1\ra = \chi_1(r) = \chi_0(\iota(r)) \leq \la t_0,\dot p_0\ra$.

    Let $\delta_0$ be the forced domain of $\dot\tau_0$.  In the extension by $\Sh(\mu,\kappa)*\dot\bbU(\dot G)*\dot\bbP_1$, forcing with $\col(\mu,\kappa)$ yields a generic for $\bbT*\dot\bbP_1$, which then gives a $\bbQ$-generic $G_\bbQ$ by projecting with $\iota$, with which we can evaluate $\dot x$.  Let $\dot y$ be a $\col(\mu,\kappa)$-name that is forced to evaluate to $\dot x^{\dot G_\bbQ}$.  For $\delta_0\leq\alpha<\delta$, let $A_\alpha \subseteq\col(\mu,\kappa)$ be the set of all conditions forcing $\dot y(\alpha) = 1$.  Then in the extension by $\Sh(\mu,\kappa)*\dot\bbU(\dot G)$, let $\dot A_\alpha$ be a $\bbP_1$-name for such an $A_\alpha$, and let $\tau_1$ be a name for a binary sequence of length $\delta$ such that for $\delta_0\leq\alpha<\delta$, it is forced that for $\beta<\kappa$,  $\tau_1(\kappa\cdot\alpha+\beta) = 1$ iff $c_\beta \in \dot A_\alpha$, where $\{ c_\beta : \beta<\kappa \}$ is the fixed enumeration of $\col(\mu,\kappa)$ from $V$.  Let $\tau_1 \rest \delta_0$ be forced to agree with $\tau_0$.

    If $\la t_1,\la\dot p_1,\dot d_1\ra\ra$ is as above, then by the definition of $\psi$, 
    $$\psi\left( t_1,\la \dot\bbP_1,\la \dot p_1,\dot d_1\ra,\dot\tau_1\ra\right) \leq \la\bbQ,q,\dot x\ra,$$
    as witnessed by $\iota : F_1 \to \bbQ$.
\end{proof}

\begin{definition}
    A condition $\la\bbP,p,\tau\ra \in \Sh(\kappa,\lambda)$ is called a \emph{double coding condition} if for some regular cardinal $\delta$,
    \begin{itemize}
        \item $|\bbP| = \delta$, with enumeration $\{ p_\alpha : \alpha<\delta\}$,
        \item $\bbP$ collapses $\delta$ to $\kappa$,
        \item $\tau$ is forced to have length $\delta$, and
        \item for some $\alpha<\delta$, it is forced that for all $\beta<\delta$ and all $\gamma<\kappa$, 
        $$\Vdash_\bbP\tau(\alpha+ \kappa\cdot\beta + \gamma) = 1 \leftrightarrow p_\beta \in \dot G.$$
    \end{itemize}
\end{definition}
In other words, $\dot G$ is coded by some final segment of $\kappa$-blocks of 0's and 1's.  It is a special type of strong coding that also occurs densely often in $\Sh(\kappa,\lambda)$.

\begin{lemma}
    Suppose $\col(\mu,\kappa)$ forces that $\la\dot\bbP,\dot p,\dot\tau\ra \in \Sh(\kappa,\lambda)^{\Sh(\mu,\kappa)*\dot\bbU(\dot G)}$ is a double coding condition with witnessing cardinal $\delta$.  Then for all $c \in \col(\mu,\kappa)$, $\psi(c,\la\dot\bbP,\dot p,\dot\tau\ra)$ is a coding condition in $\Sh(\mu,\lambda)$.
\end{lemma}

\begin{proof}
    Let $\psi(c,\la\dot\bbP,\dot p,\dot\tau\ra) = \la \col(\mu,\kappa)*\dot\bbP,\la c,\dot p\ra,\sigma\ra$.  Let $\dot G_0,\dot G_1$ be names for the generic filters over the first and second coordinates of $\col(\mu,\kappa)*\dot\bbP$ respectively.  By the way $\sigma$ is constructed from $\tau$, it is forced by $\col(\mu,\kappa)*\dot\bbP$ that there is $\alpha<\delta$ such that for all $\beta<\delta$, $\sigma(\alpha+\beta) = 1$ iff $\dot p_\beta \in \dot G_1$.  Since $\sigma$ is forced to end-extend $\tau^*$, it is forced that $\dot G_0$ can be read off from $\sigma \rest \kappa\cdot 2$.  Thus if $G_0*G_1 \subseteq \col(\mu,\kappa)*\dot\bbP$ is generic, and $W$ is any further generic extension, then any $V$-generic $H_0*H_1 \subseteq \col(\mu,\kappa)*\dot\bbP$ that differs from $G_0*G_1$ would evaluate $\sigma$ differently.
\end{proof}

\begin{lemma}
\label{iteratedshioyadist}
    The quotient forcing of the projection $\psi$ of Lemma~\ref{shproj} is forced to be $\lambda$-distributive.
\end{lemma}

\begin{proof}
    Suppose $B * H \subseteq \col(\mu,\kappa) * \dot\Sh(\kappa,\lambda)^{\Sh(\mu,\kappa)*\dot\bbU(\dot G)}$ is generic.  Let $G' = \psi(B{*}H)$, and let $G*U*K$ be the generic corresponding to $B$ per Lemma~\ref{3step}.  For any set of ordinals $x \in V[B{*}H]$ of size $<\lambda$, there is a $\col(\mu,\kappa)$-name $\dot x \in V[G{*}U{*}H]$ such that $\dot x^K = x$, and $\dot x \in V[G{*}U{*}H^\tau_\bbP]$ for some  double coding condition $\la\bbP,1,\tau\ra \in H$, so $x \in V[B{*}H^\tau_\bbP]$.  We can take names $\la\dot\bbP,\dot 1,\dot\tau\ra$ that are forced to be double coding with witness $\delta$, such that $\la\dot\bbP,\dot 1,\dot\tau\ra^{G*U} = \la\bbP, 1,\tau\ra$.  Let $\psi(1,\la\dot\bbP,\dot 1,\dot\tau\ra) = \la\col(\mu,\kappa)*\dot\bbP,\la 1,\dot 1\ra,\sigma\ra$.   Then we have that $(G')^\sigma_{\col(\mu,\kappa)*\dot\bbP} = B * H^\tau_\bbP$.  Thus $x \in V[G']$.
\end{proof}

\begin{lemma}
\label{iteratedshioyastrat}
    Suppose $\mu<\kappa<\lambda$ are regular cardinals, with $\kappa,\lambda$ inaccessible.  Suppose 
    $G_0 * U_0 * G_1 * U_1 \subseteq \Sh(\mu,\kappa) * \dot\bbU(\dot G_0) *\dot\Sh(\kappa,\lambda)* \dot\bbU(\dot G_1)$ is generic over $V$, and $K \subseteq \col(\mu,\kappa)$ is generic over $V[G_0{*}U_0{*}G_1{*}U_1]$.  
    
    Let $\bbT$ be the tree forcing of Lemma~\ref{3step}, and let $B \subseteq \bbT$ be the generic corresponding to $G_0*U_0*K$.  Let $\psi(B{*}G_1)=G_2 \subseteq \Sh(\mu,\lambda)$, where $\psi$ is the projection of Lemma~\ref{shproj}.  
    
    Then $\bbU(G_2)$ is $\lambda$-strategically closed in $V[B{*}G_1{*}U_1]$.
\end{lemma}
\begin{proof}
    Let $\la\la\bbR_i,1,\varsigma_i\ra : i<\lambda\ra$,  $\la\rho_{ji} : i<j<\lambda\ra$ and $\la r_{ji} : i < j < \lambda\ra$ be the sequences given by $U_1$.
Since double coding conditions are dense in $\Sh(\kappa,\lambda)$, the argument of Lemma~\ref{1ptext} shows that cofinally many $\la\bbR_i,1,\varsigma_i\ra$ are double coding.  Passing to a subsequence, let us assume that all $\la\bbR_i,1,\varsigma_i\ra$ are double coding.
    For each $i<j<\lambda$, choose $(\Sh(\mu,\kappa)*\dot\bbU(\dot G_0))$-names for these objects that evalutate to them:
    $\dot\bbR_i$, $\dot\varsigma_i$, $\dot\rho_{ji}$, $\dot r_{ji}$.  
We may assume there is an increasing sequence of cardinals $\la\delta_i : i < \lambda\ra$ such that $\dot\bbR_i$ is forced to be equivalent to $\dot\col(\check\kappa,\check\delta_i)$.
We may choose the names such that $\la\dot\bbR_i,1,\dot\varsigma_i\ra$ is forced to be double coding and $\dot\rho_{ji}$ is forced to be a projection from $\dot\bbR_j$ to $\dot\bbR_i$ defined on a $\kappa$-closed dense set.
For every $i<j$, some condition in $G_0*U_0$ will force that $\dot\rho_{ji}$ witnesses $\la\dot\bbR_j,\dot r_{ji},\dot\varsigma_j\ra \leq \la\dot\bbR_i,1,\dot\varsigma_i\ra$.  

    Suppose that in $V[B{*}G_1{*}U_1]$, Players I and II are playing the game on $\bbU(G_2)$.
    Suppose Player I plays $u_0 = \la\la\la\bbP^{u_0}_i,1,\tau_i^{u_0}\ra : i\leq \xi_0\ra,\la\pi^{u_0}_{ji} : i<j\leq\xi_0\ra, \la q^{u_0}_{ji} : i < j \leq \xi_0\ra\ra$.
The argument of Lemma~\ref{1ptext} shows that 
there is $\alpha_1 <\lambda$ such that Player II can respond with a condition $u_1$ of length $\xi_0+2$ with top $\Sh(\mu,\lambda)$-condition $\la\bbP_{\xi_0+1}^{u_1},1,\tau_{\xi_0+1}^{u_1}\ra=\la\bbT*\dot\bbR_{\alpha_1},1,\sigma_1\ra$, where $\la\bbT*\dot\bbR_{\alpha_1},1,\sigma_1\ra = \psi(1,  \la\dot\bbR_{\alpha_1},1,\dot\varsigma_{\alpha_1}\ra)$. 

    Suppose Player I responds with $u_2 = \la\la\la\bbP^{u_2}_i,1,\tau^{u_0}_i\ra : i\leq \xi_2\ra,\la\pi^{u_2}_{ji} : i<j\leq\xi_2\ra, \la q^{u_2}_{ji} : i < j \leq \xi_2\ra\ra$.  There is an ordinal $\alpha_3 < \lambda$ and a condition $\la t_3,\dot r_3\ra \in B * (G_1)^{\varsigma_{\alpha_3}}_{\bbR_{\alpha_3}}$ such that $\alpha_3>\alpha_1$, 
$t_3 \Vdash \dot r_3 \leq \dot r_{\alpha_3\alpha_1} \wedge \la\dot\bbR_{\alpha_3},\dot r_3,\dot \varsigma_{\alpha_3}\ra \leq \la\dot\bbR_{\alpha_1},1,\dot \varsigma_{\alpha_1}\ra $,
and $\psi(t_3,\la\dot\bbR_{\alpha_3},\dot r_3,\dot \varsigma_{\alpha_3}\ra)  = \la\bbT*\dot\bbR_{\alpha_3},\la t_3,\dot r_3\ra,\sigma_3\ra \leq \la\bbP_{\xi_2},1,\tau_{\xi_2}\ra$.
  
    
Let $\chi$ be a projection witnessing $\la\bbT*\dot\bbR_{\alpha_3},\la t_3,\dot r_3\ra,\sigma_3\ra \leq \la\bbP_{\xi_2},1,\tau_{\xi_2}\ra$.
    Let Player II's next move be $u_3$ of length $\xi_2+2$ such that:
\begin{itemize}
	\item $\la\bbP^{u_3}_{\xi_2+1},1,\tau^{u_3}_{\xi_2+1}\ra = \la\bbT*\dot\bbR_{\alpha_3},1,\sigma_3\ra$.
	\item $\pi^{u_3}_{\xi_2+1,\beta} =  \pi^{u_2}_{\xi_2,\beta}\circ\chi$ for $\xi_0+1<\beta \leq\xi_2$, $\pi^{u_3}_{\xi_2+1,\xi_0+1} = \id\times\dot\rho_{\alpha_3\alpha_1}$, and $\pi^{u_3}_{\xi_2+1,\beta} = \pi^{u_2}_{\xi_0+1,\beta} \circ (\id\times\dot\rho_{\alpha_3\alpha_1})$ for $\beta\leq\xi_0$, restricted to the dense $\mu$-closed sets on which these are defined.  
\item $q_{\xi_2+1,\beta}^{u_3} \leq \la t_3,\dot r_3\ra$ projects below $q_{\xi_2,\beta}^{u_2}$ for $\xi_0+1<\beta\leq\xi_2$, $q_{\xi_2+1,\xi_0+1}^{u+3} = \la t_3,\dot r_3 \ra$, and $q^{u_3}_{\xi_2+1,\beta}\leq \la t_3,\dot r_3\ra$ projects below $q_{\xi_0+1,\beta}^{u_2}$ for $\beta \leq \xi_0$.
\end{itemize}
The argument that this fulfills requirement (\ref{invlimU}) is the same as for Player II's second move in Lemma~\ref{strongstrategic}.


    Assume inductively that $\delta<\lambda$, and the game has gone on for $\delta$-many rounds, with the following properties:
    \begin{itemize}
        \item At step $\beta$, the move $u_\beta = \la\la\la\bbP^{u_\beta}_i,1,\tau^{u_\beta}_i\ra: i \leq \xi_\beta\ra,\la\pi^{u_\beta}_{ji}: i<j\leq\xi_\beta\ra, \la q^{u_\beta}_{ji} : 
        i < j \leq \xi_\beta\ra \ra$ for $\beta<\delta$, with 
        Player II has played in odd finite rounds and even infinite rounds.
        As those conditions are descending, we may drop the superscript $u_\beta$.
        \item There is an increasing sequence of ordinals $\la\alpha_\beta : \beta<\delta\ra$ such that for odd finite and even infinite $\beta$, 
$\la\bbP_{\xi_\beta},1,\tau_{\xi_\beta}\ra = \psi(1,\la\dot\bbR_{\alpha_\beta},1,\dot\varsigma_{\alpha_\beta}\ra)$.
        \item For $\beta<\gamma$ both rounds where II played, $\pi_{\xi_\gamma\xi_\beta} = \id \times \dot\rho_{\alpha_\gamma\alpha_\beta}$ restricted to some $\mu$-closed dense set of conditions $\la t,\dot r\ra$ such that $t \Vdash \dot r \in \dom\dot\rho_{\alpha_\gamma\alpha_\beta}$, and $q_{\gamma\beta} = \la t,\dot r\ra \in B*(G_1)^{\varsigma_{\alpha_\gamma}}_{\bbR_{\alpha_\gamma}}$ such that 
        $t \Vdash \dot r \leq \dot r_{\alpha_\gamma\alpha_\beta} \wedge \la\dot\bbR_{\alpha_\gamma},\dot r,\dot\varsigma_{\alpha_\gamma}\ra \leq \la\dot\bbR_{\alpha_\beta},1,\dot\varsigma_{\alpha_\beta}\ra$.
    \end{itemize}
    
    Let $\delta = \bar\delta + 1$ be a successor ordinal and assume the last round was a move of Player I. Then, Player II responds just as in round 3.

    If $\delta$ is a limit, then $|\delta|^{V[B{*}G_1{*}U_1]} \leq \mu$.  
    Let $\xi_\delta = \sup_{\gamma < \delta} \xi_\gamma$ and $\alpha'_\delta = \sup_{\gamma < \delta} \alpha_\gamma$. Let $v\in U_1$ be large enough so that $\dom v$ is larger than $\alpha'_\delta$. 
    
    Here we must address a delicate point that was silenced until now. The generic filter for $\bbU(\dot{G}_1)$ is composed of names which are forced by conditions in $G_0{*}U_0$ to be compatible in $\Sh(\kappa,\lambda)$, as witnessed by the projections $\dot\rho_{ji}$. 
Those relations between pairs of those names are not forced outright but only by sufficiently strong conditions in the generic filter.
At the successor step, it is enough to let the corresponding values of $\vec q^{u_\delta}$ be strong enough to force that compatibility is witnessed by these projections $\dot\rho_{\alpha_j\alpha_i}$ in the desired way.  Since Player II only increases the length of the conditions by one, it is easy for them to maintain condition (\ref{invlimU}).  But at a limit step of cofinality $\mu$, we cannot find a single condition forcing the desired properties of the names on a cofinal subset.

    As $v$ and $\vec u = \langle u_\zeta  :  \zeta < \delta\rangle$ belong to $H_{\mu^+}$ of the generic extension, there is a forcing notion $\bbP$ appearing in $G_2$ such that $v$ and $\vec u$ are added by $\bbP$. Without loss of generality, $\bbP$ is of the form 
$\mathbb{T}\ast \dot\bbR$ where $\langle \mathbb{T}\ast \dot\bbR, 1, \sigma\rangle = \psi(1,\la\dot\bbR_i,1,\dot\varsigma_i\ra)$ for some $i<\lambda$, and it is the maximal coordinate of $v$. 
Let $\alpha_\delta + 1$ be the length of $v$. 
Pick $\bbP$-names $\dot v$ for $v$ and $\dot u'_\delta$ for the coordinate-wise union of the elements of $\vec u$. 
    As $\bbP$ collapses $\alpha_\delta$ and $\xi_\delta$ to $\mu$, we may fix enumerations  
   $\vec\alpha=\langle \hat\alpha_i  :  i < \mu\rangle$ of $\{\alpha_j : j < \delta \}$ and $\vec\xi=\langle \hat\xi_i  :  i < \mu\rangle$ of $\xi_\delta$ in the generic extension by $\bbP$.  Let $f : \mu \to \mu$ be such that $\xi_{f(i)}+1$ is the length of Player II's move with top poset $\bbT * \dot\bbR_{\hat\alpha_i}$.
    
    Let $\la\la t_i,\dot r_i\ra : i < \mu\ra$ be a descending sequence contained in the generic of $\bbT*\dot\bbR$ such that for every $i < \mu$, $\la t_i,\dot r_i\ra$ decides the values of 
$f \rest i, \vec\alpha\rest i,\vec\xi \rest i,
\dot{v}\restriction \{\hat\alpha_j  :  j < i\} \cup \{\alpha_\delta\},$ and $\dot u'_\delta \restriction \{\hat\xi_j  :  j < i\}$.
Moreover, let us assume that $\la t_i,\dot r_i\ra$ forces that for all $j<i$, 
$\la\bbP_{\hat\xi_{f(j)}},1,\tau_{\hat\xi_{f(j)}}\ra = \psi(1,\la\dot\bbR_{\hat\alpha_j},1,\dot\varsigma_{\hat\alpha_j}\ra)$, and
if $j,k<i$ and $\hat\xi_{f(k)}<\hat\xi_{f(j)}$, then $\pi^{u'_\delta}_{\hat\xi_{f(j)},\hat\xi_{f(k)}} = \id \times \dot\rho_{\hat\alpha_j,\hat\alpha_k}$.
    
Let $u_\delta$ be an extension of $u'_\delta$ with $\bbP$ at its top coordinate and the following projections. For $\zeta < \xi_\delta$, let $\hat\zeta$ be the ordinal $i$ that minimizes  $\hat\xi_{f(i)}\geq \zeta$, and let $\pi^{u_\delta}_{\xi_\delta,\zeta} = \pi^{u'_\delta}_{\hat\xi_{f(\hat\zeta)}, \zeta} \circ (\id \times \dot\rho_{\alpha_\delta, \alpha_{\hat\zeta}})$, restricted to the set of points at which
all compositions of maps from 
$$\{ \id \times \dot\rho_{\alpha_\delta,\hat\alpha_i} : i \leq \hat\zeta \}
\cup \{ \pi^{u'_\delta}_{jk} : k<j \wedge j,k \in \{ \hat\xi_{f(i)} : i \leq\hat\zeta \} \cup \{ \zeta \} \}$$
are defined.


    Let $q^{u_\delta}_{\delta,\zeta} = \langle t, \dot r\rangle$ be a condition in the generic of $\bbT*\dot\bbR$ such that, for the least ordinal $i$ that is closed under $f$ and such that $\zeta,\hat\xi_{f(\hat\zeta)}$ appear in $\vec\xi$ before $i$ and $\alpha_{\hat\zeta}$ appears in $\vec\alpha$ before $i$,
    \begin{itemize}
        \item $\la t,\dot r\ra \leq \la t_i,\dot r_i\ra$. 
        \item $t$ forces $\dot r$ to be stronger than
$\dot r_{\alpha_\delta,\hat\alpha_j}$ for every $j < i$, and to project below $\dot r_{\hat\alpha_j,\hat\alpha_k}$ for every $j,k< i$ such that $\hat\alpha_k<\hat\alpha_j$.

\item $\la t,\dot r\ra$ projects below $q^{u'_\delta}_{jk}$ whenever $k<j$ and $ j,k \in \{ \hat\xi_{f(l)} : l < i \} \cup \{ \zeta \}$.

    \end{itemize}

    Let $x \in [\xi_\delta]^{<\mu}$. We would like to verify requirement (\ref{invlimU}) with respect to $x$. Note that as no new sequences of ordinals of length $<\mu$ are introduced by $G_2$, $x \in V$.

    Let us first assume that $x \subseteq \{\xi_\gamma  :  \gamma\text{ is odd and finite or infinite and even}\}$. Let $q$ be a condition projecting below $q^{u_{\delta}}_{\xi_{\gamma_0},\xi_{\gamma_1}}$ for all $\xi_{\gamma_0}, \xi_{\gamma_1}\in x \cup \{\xi_\delta\}$. In particular, for every $\xi_{\gamma_0},\xi_{\gamma_1}\in x$, $q$ is stronger than a condition $\langle t_i, r_i\rangle$ that forces 
that $\pi^{u_\delta}_{\xi_{\gamma_0},\xi_{\gamma_1}} = \id\times\dot\rho_{\alpha_j,\alpha_k}$ for some $k<j<\delta$, and the first coordinate ($\bbT$-part)  of $q$ forces that the second coordinate projects below $\dot r_{\alpha_j,\alpha_k}$.

    Let $\vec s \in \varprojlim (\vec \bbP \restriction x, \vec \pi \restriction x)$ be stronger than the projection of $q$. Then, for every $\xi_\gamma \in x$, $\vec s(\xi_\gamma) \in \bbT \ast \bbR_{\alpha_\gamma}$. As the projections to the first coordinate 
are the identity, those projections are the same for all $\xi_\gamma\in x$ and stronger than the projection of $q$ to $\bbT$. The projections of the second coordinate are done according to the projections $\rho_{\alpha_{j}, \alpha_{k}}$.  

    As $q$ forces $v$ to be a condition and forces the projections in $u_\delta$ to be of the form $\id\times\dot\rho_{\alpha_j,\alpha_k}$, we conclude that we can strengthen the second coordinate of $q$ to obtain some $q'$ that projects below $\vec s$.  

    Let us now deal with the general case. Similarly to the proof of Lemma \ref{strongstrategic}, we first extend $x$ to $x \cup y$, where $y$ is contained in the indices of posets played by Player II
and $|y|<\mu$, and then apply the argument for $y$, using commutativity in order to obtain the a condition whose restriction to $x$ is as desired.

    In this case, let $y = \{\hat\xi_{f(\hat\zeta)}  :  \zeta \in x\}$. The argument that we can extend $\vec s \in \varprojlim (\vec \bbP \restriction x, \vec \pi \restriction x)$ to $\vec t\in \varprojlim (\vec \bbP \restriction (x \cup y), \vec \pi \restriction (x\cup y))$ is identical to the argument of Lemma~\ref{strongstrategic}. The main difference is that in this case, $y$ is not the list of \emph{all} steps of Player II but rather just those with index $\hat\xi_{f(\hat\zeta)}$ for $\zeta \in x$.
So, we must verify that the $\bbT$-part of $q$ forces that the relevant conditions $\dot r_{\alpha_{j},\alpha_{k}}$ are weaker than the projections of the second coordinate of $q$ to posets indexed by $y$. This is ensured by that way that $q$ was constructed:  For each such pair, let $i$ be the index in which the last one of them is enumerated.  Then $q = \langle t, \dot r\rangle$ is stronger than $\langle t_i, \dot r_i\rangle$, and $t$ forces that $\dot r$ projects below $\dot r_{\alpha_{j},\alpha_{k}}$ for the relevant ordinals $j,k$.
\end{proof}

\begin{theorem}
\label{iteratedshioya}
    Suppose $\kappa$ is almost-huge with target $\lambda$ and $\mu<\kappa$ is regular uncountable.  Then $\Sh(\mu,\kappa)*\dot\bbU(\dot G_0) *\dot\Sh(\kappa,\lambda)*\dot\bbU(\dot G_1)$ forces that there is a normal ideal $I$ on $\kappa$ such that $\p(\kappa)/I \cong \mathcal{B}(\col(\mu,\kappa))$.
\end{theorem}

\begin{proof}
    Let $j: V \to M$ be an embedding witnessing that $\kappa$ is almost-huge with target $\lambda$.  Let us assume that $j(\lambda)<\lambda^+$ and $j[\lambda]$ is cofinal in $j(\lambda)$.

    Suppose $G_0 * U_0 * G_1 * U_1 \subseteq \Sh(\mu,\kappa) * \dot\bbU(\dot G_0) *\dot\Sh(\kappa,\lambda)* \dot\bbU(\dot G_1)$ is generic over $V$, and $K \subseteq \col(\mu,\kappa)$ is generic over $V[G_0{*}U_0{*}G_1{*}U_1]$.  Let $\bbT$ be the tree forcing of Lemma~\ref{3step}, and let $B \subseteq \bbT$ be the generic corresponding to $G_0*U_0*K$.  Let $\psi(B{*}G_1)=G_0' \subseteq \Sh(\mu,\lambda)$, where $\psi$ is the projection of Lemma~\ref{shproj}.  By Lemma~\ref{3step}, $G_0 \subseteq G_0'$.  Thus the embedding can be lifted to $j: V[G_0] \to M[G_0']$.  Since every set of ordinals of size ${<}\lambda$ in $V[B{*}G_1{*}U_1]$ is in $V[(G_0')^{\tau}_\bbP]$ for some coding condition $\la\bbP,1,\tau\ra \in G_0'$, $M[G_0']$ is ${<}\lambda$-closed in $V[B{*}G_1{*}U_1]$.

    Let $u^*$ be the condition in $\bbU(G_0')$ given by Lemma~\ref{3step}.  Since $|\p(\lambda)^{M[G_0']}| = \lambda$ in $V[B{*}G_1]$, and $\bbU(G_0')$ is $\lambda$-strategically closed in $V[B{*}G_1{*}U_1]$ by Lemma~\ref{iteratedshioyastrat}, we can build a filter $U_0' \subseteq \bbU(G_0')$ with $u^* \in U_0'$ that is generic over $M[G_0']$.  By Lemma~\ref{3step}, we have $U_0 \subseteq U_0'$, and thus the embedding can be lifted further to $j : V[G_0{*}U_0] \to M[G_0'{*}U_0']$.

    Next, we use $U_1$ just as in the proof of Theorem~\ref{omega1} to build a descending sequence $\la q_\alpha : \alpha < \lambda \ra \subseteq \Sh(\lambda,j(\lambda))^{M[G_0'*U_0']}$ such that:
    \begin{itemize}
         \item for every $p \in G_1$, there is $\alpha <\lambda$ such that $q_\alpha \leq j(p)$, and
        \item for every $X \subseteq \kappa$ in $V[G_0{*}U_0{*}G_1]$ and every name $\tau \in V[G_0{*}U_0]$ such that $\tau^G = X$, there is $\alpha<\lambda$ such that $q_\alpha$ decides whether $\kappa \in j(\tau)$.
    \end{itemize}
    The construction is the same, so we will not repeat it, but we note that to carry out the argument we need that all ${<}\lambda$-sized subsets of $G_1$ are in $M[G_0']$, which is guaranteed by Lemma~\ref{iteratedshioyadist} and the closure of $M[G_0']$.  
    Then we define an ultrafilter $F$ on $\p(\kappa)^{V[G_0{*}U_0{*}G_1]}$ by $X \in F$ iff for some name $\tau$ such that $\tau^G = X$ and some $\alpha<\lambda$, $q_\alpha \Vdash \check\kappa \in j(\tau)$.  As before, $F$ is a $V[G_0{*}U_0{*}G_1]$-normal ultrafilter.
    In $V[G_0{*}U_0{*}G_1{*}U_1]$, we define an ideal $I$ on $\kappa$ by $X \in I$ iff $\col(\mu,\kappa)$ forces $X \notin \dot F$.  $I$ is a normal ideal, and $e : [X]_I \mapsto || X \in \dot F||_{\mathcal{B}(\col(\mu,\kappa))}$ is a complete embedding.  It follows that $I$ is $\kappa$-dense.

    It remains to show that $\p(\kappa)/I \cong \mathcal{B}(\col(\mu,\kappa))$.  For this, it suffices to show that there is a name $\tau$ forcing language of $\p(\kappa)/I$ such that, whenever $K \subseteq \col(\mu,\kappa)$ is generic, if $F = e^{-1}[K]$, then $K = \tau^F$.  This is very similar to Claim \ref{claim:dense-embedding-to-collapse}. Since unlike in Claim \ref{claim:dense-embedding-to-collapse}, we did not specify an $M$-generic for $j(\Sh(\mu,\kappa) * \bbU(\dot{G}_0) * \Sh(\kappa,\lambda) * \bbU(\dot G_1))$, we need to be a bit more careful about our elementary embeddings.
    
    Let $K,B,G_0',U_0'$ be as above, and let $F = e^{-1}[K]$.  Let $F_0 = F \cap V[G_0{*}U_0]$.  There is $\alpha_0 <\lambda$ such that $q_{\alpha_0}$ decides $F_0$ as the collection of $X \in \p(\kappa)^{V[G_0*U_0]}$ for which $\kappa \in j(X)$.  Let  $i_0 : V[G_0{*}U_0] \to N_0 = \Ult(V[G_0{*}U_0],F_0)$ be the ultrapower embedding.  There is an elementary map $k_0 : N_0 \to M[G_0'{*}U_0']$ given by $k_0([h]_{F_0}) = j(h)(\kappa)$, and $j \rest V[G_0{*}U_0] = k_0 \circ i_0$.  Since $\crit(i_0) = \kappa$, $\kappa = \mu^+$ in $V[G_0]$, and the critical point of $k_0$ is a cardinal of $V[G_0]$ above $\mu$, $\crit(k_0)$ is greater than $\kappa+\kappa$.  Recall that $G_0'$ possesses a strong coding condition of the form $\la \bbT,1,\tau^*\ra$, where $\tau^*$ is forced to have length $\kappa+\kappa$, and $(G_0')^{\tau^*}_\bbT = B$. 
    Since $\bbT,\tau^*\in V$ and these are coded by subsets of $\kappa$, they are in $N_0$.  
    By the elementarity of $k_0$, $\la\bbT,1,\tau^*\ra\in i_0(G_0)$, and moreover $i_0(G_0)^{\tau^*}_\bbT=B$.
    $K$ can be read off as the filter on the third coordinate of the decomposition of $\bbT$ given by Lemma~\ref{3step}.  Therefore, $K \in N_0$ and there is a $(\p(\kappa)/I)$-name $\tau$ that evaluates to $K$, as desired.
\end{proof}

\begin{corollary}
    Suppose $\la\kappa_n : n <\omega\ra$ is a sequence such that for all  $n$, $\kappa_n$ is almost-huge with target $\kappa_{n+1}$.  Then
     some forcing extension satisfies that for every $n<\omega$, there is a normal ideal $I$ on $\omega_{n+1}$ such that 
     $$\p(\omega_{n+1})/I \cong \mathcal{B}(\col(\omega_{n},\omega_{n+1})).$$
\end{corollary}

Suppose $\kappa$ is a huge cardinal, and $j : V \to M$ is an elementary embedding such that $\crit(j) = \kappa$ and $M^{j(\kappa)} \subseteq M$.  
If $\mathcal{U}$ is the normal ultrafilter on $\kappa$ derived from $j$, then a standard reflection argument (using the characterization of almost-hugeness from \cite[Theorem 24.11]{Kanamori}) shows that there is a set $A \in \mathcal U$ such that for $\alpha<\beta$ from $A$, $\alpha$ is almost-huge with target $\beta$.
A similar argument shows that for every $A \subseteq \kappa$, there is $X \in \mathcal U$ such that for all $\alpha \in X$, there is an embedding $i : V \to N$ witnessing that $\alpha$ is almost-huge with target $\kappa$, and $i(A \cap \alpha) = A$.

\begin{theorem}\label{thm:dense-ideals-everywhere}
If $\theta$ is a huge cardinal, then there is a generic extension in which $\theta$ remains inaccessible, 
there is a stationary set of measurable cardinals below $\theta$,
and $V_\theta$ satisfies that for all regular cardinals $\mu$, there is a normal ideal $I$ on $\mu^+$ such that 
     $\p(\mu^+)/I \cong \mathcal{B}(\col(\mu,\mu^+)).$\footnote{Shioya told the first author that the question of whether there can exist a dense ideal at the successor of a measurable cardinal was raised by Woodin at the 2004 workshop ``Singular Cardinal Combinatorics'' at the Banff International Research Station.}
\end{theorem}
\begin{proof}
Let $j: V \to M$ witness that $\theta$ is huge with target $\lambda$. 
Let $A \subseteq \theta$ be such that 
for all $\alpha<\beta$ in $A$, $\alpha$ is almost-huge with target $\beta$.  

Define an iteration $\la \bbP_\alpha,\dot\bbQ_\alpha : \alpha< \theta \ra$ as follows. 
Enumerate the elements of the closure of $A$ 
as $\la \kappa_\alpha : \alpha < \theta\ra$. 
Let $\bbQ_0 = \bbA(\omega,\kappa_0)$, and for $1 \leq n < \omega$, let $\dot\bbQ_n = \dot\Sh(\kappa_{n-1},\kappa_n) * \dot\bbU(\dot G_n)$, where $\dot G_n$ is a name for the generic on the first coordinate of this pair.  If $\alpha$ is a limit and $\kappa_\alpha$ is singular, let $\dot\bbQ_{\alpha} = \dot\Sh(\kappa_\alpha^+,\kappa_{\alpha+1}) * \dot\bbU(\dot G_\alpha)$, where $\dot G_\alpha$ is a name for the generic on the first coordinate.  If $\alpha$ is a limit and $\kappa_\alpha$ is inaccessible, let $\dot\bbQ_{\alpha} = \dot\Sh(\kappa_\alpha,\kappa_{\alpha+1}) * \dot\bbU(\dot G_\alpha)$, where $\dot G_\alpha$ is a name for the generic on the first coordinate.  Take the limits $\bbP_\alpha$ with Easton support.  
Using the strategic closure of $\Sh(\alpha,\beta)*\dot\bbU(\dot G)$ given by Lemma~\ref{strongstrategic}, standard arguments show that $\theta$ remains inaccessible and the set of regular uncountable cardinals below $\theta$ 
becomes the set of ordinals $\kappa_\alpha$ for successor $\alpha$, $\kappa_\alpha^+$ when $\kappa_\alpha$ is a singular limit, and $\kappa_\alpha$ when $\kappa_\alpha$ is inaccessible.  
By Theorem~\ref{iteratedshioya}, $\bbP_\theta$ 
forces that for all regular cardinals $\mu<\theta$ 
there is a normal ideal $I$ on $\mu^+$ such that $\p(\mu^+)/I \cong \mathcal{B}(\col(\mu,\mu^+))$.

To show that stationary-many measurable cardinals below $\theta$ are preserved, let $\delta<\theta$ be such that there is an almost-huge embedding $i : V \to M$ with $\crit(i) = \delta$, $i(\delta) = \theta$, and $i(A\cap\delta) = A$.  Since the set of such cardinals is stationary in $V$ and $\bbP_\theta$ is $\theta$-c.c., it suffices to show the measurability of such cardinals is preserved. Let $H \subseteq \bbP_\theta$ be generic, and let $H_\alpha = H \rest \bbP_\alpha$ for $\alpha<\theta$. The embedding $i$ easily lifts to $i : V[H_\delta] \to N[H]$.

Next, we must lift the embedding through $H_{\delta+1}$.  Let $H(\delta) = G_\delta * U_\delta \subseteq \Sh(\kappa_\delta,\kappa_{\delta+1})*\dot\bbU(\dot G_\delta)$.  Let $\la \la \bbR_\alpha,1,\tau_\alpha \ra : \alpha < \kappa_{\delta+1} \ra$, $\la \pi_{\beta\alpha} : \alpha<\beta<\kappa_{\delta+1} \ra$, and  $\la q_{\beta\alpha} : \alpha<\beta<\kappa_{\delta+1}\ra$ be the sequences given by $U_\delta$. 
For $\alpha<\kappa_{\delta+1}$, let $m_\alpha = \inf \{ i(p) : p \in (G_\delta)^{\tau_\alpha}_{\bbR_\alpha} \}$, and let $\vec m = \la m_\alpha : \alpha<\kappa_{\delta+1} \ra$.  Note that for all $\alpha<\beta<\kappa_{\delta+1}$, $m_\alpha = i(\pi_{\beta\alpha})(m_\beta)$, and $m_\beta \leq i(q_{\beta\alpha})$.  In $N[H]$, the pair
$$\la \vec\bbR,\vec\pi\ra = \la \la i(\bbR_\alpha) \rest m_\alpha : \alpha < \kappa_{\delta+1} \ra, \la i(\pi_{\beta\alpha}) : \alpha<\beta<\kappa_{\delta+1} \ra\ra$$
is a $\theta$-good inverse system of length $\kappa_{\delta+1} < \theta$.

For each $\alpha<\kappa_{\delta+1}$, it is possible to lift the embedding to $i: V[H_\delta][(G_\delta)^{\tau_\alpha}_{\bbR_\alpha}] \to N[H][K]$ by further forcing, and using the fact that $U_\delta \in N[H]$,
the same argument as in the proof of Theorem~\ref{omega1} shows that $\la\vec\bbR,\vec\pi\ra$ satisfies the hypotheses of Lemma~\ref{invlim_general}.  Thus for $\alpha<\kappa_{\theta+1}$, the map $\vec p \mapsto \vec p(\alpha)$ is a projection from $\la\vec\bbR,\vec\pi\ra$ to $i(\bbR_\alpha)\rest m_\alpha$, and it makes sense to define a $\varprojlim(\vec\bbR,\vec\pi)$-name $\tau^*$ as the concatenation of all $i(\tau_\alpha)$ for $\alpha<\kappa_{\theta+1}$, as evaluated through these projections.

Let $\gamma = \sup i[\kappa_{\delta+1}]$, and let $\bbQ$ be the disjoint sum of $\varprojlim(\vec\bbR,\vec\pi)$ and $\col(\theta,{<}\gamma)$, so that for a maximal antichain $\{ q_0,q_1 \}$, $\bbQ \rest q_0 = \varprojlim(\vec\bbR,\vec\pi)$ and $\bbQ \rest q_1 = \col(\theta,{<}\gamma)$.  Define a projection $\sigma_\alpha : \bbQ \to i(\bbR_\alpha)$ by putting $\sigma_\alpha$ below $q_0$ as the canonical restriction map, and $\sigma_\alpha$ below $q_1$ as anything that projects to $i(\bbR_\alpha)$, which can be done since $\gamma > |i(\bbR_\alpha)|$.  Amend $\tau^*$ to ensure that $q_1$ also forces it to code the generic.
Hence $\la\bbQ,q_0,\tau^*\ra$ is a coding condition in $\Sh(\theta,i(\kappa_{\delta+1}))^{N[H]}$ below $i[G_\delta]$.  Thus we may lift the embedding further to $i : V[H_\delta][G_\delta] \to N[H][K]$ whenever $K$ is a $\Sh(\theta,i(\kappa_{\delta+1}))^{N[H]}$-generic possessing $\la\bbQ,q_0,\tau^*\ra$.  Then we may argue similarly as in the limit case of Lemma~\ref{strongstrategic} that a lower bound $u^*$ to $i[U_\delta]$ can be constructed with $\la\bbQ,1,\tau^*\ra$ at the top.  Thus there exists a lower bound $p^*$ to $i[G_\delta {*} U_\delta]$.

Thus by forcing with 
$i(\Sh(\kappa_\delta,\kappa_{\delta+1})1*\dot\bbU(\dot G_\delta))^{N[H]}$ below $p^*$, we can lift the embedding $i$.  In this model, let $\mathcal U = \{ X \subseteq \delta : \delta \in i(X) \}$.  Then $\mathcal U$ is a $V[H_{\delta+1}]$-normal ultrafilter.  Since the last forcing was $\delta^{++}$-distributive, $\mathcal U \in V[H]$.  Since $\p(\delta)^{V[H]} = \p(\delta)^{V[H_{\delta+1}]}$, $\mathcal U$ is a normal measure over $\delta$ in $V[H]$.
\end{proof}

Working a bit harder, it is possible to show that even stronger large cardinal properties are preserved.  But we will leave the details of this exploration to a future work.

\section{Applications}\label{sec: applications}
\subsection{GCH, diamond, and reflection}\label{subsec: gch}

\begin{prop}
\label{gch}
    Suppose $\mu,\alpha<\kappa$, $I$ is a $\kappa$-complete ideal on $X$, and the forcing $\p(X)/I$ adds no elements of $[\alpha]^{<\mu}$.  Then $\alpha^{<\mu} < \kappa$.
\end{prop}

\begin{proof}
    Let $G \subseteq \p(X)/I$ be generic, and let $j_G: V \to \Ult(V, G)$ be the generic ultrapower embedding.  We do not assume that the ultrapower is well-founded, but it is well-founded up to $\kappa$, and $\kappa \leq \crit(j_G)$ (see \cite{foremanhandbook}).  Let $\vec x = \la x_\beta : \beta < \delta \ra$ be an enumeration of $[\alpha]^{<\mu}$.  If $\delta\geq\kappa$ then $j_G(\vec x)([\id]_G)$ is a member of $[\alpha]^{<\mu}$ in $V[G]$, and $j_G(\vec x)([\id]_G) \not= x_\beta = j(\vec x)(j(\beta))$ for any $\beta < \delta$.  This contradicts the assumption that $\p(X)/I$ adds no elements of $[\alpha]^{<\mu}$.
\end{proof}

It follows that if $\mu$ is regular, $\kappa = \mu^+$ and $I$ is a $\kappa$-complete ideal on $\kappa$ such that $\p(\kappa)/I$ is forcing-equivalent to $\col(\mu,\kappa)$, then $\kappa^{<\mu} = \kappa$, and therefore $I$ is $\kappa$-dense and thus $\kappa^+$-saturated.  By Solovay \cite{solovay},  whenever $G \subseteq \p(\kappa)/I$ is generic, $\Ult(V, G)$ is isomorphic to a transitive class $M \subseteq V[G]$ that is closed under $\mu$-sequences from $V[G]$.  Furthermore, $\p(\kappa)/I$ is a complete Boolean algebra by Tarski \cite{tarski}.

\begin{prop}
    If $\mu>\omega$ is regular, $\kappa = \mu^+$ and $I$ is a $\kappa$-complete ideal on $\kappa$ such that $\p(\kappa)/I$ is forcing-equivalent to $\col(\mu,\kappa)$, then $\diamondsuit_\mu(S)$ holds for all stationary $S \subseteq \mu$.
\end{prop}

\begin{proof}
    Let $g: \mu \to \kappa$ be a generic surjection corresponding to a $\col(\mu,\kappa)$-generic filter.  In $V[g]$, define a sequence $\vec a = \la a_\alpha : \alpha < \mu \ra$ by $a_\alpha = \{ \beta < \alpha : g(\alpha+\beta) = 0\}$.  It is a standard exercise to show that, for all stationary $S \subseteq \mu$ in $V$, $\la a_\alpha: \alpha \in S \ra$ is a $\diamondsuit_\mu(S)$-sequence in $V[g]$.
    
    Let $j: V \to M$ be the ultrapower embedding by the associated generic filter on $\p(\kappa)/I$.  By the closure of $M$, $\vec a \in M$, and for each stationary $S\subseteq \mu$ from $V$, $M$ satisfies that $\vec a \rest S$ is a $\diamondsuit_\mu(S)$-sequence.  Thus $M \models \diamondsuit_\mu(S)$ for each such $S$, and by elementarity and the fact that $\crit(j) > \mu$, $V \models  \diamondsuit_\mu(S)$ for all stationary $S \subseteq \mu$.
\end{proof}

\begin{prop}\label{prop:diagonal-stationary-relfection}
    Suppose $\omega<\mu<\kappa$ are regular, 
    and $I$ is a $\kappa$-complete ideal on $\kappa$ such that $\p(\kappa)/I$ is equivalent to a $\mu$-closed forcing.  Then for every sequence of stationary sets $\la S_\alpha : \alpha < \kappa \ra$ such that each $S_\alpha \subseteq \kappa \cap \cof({<}\mu)$, there is $\delta<\kappa$ such that for all $\alpha<\delta$, $S_\alpha \cap \delta$ is stationary in $\delta$.
\end{prop}

\begin{proof}
        By Proposition~\ref{gch}, $\alpha^{<\mu} < \kappa$ for all $\alpha<\kappa$, and thus by a well-known argument using sufficiently closed elementary submodels, $\mu$-closed forcing preserves every stationary $S \subseteq \kappa \cap \cof({<}\mu)$ (see for example, \cite{Magidor82}).
        It follows from \cite[Proposition 2.7]{foremanhandbook} that if $\p(X)/I$  is equivalent to a countably closed forcing, then it is precipitous (i.e.\ generic ultrapowers are always well-founded).  
        
        Let $G \subseteq \p(\kappa)/I$ be generic, and let $j: V \to M$ be the ultrapower embedding, where $M$ is transitive.  Let $\vec S = \la S_\alpha : \alpha < \kappa \ra$ be as hypothesized, and let $\vec T = j(\vec S)$.  Then for every $\alpha<\kappa$, $j(S_\alpha) \cap \kappa = S_\alpha$ is stationary in $V[G]$, and so $M \models (\forall \alpha<\kappa) \vec T(\alpha) \cap \kappa$ is stationary in $\kappa$.  The conclusion follows by elementarity.
\end{proof}
The same argument shows that the reflection principle $\mathrm{DSR}({<}\mu,\kappa\cap\cof({<}\mu))$, which was introduced by Fuchs \cite{fuchs}, holds. By \cite[Theorem 3.4]{fuchs}, the conclusion of the above proposition implies the failure of the principle $\square(\kappa,{<}\mu)$.\footnote{In Fuchs' definition, we require the existence of an ordinal $\gamma$, such that $S_\alpha\cap\gamma$ is stationary for club many $\alpha < \gamma$. In our case, we can take this club to be $\gamma$ itself and thus we obtain $\mathrm{DSR}({<}\kappa,\kappa\cap\cof({<}\mu))$.}

Let $\lambda$ be a cardinal. Recall that a partial $C$-sequence on $S\subseteq \lambda$ is a sequence of the form $\langle C_\alpha :\alpha \in S\rangle$ where $C_\alpha$ is closed and cofinal at $\alpha$. A partial $C$-sequence is \emph{coherent} if for every $\alpha \in S$, and $\beta \in \acc C_\alpha$, we have that $\beta \in S$ and $C_\beta = C_\alpha \cap \beta$.

A coherent partial $C$-sequence with domain $S$ is \emph{trivial} if there is a club $D$ such that $\forall \gamma \in \acc D\cap S$, $D \cap \gamma = C_\gamma$.\footnote{In this case, if $S$ is stationary then if follows that $\acc D \subseteq S$, using coherency.}

Let $S^{\lambda}_{\leq\kappa}=\{\alpha < \lambda \mid \cf \alpha \leq \kappa\}$.
\begin{lemma}\label{lemma:lifting-partial-squares}
    Let $\kappa$ be an uncountable regular cardinal and let $\lambda >\kappa^{+}$ be regular. Let us assume that there is a saturated ideal $I$ on $\kappa^{+}$ such that $\mathcal{P}(\kappa^{+}) / I$ contains a $\sigma$-closed dense set.

    Then, every coherent partial $C$-sequence $\langle C_\alpha \mid \alpha \in S\rangle$, $S^{\lambda}_{\leq \kappa} \subseteq S$, can be (uniquely) extended to a coherent partial $C$-sequence with domain containing $S^\lambda_{\leq \kappa^{+}}$.
\end{lemma}
\begin{proof}
    Let $\mathcal{C} = \langle C_\alpha \mid \alpha \in S\rangle$, $S \supseteq S^{\lambda}_{\leq\kappa}$, be a coherent partial $C$-sequence.
    
    Let $H \subseteq \p(\kappa^{+})/I$ be $V$-generic. In $V[H]$, there is an elementary embedding $j\colon V \to M$, with $\crit j = \kappa^{+}$ and $M^\kappa \cap V[H] \subseteq M$ (by the saturation of $I$). 

    Fix $\delta <\lambda$, $\cf \delta = \kappa^{+}$. Let us look at $j(\mathcal{C})$ and let $\tilde\delta = \sup j[\delta]$. By the closure of $M$, $\cf^M \tilde\delta = \kappa = j(\kappa)$ and thus $D = j(\mathcal{C})_{\tilde\delta}$ is a club. Let $C_{\delta}$ be the closure of $j^{-1}[D]:=\{\alpha : j(\alpha) \in D\}$. We need to show that $\sup C_\delta = \delta$ and that the extended partial $C$-sequence is coherent, namely for every $\gamma \in \acc C_{\delta}$, $C_\delta \cap \gamma = C_\gamma$.  

    First, since the embedding $j$ is $\sigma$-continuous, $j[\delta]$ and $D$ are both $\sigma$-clubs at $\tilde\delta$ and thus must intersect cofinally. 

    Next, for $\gamma \in \acc C_\delta$ such that $\cf^V\gamma \neq \kappa^+$, $\sup j[\gamma] = j(\gamma)$ and thus $j(\gamma)$ is an accumulation point of $D$.\footnote{If $\cf \gamma \leq \kappa = \crit j$, continuity is clear. For $\cf \gamma \geq \kappa^{++}$, this follows from the assumption that the ideal is on $\kappa^{+}$.} By the coherency of $j(\mathcal{C})$, $D\cap j(\gamma) = j(\mathcal{C})_{j(\gamma)} = j(C_\gamma)$. Therefore, $j^{-1}[D]\cap \gamma = C_\gamma$ which is already closed.\footnote{Note that there is always an accumulation point $\tilde\gamma$ of $C_\delta$ of countable cofinality above $\gamma$, so $\tilde{\gamma} \in S$. In particular, $\gamma \in \acc C_{\tilde{\gamma}}$ and thus in $S$.}
    
    To show coherence for $\gamma$ of cofinality $\kappa^{+}$, let us first show that $C_\delta$ is unique in the following strong sense. Let $E$ be a club at $\delta$ in any extension $W$ of $V[H]$ by a $\sigma$-closed forcing, such that for every $\gamma \in \acc E \cap (S^\lambda_{\leq \kappa})^V$, $E \cap \gamma = C_\gamma$. Then $E = C_\delta$. 
    
    Indeed, in $W$, $\cf \delta >\omega$ and $S^\delta_\omega$ is still stationary at $\delta$. 
    So, $\acc E \cap \acc C_\delta \cap S^\delta_\omega$ is cofinal at $\delta$. For every $\gamma \in \acc E \cap \acc C_\delta \cap S^\delta_\omega$, $E \cap \gamma = C_\delta \cap \gamma$. 

    Now, applying the uniqueness for $W = V[H]$, and assuming that $\gamma$ is of $V$-cofinality $\kappa^{+}$ and $\gamma \in \acc C_\delta$. Then as $C_\delta \cap \gamma$ satisfies the requirements with respect to $\gamma$, it must be equal to $C_\gamma$. 

    Finally, let $H'$ be $V[H]$-generic filter for $\p(\kappa^{+})/I$. So, working in $V[H']$, we obtain a club $C'_\delta$ threading $\mathcal{C}\restriction \delta$. Applying uniqueness for $W= V[H][H']$, $C_\delta' = C_\delta$. Therefore, $C_\delta \in V[H] \cap V[H'] = V$. By uniqueness, the map $\delta \mapsto C_\delta$ for $\delta \in S^\lambda_{\kappa^{+}}$ belongs to $V$. 
\end{proof}
\begin{theorem}
    Assume that for every $n < \omega$ there is an ideal $I_n$ such that $\p(\omega_{n+1})/I_n\cong \mathcal{B}(\col(\omega_n,\omega_{n+1}))$. Then, for every $m\geq 2$, every partial coherent $C$-sequence $\mathcal{C} = \langle C_\alpha \mid \alpha \in S\rangle$ such that $S \supseteq S^{\omega_{m}}_{\leq\omega_1}$ is trivial. 
\end{theorem}
\begin{proof}
    By Lemma \ref{lemma:lifting-partial-squares}, any such partial $C$-sequence can be extended uniquely to a coherent $C$-sequence with full domain. If the obtained sequence is non-trivial, then it is a $\square(\omega_{m})$-sequence, and in particular it is incompatible with the existence of a $\sigma$-closed ideal on $\omega_{m}$ by Proposition \ref{prop:diagonal-stationary-relfection}.
\end{proof}

Let us remark that given a model with a dense ideal on $\kappa^{+}$, one can add (using forcing) a square sequence $\square_{\kappa^{+}}$ by initial segments. This forcing does not add any subset to $\kappa^{+}$, so the ideal remains dense in the generic extension. 

\subsection{Boolean transfer}\label{subsec: boolean-transfer}

We say that an ideal on $X$ is \emph{uniform} if every $Y \subseteq X$ of size $<|X|$ is in the ideal.  A filter is said to be uniform when its dual ideal is uniform.

\begin{prop}
\label{uf_trick}
    If $\kappa\leq \lambda$ are infinite cardinals and $I$ is an ideal on $\kappa$, then there exists a uniform ideal $J$ on $\lambda$ such that $\p(\kappa)/I \cong \p(\lambda)/J$.  If $I$ is $\delta$-complete and there exists a $\delta$-complete uniform ultrafilter on $\lambda$, then $J$ can be taken to be $\delta$-complete.
\end{prop}

\begin{proof}
    Let $\la X_\alpha : \alpha < \kappa \ra$ be a partition of $\lambda$ into $\kappa$-many pieces of size $\lambda$.  For each $\alpha<\kappa$, let $U_\alpha$ be a uniform ultrafilter on $X_\alpha$.  Let $J_0$ be the collection of subsets of $\lambda$ defined by $A \in J_0$ if and only if for all $\alpha<\kappa$, $A \cap X_\alpha \notin U_\alpha$.  Since for each $\alpha<\kappa$, $\p(X_\alpha) \setminus U_\alpha$ is an ideal on $X_\alpha$, $J_0$ is an ideal.
    
    Let $\pi : \lambda \to \kappa$ be defined by $\pi(\beta) = \alpha$, where $\alpha$ is the unique ordinal such that $\beta \in X_\alpha$.  Let $\varphi : \p(\kappa) \to \p(\lambda) / J_0$ be defined by $\varphi(B) = [\pi^{-1}[B]]_{J_0}$.
    It is clear that $\varphi$ is an injective Boolean homomorphism.  To check that it is surjective, let $A \subseteq \lambda$, and let $B = \{ \alpha < \kappa : A \cap X_\alpha \in U_\alpha \}$.  It suffices to show that $A$ is mod-$J_0$-equivalent to $C = \bigcup_{\alpha \in B} X_\alpha$.  But $A \setminus C = \bigcup_{\alpha \notin B} A \cap X_\alpha \in J_0$, and $C \setminus A = \bigcup_{\alpha \in B} X_\alpha \setminus A_\alpha \in J_0$.

    Now let $I$ be any ideal on $\kappa$.  Then $\varphi[I]$ is an ideal on the Boolean algebra $\p(\lambda)/J_0$, and we define an ideal $J$ on $\lambda$ by $A \in J$ iff $[A]_{J_0} \in \varphi[I]$.  It is clear that $\p(\kappa)/I \cong \p(\lambda)/J$.

    Finally, suppose that $\delta$ is a regular cardinal such that $I$ is $\delta$-complete and that each $U_\alpha$ is $\delta$-complete.  Then $J_0$ is $\delta$-complete.  Suppose $\la A_\alpha : \alpha < \eta \ra \subseteq J$, where $\eta<\delta$.  For each $\alpha$, let $B_\alpha \subseteq \kappa$ be such that $\left[A_\alpha\right]_{J_0} = \left[\pi^{-1}[B_\alpha]\right]_{J_0}$.  By the $\delta$-completeness of $J_0$, 
    $\left[\bigcup_{\alpha<\eta} A_\alpha\right]_{J_0} = \left[\bigcup_{\alpha<\eta} \pi^{-1}[B_\alpha]\right]_{J_0} = \left[\pi^{-1}[\bigcup_{\alpha<\eta} B_\alpha]\right]_{J_0}$.  
    Since $I$ is $\delta$-complete, $\bigcup_{\alpha<\eta} B_\alpha \in I$, and thus $\bigcup_{\alpha<\eta} A_\alpha \in J$.
    \end{proof}

The argument for the next result is extracted from \cite{foreman98}.  The claim about transfer via dense ideals is attributed there to Woodin, along with the remark that it has a relatively simple proof, and the bulk of the paper is dedicated to the more complicated proof of Foreman's general result about very strongly layered ideals.  In the argument below, we employ the idea of Theorem 4.1 of \cite{foreman98}, the ``selection hypothesis'', noticing that it allows both a further simplification (by eliminating the ``coding condition'') and a strengthened conclusion.

\begin{theorem}
\label{transfer}
    Assume $\kappa$ is regular, there is a normal ideal $J$ on $\kappa^+$ such that $\p(\kappa^+)/J \cong \mathcal B(\col(\kappa,\kappa^+))$, and $\diamondsuit_{\kappa^+}(\cof(\kappa))$ holds.  
    Then there is a $\kappa$-complete ideal $K \supseteq J$ and a function $\pi : \kappa^+ \to \kappa$ such that:
    \begin{enumerate}
        \item For all $A \subseteq \kappa$, $|A|<\kappa$ iff $\pi^{-1}[A] \in K$.
        \item $[A]_{\mathrm{bd}} \mapsto [\pi^{-1}[A]]_K$ is an isomorphism from $\p(\kappa)/\mathrm{bd}$ to $\p(\kappa^+)/K$.
        \item For $\mu\leq\kappa$, $[f]_{\mathrm{bd}} \mapsto [f \circ \pi]_K$ is a bijection from $\mu^\kappa /\mathrm{bd}$ to $\mu^{\kappa^+}/K$.
    \end{enumerate}
    Here, $\mathrm{bd}$ denotes the ideal of bounded subsets of $\kappa$.
\end{theorem}

\begin{proof}
    Choose a maximal antichain $\la x_i : i < \kappa \ra$ of $\mathbb B = \mathcal B(\col(\kappa,\kappa^+))$.  
     Fix an isomorphism $\Psi : \p(\kappa^+)/J \to \mathbb B$, and let $\la X_i : i < \kappa \ra$ be a sequence of pairwise disjoint sets such that $\Psi([X_i]_J) = x_i$.
    Define $\pi : \kappa^+ \to \kappa$ by $\pi(\alpha)  = i \iff \alpha \in X_i$.
    The main ingredient in the proof is the following sequence.
    \begin{lemma}\label{lemma:lifting-functions}
    There is a sequence of functions $\vec f = \la f_\alpha : \alpha < \kappa^+ \ra$ such that:
    \begin{enumerate}
        \item For each $\alpha$, $f_\alpha : \kappa \to \col(\kappa,\kappa^+)$.
        \item For each $\alpha<\kappa^+$ and $i<\kappa$, $f_\alpha(i) \leq x_i$.
        \item For $\alpha<\beta<\kappa^+$, $f_\beta(i) \leq f_\alpha(i)$ for all but boundedly-many $i<\kappa$.
        \item\label{genericity} For all maximal antichains $\mathcal A \subseteq \col(\kappa,\kappa^+)$, there is $\alpha<\kappa^+$ such that for all $i<\kappa$, $f_\alpha(i) \leq a$ for some $a \in \mathcal A$.
    \end{enumerate}
    \end{lemma}
    Let us assume the sequence $\vec f$ is given and show how to prove the theorem.
    Given such a sequence, we then define an ideal $I$ on $\mathbb B$ by putting $b \in I$ when there is $\alpha<\kappa^+$ such that for all but boundedly many $i<\kappa$, $f_\alpha(i) \leq \neg b$.  
    
    Let us first check that $I$ is a $\kappa$-complete ideal on $\mathbb B$.  Suppose $\delta<\kappa$ and $\{ b_j : j < \delta \} \subseteq I$.  There is $\alpha<\kappa^+$ and $i_0 < \kappa$ such that whenever $i_0\leq i <\kappa$ and $j <\delta$, $f_\alpha(i) \leq \neg b_j$.  Thus for $i_0\leq i <\kappa$, $f_\alpha(i) \leq \prod_j \neg b_j = \neg \sum_j b_j$, and so $\sum_j b_j \in I$.

    Let $K = \{ A \subseteq \kappa^+ : \Psi([A]_J) \in I \}$.  
    We have $\p(\kappa^+)/K \cong \mathbb B/I$.  The $\kappa$-completeness of $K$ follows from that of both $J$ and $I$.
    To show that $\pi^{-1}[A] \in K$ iff $A$ is bounded in $\kappa$, note that for $A \subseteq \kappa$, $\Psi(\pi^{-1}[A]_J) = \sum_{i \in A} x_i$, which is in $I$ iff $A$ is bounded.
    
    To show that $[A]_{\mathrm{bd}} \mapsto [\pi^{-1}[A]]_K$ is an isomorphism, it suffices to show that this map is surjective.  By (\ref{genericity}), for any $B \subseteq \kappa^+$, there is $\alpha<\kappa^+$ such that for a cobounded set of $i<\kappa$, $f_\alpha(i) \leq \Psi^{-1}([X_i \cap B]_J)$ or $f_\alpha(i) \leq \Psi^{-1}([X_i \setminus B]_J)$.  Let $A = \{ i : f_\alpha(i) \leq \Psi([B]_J)\}$.
    Thus for a cobounded set of $i$, either $f_\alpha(i)$ is below both $\Psi([\pi^{-1}[A]]_J)$ and $\Psi([B]_J)$, or below the meet of their complements, so $[B]_K =[\pi^{-1}[A]]_K$.
    
    
    For the claim about reduced powers, first note that if $\mu\leq\kappa$, $f,g \in \mu^\kappa$, and $f \not= g$ on an unbounded set $A$, then $\pi^{-1}[A] \notin K$, and $f \circ \pi \not= g \circ \pi$ on $\pi^{-1}[A]$.  To show surjectivity, suppose $g: \kappa^+ \to \mu$.  By the $\kappa^+$-completeness of $J$, $\mathcal A = \{ [g^{-1}\{ \gamma \}]_J :  g^{-1}\{ \gamma \} \in J^+ \}$ is a maximal antichain in $\p(\kappa^+)/J$.  By (\ref{genericity}), there is $\alpha<\kappa^+$ such that for each $i<\kappa$, $f_\alpha(i) \leq \Psi(a)$ for some $a \in \mathcal A$.  For each $i<\kappa$, let $\gamma_i$ be the unique $\gamma<\mu$ such that $f_\alpha(i) \leq \Psi([g^{-1}\{ \gamma \}]_J)$   Define $h : \kappa \to \mu$ by $h(i) = \gamma_i$.  Then for all $i<\kappa$ and all $\beta \in X_i$, $h(\pi(\beta)) = h(i) = \gamma_i$, and $f_\alpha(i) \leq \Psi([\{ \beta \in X_i : g(\beta) = h(\pi(\beta)) \}]_J)$.  Thus $[g]_K = [h\circ\pi]_K$.
    

    It remains to prove Lemma \ref{lemma:lifting-functions}. 
    \begin{proof}[Proof of Lemma \ref{lemma:lifting-functions}]
        Let us construct the sequence of functions $\la f_\alpha: \alpha<\kappa^+\ra$.    Fix a $\diamondsuit_{\kappa^+}(\cof(\kappa))$-sequence $\la d_\alpha : \alpha \in \kappa^+ \cap \cof(\kappa) \ra$.  Let $f_0 : \kappa \to \col(\kappa,\kappa^+)$ be such that for all $i<\kappa$, $f_0(i) \leq x_i$.  Given $f_\alpha$, let $f_{\alpha+1} = f_\alpha$.
        Suppose we have $\la f_\alpha : \alpha<\lambda \ra$, where $\lambda$ is a limit.  If $\cf(\lambda) = \delta <\kappa$, let $\la \alpha_\beta : \beta< \delta \ra$ be an increasing cofinal sequence in $\lambda$.  Let $i_0 < \kappa$ be such that $f_{\alpha_\gamma}(i) \leq f_{\alpha_\beta}(i)$ whenever $i \geq i_0$ and $\beta<\gamma<\delta$.  Then let $f_\lambda(i) = \inf_{\beta<\delta} f_{\alpha_\beta}(i)$ for $i \geq i_0$ and $f_\lambda(i) = f_0(i)$ for $i < i_0$.

        Suppose we have $\la f_\alpha : \alpha<\lambda \ra$, and $\cf(\lambda) = \kappa$.  First, construct a ``diagonal lower bound'' to the sequence as follows.  Let $\la \alpha_\beta : \beta < \kappa \ra$ be a cofinal increasing sequence in $\lambda$.  Let $\la i_\beta : \beta < \kappa \ra \subseteq \kappa$ be an increasing continuous sequence such that for $\beta<\gamma<\kappa$ and $i \geq i_{\gamma+1}$, $f_{\alpha_\gamma}(i) \leq f_{\alpha_\beta}(i)$.  Then define $f'_\lambda$ by $f'_\lambda(i) = \inf_{\beta<\gamma} f_{\alpha_\beta}(i)$ for all $i \in [i_\gamma,i_{\gamma+1})$ and $\gamma<\kappa$.  Then we have that for all $\alpha<\lambda$, $f'_\lambda \leq f_\alpha$ on a co-bounded set.

        Next, examine the set $d_\lambda$ given by the diamond sequence.  Ask whether $d_\lambda$ codes, via G\"odel pairing, a tuple $\la M,\mathcal A,\vec g \ra$ such that:
        \begin{enumerate}
            \item $M$ is a ${<}\kappa$ closed transitive model of the theory of $H_{\kappa^{++}}$, with $\lambda = (\kappa^+)^M$.
            \item $M \models \mathcal A$ is a maximal antichain contained in $\col(\kappa,\lambda)$.
            \item $\vec g \in M$ and $\vec g = \la f_\alpha : \alpha<\lambda \ra$.
        \end{enumerate}
        If this fails, let $f_\lambda = f'_\lambda$.  Otherwise, note that by the ${<}\kappa$-closure of $M$, $f'_\lambda(i) \in M$ for all $i<\kappa$.  For each $i$, let $f_\lambda(i) \in \col(\kappa,\lambda)$ be below both $f'_\lambda(i)$ and some $a \in \mathcal A$.

        To see that we obtain condition (\ref{genericity}), let $\mathcal A \subseteq \col(\kappa,\kappa^+)$ be a maximal antichain.  Let $N \prec H_{\kappa^{++}}$ be a $\kappa$-closed elementary submodel of size $\kappa^+$ with $\kappa^+ \cup \{\la f_\alpha : \alpha<\kappa^+ \ra, \mathcal A\} \subseteq N$.  Let $D \subseteq \kappa^+$ code $\la N,\mathcal A, \la f_\alpha : \alpha<\kappa^+\ra\ra$ via G\"odel pairing.  Let $\la N_\alpha : \alpha < \kappa^+ \ra$ be a continuous $\subseteq$-increasing sequence of elemenatary submodels of $N$ such that $N = \bigcup_\alpha N_\alpha$, each $N_{\alpha+1}$ is ${<}\kappa$-closed, and $\la f_\alpha : \alpha<\kappa^+ \ra, \mathcal A \in N_0$.  
        
        Since $\la d_\alpha : \alpha \in \kappa^+ \cap \cof(\kappa) \ra$ is a $\diamondsuit_{\kappa^+}(\cof(\kappa))$-sequence, there is some $\lambda \in \kappa^+ \cap \cof(\kappa)$ such that $D \cap \lambda = d_\lambda$ and $D \cap \lambda$ codes the transitive collapse $\la M, \mathcal A',\vec g \ra$ of $\la N_\lambda,\mathcal A,\la f_\alpha : \alpha < \kappa^+ \ra\ra$.  Note that $\mathcal A' \subseteq \mathcal A$, $\vec g = \la f_\alpha : \alpha < \lambda \ra$, and $M$ is ${<}\kappa$ closed since $\cf(\lambda) = \kappa$.  Thus by construction, the function $f_\lambda$ has the property that for all $i<\kappa$, $f_\lambda(i) \leq a$ for some $a \in \mathcal A' \subseteq \mathcal A$.
    \end{proof}
    This concludes the proof of Theorem~\ref{transfer}.
    \end{proof}

    \begin{corollary}
    \label{lowerend}
        If $\kappa$ is regular, $\diamondsuit_{\kappa^+}(\cof(\kappa))$ holds, and there is a normal ideal $J$ on $\kappa^+$ such that $\p(\kappa)/J \cong \mathcal B(\col(\kappa,\kappa^+))$, then for all regular $\delta\leq\kappa$ and all $\delta$-complete uniform ideals $I$ on $\kappa$, there is a $\delta$-complete uniform ideal $I'$ on $\kappa^+$ such that $\p(\kappa)/I \cong \p(\kappa^+)/I'$.  
        
        Furthermore, there is a function $\pi : \kappa^+ \to \kappa$ such that:
        \begin{enumerate}
        \item $I = \{ A \subseteq\kappa : \pi^{-1}[A] \in I' \}$.
        \item $[A]_I \mapsto [\pi^{-1}[A]]_{I'}$ is an isomorphism from $\p(\kappa)/I$ to $\p(\kappa^+)/I'$.
        \item For $\mu\leq\kappa$, $[f]_I \mapsto [f \circ \pi]_{I'}$ is a bijection from $\mu^\kappa/I$ to $\mu^{\kappa^+}/I'$.
        \end{enumerate}
    \end{corollary}

    \begin{proof}
        By Theorem~\ref{transfer}, there is a function $\pi : \kappa^+ \to \kappa$ and a uniform $\kappa$-complete ideal $K\supseteq J$ on $\kappa^+$ satisfying the conclusions, in particular such that $[A]_{\mathrm{bd}} \mapsto [\pi^{-1}[A]]_K$ gives an isomorphism $\Phi : \p(\kappa)/\mathrm{bd} \cong \p(\kappa^+)/K$.  
        A $\delta$-complete uniform ideal $I$ on $\kappa$ corresponds to a $\delta$-complete ideal $I_0$ on the Boolean algebra $\p(\kappa)/\mathrm{bd}$.  If $I_1 = \Phi[I_0]$, then $I_1$ is a $\delta$-complete ideal on $\p(\kappa^+)/K$, and we can set $I' = \{ A : [A]_K \in I_1 \}$. We have 
        $$\p(\kappa)/I \cong (\p(\kappa)/\mathrm{bd})/I_0 \cong (\p(\kappa^+)/K)/I_1 \cong \p(\kappa^+)/I'.$$
        Since both $I_1$ and $K$ are $\delta$-complete, so is $I'$.
        The ``furthermore'' part follows from the corresponding part of Theorem~\ref{transfer} and the definition of $I'$.
    \end{proof}
    By applying Corollary \ref{lowerend} repeatedly, we obtain the following.
    \begin{corollary}\label{cor:aleph_1-dense-ideals}
        Suppose that for every $n$ there is an $\omega_{n+1}$-complete ideal $I_n \subseteq \p(\omega_{n+1})$ such that $\p(\omega_{n+1})/I_n \cong \mathcal{B}(\col(\omega_n,\omega_{n+1}))$. Then for every $n\leq m$, there is a uniform, $\omega_n$-complete, $\omega_n$-dense ideal on $\omega_m$.

        In particular, the existence of a uniform, $\sigma$-complete,  $\omega_1$-dense ideal on $\omega_{n+1}$ for all $n < \omega$ is consistent relative to a huge cardinal.  
    \end{corollary}
\subsection{Irregular ultrafilters}\label{subsec: irregular-ultrafilters}
For any regular cardinal $\kappa$, there exists a sequence $\la f_\alpha: \alpha < \kappa^+ \ra$ of functions from $\kappa$ to $\kappa$ that is increasing modulo bounded sets.  Thus for any uniform ultrafilter $\mathcal U$ on $\kappa$, $| \kappa^\kappa/\mathcal U | \geq \kappa^+$. The same argument shows that for every $\kappa$-decomposible ultrafilter $\mathcal U$ over any set $X$,\footnote{An ultrafilter $\mathcal U$ is $\kappa$-decomposible if there are sets $\langle X_i  :  i < \kappa\rangle$ such that $\bigcup_{i<\kappa} X_i = X$, and for every $\alpha < \kappa$, $\bigcup_{i<\alpha} X_i\notin \mathcal U$.} $|\kappa^X/\mathcal U| \geq \kappa^+$. It is also easy to see that for an ultrafilter $\mathcal U$ on $X$ and a successor cardinal $\kappa = \mu^+$, $| \kappa^X / \mathcal U | \leq | \mu^X / \mathcal U |^+$.  Thus for $m \leq n <\omega$, it follows that for all uniform ultrafilters over $\omega_n$, $\omega_{m+1} \leq | \omega_m^{\omega_n} / \mathcal U | \leq 2^{\omega_n}$.

More generally, Keisler \cite{keisler} gave a criterion for when ultrapowers have the maximal possible cardinality.  
An ultrafilter $\mathcal U$ is called \emph{$(\mu,\kappa)$-regular} when there is a sequence $\la X_\alpha : \alpha < \kappa \ra\subseteq \mathcal U$ such that for all $Y \subseteq \kappa $ with $\ot(Y) = \mu$, $\bigcap_{\alpha \in Y} X_\alpha = \emptyset$.  Every countably incomplete ultrafilter is $(\omega,\omega)$-regular, and every uniform ultrafilter on a cardinal $\kappa$ is $(\kappa,\kappa)$-regular.  For any infinite cardinal $\kappa$, any ultrafilter extending the minimal fine filter on $[\kappa]^{<\omega}$ is $(\omega,\kappa)$-regular, so there exists an $(\omega,\kappa)$-regular uniform ultrafilter on $\kappa$.
\begin{theorem}[Keisler \cite{keisler}]
    If $\mathcal U$ is a $(\mu,\kappa)$-regular ultrafilter on $X$, then 
    $$| (2^{<\mu})^X / \mathcal U | \geq 2^\kappa.$$
\end{theorem}

The existence of irregular ultrafilters --- namely uniform ultrafilters on a cardinal $\kappa$ which are not $(\omega,\kappa)$-regular --- was one of the driving forces behind results regarding strong properties of ideals in the early years of the field. For example, in \cite{Magidor79}, Magidor showed the consistency of the existence of an irregular ultrafilter on $\omega_2$ from a huge cardinal. This result was improved by Foreman, in \cite{foreman98}, who showed that from an $\aleph_1$-dense ideal on $\omega_2$, one can obtain an ultrafilter $\mathcal U$ such that $|\omega^{\omega_2}/\mathcal U| = \aleph_1$, using the transfer principle. 

In our model, all possibilities for sizes of ultrapowers on the $\omega_n$'s can be realized, subject to the \emph{a priori} constraint given by $\GCH$.

\begin{theorem}
\label{up_sizes}
Suppose for all $n < \omega$, there is a normal ideal $I$ on $\omega_{n+1}$ such that $\p(\omega_{n+1})/I \cong \mathcal B(\col(\omega_{n},\omega_{n+1}))$.  Then for all $a,b,c \in \omega$ such that $a < b \leq c+1$, there is a uniform ultrafilter $\mathcal U$ on $\omega_c$ such that $| \omega_a^{\omega_c} / \mathcal U | = \omega_b$.
\end{theorem}

\begin{proof}
    Note that the hypothesis implies $\GCH$ below $\aleph_\omega$ by Proposition~\ref{gch}.
    
    We argue by induction on $c$.  If $c = 0$, then $a = 0$ and $b = 1$.  Any non-principal ultrafilter $\mathcal U$ on $\omega$ has the property that 
    $| \omega^\omega/\mathcal U | = 2^\omega = \omega_1$.

    Suppose the conclusion holds when $c<n$.  To realize $|\omega_a^{\omega_n}/\mathcal U | = \omega_{n+1}$ for $a \leq n$, choose any $(\omega,\omega_n)$-regular uniform ultrafilter on $\omega_n$.  To obtain $|\omega_a^{\omega_n}/\mathcal U | = \omega_b$ for $a < b \leq n$, we have by induction that there is a uniform ultrafilter $\mathcal U_0$ on $\omega_{n-1}$ such that $|\omega_a^{\omega_{n-1}}/\mathcal U_0| = \omega_b$.  By Corollary~\ref{lowerend}, there is a uniform ultrafilter $\mathcal U_1$ on $\omega_n$ such that 
    $|\omega_a^{\omega_n}/\mathcal U_1 | = |\omega_a^{\omega_{n-1}}/\mathcal U_0 |$.
\end{proof}

\subsection{The Foreman-Kunen spectrum}\label{subsec: FK-spectrum}
In \cite{forbidden}, Foreman generalized an unpublished result of Kunen that there is no uniform, countably complete, $\omega_2$-saturated ideal on any cardinal in the interval $[\aleph_\omega,\aleph_{\omega_1})$.  For a regular cardinal $\kappa$, Foreman defined the collection $\mathcal C_\kappa$ as the smallest class such that:
\begin{enumerate}
    \item $\kappa \subseteq \mathcal C_\kappa$.
    \item If $\alpha,\beta \in \mathcal C_\kappa$ and $\beta^{+\alpha} \geq \kappa^{+\omega}$, then $\beta^{+\alpha} \in \mathcal C_\kappa$.
    \item If $\beta \in \mathcal C_\kappa$ and $\beta \geq \kappa$, then every cardinal in the interval $[\beta,\beta^{+\kappa})$ is in $\mathcal C_\kappa$.
\end{enumerate}
(Note that (3) actually follows from (1) and (2) since we can show by induction that $\mathcal C_\kappa \cap [\kappa,\kappa^{+\omega}) = \emptyset$ and apply (2) for each $\alpha<\kappa$.)
He showed that if $\kappa$ is a successor cardinal, then there is no uniform, $\kappa$-complete, $\kappa^{+}$-saturated ideal on any $\lambda \in \mathcal C_\kappa$.  He asked whether $\mathcal C_\kappa$ fully describes the extent of the limitation.  Specifically, is it consistent that for all regular $\lambda \notin \mathcal C_{\omega_1}$, there is a uniform, countably complete, $\omega_2$-saturated ideal on $\lambda$?

In \cite{foremanhandbook}, Foreman emphasized the question of whether the lower end of this restriction is sharp.  The present work shows that it is.  It follows by Corollary~\ref{lowerend} that in our model, for every pair of positive natural numbers $n \leq m$, there is a uniform, $\omega_n$-complete, $\omega_n$-dense ideal on $\omega_m$.

We would also like to answer Foreman's question about $\mathcal C_\kappa$ by showing that it is actually not restrictive enough.  There are further cardinals on which there is provably no uniform, countably complete, $\omega_2$-saturated ideal.
\begin{definition}
Let $\Gamma$ be a class of forcings.

The \emph{Foreman-Kunen spectrum} $\mathsf{FK}(\kappa,\Gamma)$ is the class of regular cardinals $\lambda$ such that there exists a uniform, $\kappa$-complete, $\kappa^+$-incomplete, precipitous ideal $I$ on $\lambda$ with the property that $\p(\lambda)/I \in \Gamma$.
\end{definition}
We define $\mathrm{fix}(\kappa,\Gamma)$ as the collection of ordinals $\lambda$ such that for every $\bbP \in \Gamma$, if $\bbP$ forces an elementary embedding $j  :  V \to M \subseteq V[G]$, where $M$ is transitive and $\crit(j) = \kappa$, then $j(\lambda) = \lambda$.\footnote{To render this collection first-order definable, we should restrict to embeddings with domain some $V_\theta$, $\theta>\kappa+\lambda$, and codomain some transitive set $M\in V[G]$.}
Foreman's argument shows that if $\kappa$ is a successor cardinal and $\Gamma$ is the class of $\kappa^+$-c.c.\ forcings, then $\mathcal C_\kappa \subseteq \mathrm{fix}(\kappa,\Gamma)$.

\begin{prop}
    $\mathsf{FK}(\kappa,\Gamma) \cap \mathrm{fix}(\kappa,\Gamma) =\emptyset$. 
\end{prop}

\begin{proof}
    For any regular cardinal $\lambda$, there is a sequence of functions $\la f_\alpha:  \alpha<\lambda^+ \ra$ from $\lambda$ to $\lambda$ that is increasing modulo bounded sets.  If $\lambda \in \mathsf{FK}(\kappa,\Gamma)$, then there is a uniform $\kappa$-complete precipitous ideal $I$ on $\lambda$
    such that $\p(\lambda)/I \in \Gamma$.  
    If $G \subseteq \p(\lambda)/I$ is generic, then let $j:  V \to M \subseteq V[G]$ be the ultrapower embedding. Then $[f_\alpha]_G< [f_\beta]_G$ for all $\alpha<\beta<\lambda^+$.  Thus $j(\lambda) \geq (\lambda^+)^V$.
\end{proof} 

For a class of forcings $\Gamma$, let us say that a property 
$\varphi(x)$ is \emph{$\Gamma$-robust above $\delta$} if every forcing $\mathbb P \in \Gamma$ preserves $\varphi(\alpha)$ for every ordinal $\alpha\geq\delta$.  
If $\Gamma$ is a class of forcings of size less than $\delta$, then all standard large cardinal notions are $\Gamma$-robust above $\delta$.

\begin{prop}
\label{forbidden}
    Suppose $\kappa$ is a successor cardinal, $\Gamma$ is a class of $\kappa^{+\omega}$-c.c.\ forcings, and $\varphi(x)$ is a $\Pi_1$ formula that is $\Gamma$-robust above $\kappa^{+\omega}$.  
    \begin{enumerate}
        \item $\kappa \cup \{\kappa^{+\omega}\} \subseteq \mathrm{fix}(\kappa,\Gamma)$.
        \item  If $\alpha,\beta \in \mathrm{fix}(\kappa,\Gamma)$, $\beta\geq\kappa^{+\omega}$, and $\delta$ is the $\alpha^{th}$ ordinal above $\beta$ at which $\varphi(x)$ holds, then $\delta \in \mathrm{fix}(\kappa,\Gamma)$.
        \item If $\alpha$ is a limit point of $\mathrm{fix}(\kappa,\Gamma)$ such that $\cf(\alpha)\in\mathrm{fix}(\kappa,\Gamma)$, then $\alpha \in \mathrm{fix}(\kappa,\Gamma)$.
    \end{enumerate}
\end{prop}

Before giving the proof, let us compare this claim with Foreman's result about $\mathcal C_\kappa$.  Since being a cardinal is a $\Pi_1$ property that is $\Gamma$-robust above $\kappa^{+\omega}$ for any $\Gamma$ contained in the class of $\kappa^{+\omega}$-c.c.\ forcings, the above proposition implies that $\mathcal C_\kappa \subseteq \mathrm{fix}(\kappa,\Gamma)$ for successor $\kappa$.  The class $\mathcal C_\kappa$ can be formed as an increasing $\omega$-sequence in the following way.  Let $\mathcal C_\kappa^0 = \kappa$, and given $\mathcal C^n_\kappa$, let 
$$\mathcal C_\kappa^{n+1} = \{ \beta^{+\alpha} :  \alpha,\beta \in \mathcal C_\kappa^n \wedge \beta^{+\alpha} \notin [\kappa,\kappa^{+\omega}) \}.$$ 
Let $\mathcal C_\kappa = \bigcup_{n} \mathcal C_\kappa^n$.  If $\kappa$ is a successor cardinal, then $\sup(\mathcal C_\kappa^1) = \kappa^{+\kappa} = \aleph_\kappa > \kappa$.
Assume $n\geq 2$ and for $1\leq i<n$, $\delta_i = \sup(\mathcal C_\kappa^i)$ is a limit cardinal, and $\delta_i = \aleph_{\delta_{i-1}} > \delta_{i-1}$.  
Then 
$$\delta_n := \sup \mathcal C_\kappa^n = \sup \{ \beta^{+\alpha} : \alpha,\beta \in \mathcal C_\kappa^{n-1}\} = \aleph_{\delta_{n-1}} > \aleph_{\delta_{n-2}} = \delta_{n-1}$$
Thus if $\delta = \sup_n \delta_n$, then $\sup \mathcal C_\kappa = \delta = \aleph_{\delta}$, and this is the first fixed point of the $\aleph$-function above $\kappa$.  However, note that being a fixed point of the $\aleph$-function is a $\Pi_1$ property that is $\Gamma$-robust above $\kappa^{+\omega}$ when $\Gamma$ is a class of $\kappa^{+\omega}$-c.c.\ forcings.  Also, $\delta$ is a limit of $\mathrm{fix}(\kappa,\Gamma)$ of countable cofinality.  Thus $\mathrm{fix}(\kappa,\Gamma)$ is strictly larger than $\mathcal C_\kappa$ for such $\Gamma$.  Other examples of $\Pi_1$ properties that are $\Gamma$-robust for suitable $\Gamma$ above sufficiently large ordinals include being regular, being a limit cardinal, and some degrees of (weak) Mahloness.

\begin{proof}[Proof of Proposition~\ref{forbidden}]
First note that $\kappa \subseteq \mathrm{fix}(\kappa,\Gamma)$ by definition.
Next note that if $\kappa$ is a successor cardinal and $j : V \to M \subseteq V[G]$ is a generic embedding with critical point $\kappa$, where $V[G]$ is 
a $\kappa^{+\omega}$-c.c.\ forcing extension and $M$ is transitive, then $\kappa^{+\omega}$ is closed under $j$. This holds because otherwise, if we take the least $n$ such that $j(\kappa^{+n}) > \kappa^{+\omega}$, then we would have a tail of cardinals below $\kappa^{+\omega}$ that are forced to have the same cardinality as $j(\mu)$, where $\kappa^{+n} = \mu^+$.  It follows that $\kappa^{+\omega} \in \mathrm{fix}(\kappa,\Gamma)$.

Suppose $\alpha,\beta \in \mathrm{fix}(\kappa,\Gamma)$, $\beta\geq\kappa^{+\omega}$, and $\delta$ is the $\alpha^{th}$ ordinal above $\beta$ at which $\varphi(x)$ holds.  Suppose $\bbP \in \Gamma$, $G \subseteq \bbP$ is generic over $V$, and $j : V \to M \subseteq V[G]$ is an elementary embedding with $M$ transitive and $\crit(j) = \kappa$.
Then $M \models$ ``$j(\delta)$ is the $j(\alpha)^{th}$ ordinal $\gamma \geq j(\beta) \geq j(\kappa^{+\omega})$ such that $\varphi(\gamma)$,'' and note that $j$ is the identity on $\alpha,\beta,\kappa^{+\omega}$.  Since $\varphi(x)$ is $\Gamma$-robust above $\kappa^{+\omega}$ and a $\Pi_1$ property, 
$$\{ \gamma \geq \kappa^{+\omega} : V \models \varphi(\gamma) \} = \{ \gamma \geq \kappa^{+\omega} : V[G] \models \varphi(\gamma) \}.$$
Again because $\varphi(x)$ is $\Pi_1$,
$$\{ \gamma \geq \kappa^{+\omega} : M \models \varphi(\gamma) \} \supseteq \{ \gamma \geq \kappa^{+\omega} : V[G] \models \varphi(\gamma) \}.$$
Thus the $\alpha^{th}$ ordinal $\gamma\geq\beta$ such that $M\models\varphi(\gamma)$, is less than or equal to the $\alpha^{th}$ ordinal $\gamma\geq\beta$ such that $V[G] \models\varphi(\gamma)$.  Thus $j(\delta) \leq \delta$, so $j(\delta) = \delta$.

For the final claim, suppose $\alpha$ is a limit point of $\mathrm{fix}(\kappa,\Gamma)$ and $\cf(\alpha)\in \mathrm{fix}(\kappa,\Gamma)$.  Let $\la \alpha_i : i < \cf(\alpha) \ra$ be an increasing sequence contained in $\mathrm{fix}(\kappa,\Gamma)$ and cofinal in $\alpha$. 
Suppose $\bbP \in \Gamma$, $G \subseteq \bbP$ is generic over $V$, and $j : V \to M \subseteq V[G]$ is an elementary embedding with $M$ transitive and $\crit(j) = \kappa$.
Then 
$$j(\la \alpha_i : i < \cf(\alpha)\ra)  = \la \beta_i : i < j(\cf(\alpha)) \ra = \la \beta_i : i < \cf(\alpha) \ra $$ 
is cofinal in $j(\alpha)$, and 
$$\la j(\alpha_i) : i < \cf(\alpha) \ra = \la \alpha_i : i <\cf(\alpha) \ra$$
is cofinal in $\la \beta_i : i <\cf(\alpha)\ra$.  Thus $j(\alpha) = \sup_i \beta_i = \sup_i \alpha_i = \alpha$.
\end{proof}

The restriction to $\Pi_1$ properties is important, as can be seen by the $\Sigma_2$ property of being a measurable cardinal.  If $\kappa$ is a successor cardinal, there is a $\kappa^+$-saturated normal ideal on $\kappa$, and $\lambda>\kappa$ is measurable, then by Proposition~\ref{uf_trick}, there is a $\kappa$-complete uniform ideal $J$ on $\lambda$ such that $\p(\lambda)/J \cong \p(\kappa)/I$.  Thus $\lambda \in \mathsf{FK}(\kappa,\kappa^+\text{-c.c.})$.  By the same token, if $\lambda > \kappa$ is strongly compact, then every regular $\delta \geq \lambda$ is in $\mathsf{FK}(\kappa,\kappa^+\text{-c.c.})$.



It is natural to ask exactly how large $\mathsf{FK}(\kappa,\Gamma)$ can be.  
Foreman~\cite{forbidden} showed that consistently, $\mathsf{FK}(\kappa,\kappa\text{-centered})$ contains $\{ \lambda^+ : \lambda = |\lambda| >\cf(\lambda) = \kappa \}$ simultaneously for all successor cardinals $\kappa$.

\subsection{Graph colorings}\label{subsec: graph-colorings}

A \emph{graph} is a pair $G = \la V_G,E_G \ra$, where $E_G \subseteq [V_G]^2$.  $V_G$ is the set of \emph{vertices}, and $E_G$ is the set of \emph{edges}.  For graphs $G,H$, we say that $H$ is a \emph{subgraph} of $G$ when $V_H \subseteq V_G$ and $E_H = E_G \cap [V_H]^2$.

A \emph{vertex coloring} of a graph $G$ is a function $c : V_G \to \kappa$ for some cardinal $\kappa$ such that $c(x) \not= c(y)$ whenever $\{ x,y \} \in E_G$.  The \emph{chromatic number} of a graph $G$, or $\chi(G)$, is the least cardinal $\kappa$ such that there exists a vertex coloring of $G$ with range $\kappa$.

\begin{theorem}[Erd\H{o}s-Hajnal \cite{EH}]
\label{EH}
For every infinite cardinal $\kappa$ and every $n<\omega$, there exists a graph $G$ with $\beth_n(\kappa)^+$-many vertices such that $\chi(G) > \kappa$, but $\chi(H) \leq \kappa$ for every subgraph $H$ of $G$ such that $|V_H| \leq \beth_n(\kappa)$.
\end{theorem}

Here, we use the standard notation that $\beth_0(\kappa) = \kappa$, and $\beth_{n+1}(\kappa) = 2^{\beth_n(\kappa)}$.

In \cite{EH}, Erd\H{o}s and Hajnal also introduced their universal graph $\mathsf{EH}(\kappa,\lambda)$, defined by taking the vertex set as the set of all functions $f: \kappa \to \lambda$, and connecting $f,g$ with an edge when $| \{ \alpha: f(\alpha) = g(\alpha) \} | < \kappa$.  (So two functions are connected when they \emph{dis}agree almost everywhere.)  They showed that for every graph $G$ with vertex set of size $\kappa$, if every subgraph $H$ with $|V_H| < |V_G|$ has $\chi(H) \leq \lambda$, then there is an edge-preserving homomorphism $f : G \to \mathsf{EH}(\kappa,\lambda)$.  This implies that $\chi(G) \leq \chi(\mathsf{EH}(\kappa,\lambda))$.  
By Theorem~\ref{EH}, we have that for every infinite $\kappa$ and $n<\omega$, 
$\chi(\mathsf{EH}(\beth_n(\kappa)^+,\kappa)) > \kappa$.

Because of their usefulness in computing the chromatic numbers of infinite graphs in general, several authors, such as Todor\v{c}evi\'c \cite{todorcevic}, have urged the computation of the chromatic numbers of these graphs.  We note that \cite{komjath} shows that the value of $\chi(\mathsf{EH}(\omega_2,\omega))$ is independent of $\ZFC+\GCH$.

In our model, the value of $\chi(\mathsf{EH}(\omega_m,\omega_n))$ equals the lower bound given by Erd\H{o}s and Hajnal for all $n \leq m < \omega$.

\begin{theorem}
    Suppose for all $n < \omega$, there is a normal ideal $I$ on $\omega_{n+1}$ such that $\p(\omega_{n+1})/I \cong \mathcal B(\col(\omega_{n},\omega_{n+1}))$.
    Then for all $n \leq m < \omega$, $\chi(\mathsf{EH}(\omega_m,\omega_n)) = \omega_{n+1}$.
\end{theorem}

\begin{proof}
    By Theorem~\ref{up_sizes}, there exists a uniform ultrafilter $\mathcal U$ on $\omega_m$ such that $|\omega_n^{\omega_m} / \mathcal U| = \omega_{n+1}$.  We can color the vertices of $\mathsf{EH}(\omega_m,\omega_n)$ by their equivalence classes mod $\mathcal U$.  If $f$ and $g$ are connected with an edge, then they disagree on a cobounded set, so $[f]_{\mathcal U} \not= [g]_{\mathcal U}$.  Since the hypothesis implies $\GCH$ below $\aleph_\omega$, $\chi(\mathsf{EH}(\omega_m,\omega_n)) \geq \omega_{n+1}$, so this is an optimal vertex coloring.
\end{proof}

In our model, we obtain many instances of the Foreman-Laver graph reflection property of \cite{Foreman86,Foreman-Laver}.  Recall that $[\kappa,\lambda] \twoheadrightarrow [\kappa',\lambda']$ means that every graph with $\kappa$ vertices and chromatic number $\lambda$ has a subgraph with $\kappa'$ vertices and chromatic number $\lambda'$.  We write $[\kappa,{\geq}\lambda] \twoheadrightarrow [\kappa',{\geq}\lambda']$ to mean that the chromatic numbers are not assumed to be exactly $\lambda,\lambda'$ respectively, but bounded below by these cardinals.

\begin{corollary}
    Suppose for all $n < \omega$, there is a normal ideal $I$ on $\omega_{n+1}$ such that $\p(\omega_{n+1})/I \cong \mathcal B(\col(\omega_{n},\omega_{n+1}))$.
    Then for all natural numbers $k,m,n$ with $n\leq m$,
    $[\omega_{m+k},\omega_{m}] \twoheadrightarrow [\omega_{n+k},\omega_{n}].$   
\end{corollary}

\begin{proof}
    We first establish $[\omega_{m+k},\geq\!\omega_{m}] \twoheadrightarrow [\omega_{n+k},\geq\!\omega_{n}]$ for all $k,m,n$ with $n\leq m$.  It suffices to assume that $m = n + 1$, since then the other instances follow from the transitivity of the transfer property.

    First, let us also assume that $n \geq 1$.  Suppose $G$ is a graph such that $|V_G| = \omega_{m+k}$ and $\chi(G) \geq \omega_{m}$.  If the transfer property fails, then for every subgraph $H$ such that $|V_H| \leq \omega_{n+k} = \omega_{m+k-1}$, $\chi(H) \leq \omega_{n-1}$.  But then there is an edge-preserving homomorphism $f  :  G \to \mathsf{EH}(\omega_{m+k},\omega_{n-1})$.  Since $\chi(\mathsf{EH}(\omega_{m+k},\omega_{n-1})) = \omega_{n} = \omega_{m-1}$, $\chi(G) < \omega_{m}$, a contradiction.

    If $n = 0$, then let $G$ be a graph with $V_G = \omega_{m+k}$ and $\chi(G)\geq \omega_{m}$, and let $M$ be an elementary substructure of $\la H_{\omega_{m+k+1}},\in,G \ra$ of size $\omega_{k}$.  Let $H$ be the induced subgraph with $V_H = V_G \cap M$.  By the De Bruijn–Erd\H{o}s theorem, having infinite chromatic number is a first-order property, so indeed $\chi(H) \geq \omega$.

    For the stronger property $[\omega_{m+k},\omega_{m}] \twoheadrightarrow [\omega_{n+k},\omega_{n}]$, we argue by induction on $k$.  First note that the stronger relation is same as the weaker one when $k = 0$.
    Suppose that $k>0$ and $[\omega_{m+k'},\omega_{m}] \twoheadrightarrow [\omega_{n+k'},\omega_{n}]$ holds for all $n \leq m < \omega$ and $k' < k$.  Let $G$ be any graph on $\omega_{m+k}$ with $\chi(G) = \omega_m$.  Then there is an independent set of vertices, $X$, of size $\omega_{m+k}$.  Let $H_0$ be a subgraph witnessing $[\omega_{m+k},\geq\!\omega_{m}] \twoheadrightarrow [\omega_{n+k},\geq\!\omega_{n}]$.  
    If $\chi(H_0) = \omega_n$, we are done.  If $\chi(H_0) = \omega_{n+r}$ where $0<r\leq k$, then by induction we have $[\omega_{n+k},\omega_{n+r}]\twoheadrightarrow[\omega_{n+k-r},\omega_n]$.
    Let $H_1$ be a subgraph of $H_0$ witnessing this relation.
    Then let $X' \subseteq X$ have size $\omega_{n+k}$, and let $H$ be the subgraph of $G$ with $V_H = V_{H_1} \cup X'$.  Since $X'$ is independent, any vertex coloring of $H_1$ can be extended to one of $H$ by using only one more color for the members of $X' \setminus H_1$.  Thus $|V_H| = \omega_{n+k}$, and $\chi(H) = \chi(H_1) = \omega_n$.
    \end{proof}

    We note that the conclusion above is optimal in a sense.  One the one hand, if $k>0$ and $[\omega_{m+k},\omega_m] \twoheadrightarrow [\omega_{n+k},\omega_n]$ holds, then also $[\omega_{m+k},\omega_m] \twoheadrightarrow [\omega_{r+k},\omega_n]$ holds for $n \leq r \leq m$, which can be seen adjoining a large enough set of independent vertices to a witness of the first relation.  In other words, a nonzero gap between cardinality and chromatic number can be increased in a subgraph.  However, it cannot generally be decreased.\footnote{This contrasts to the cosmetically similar Chang's Conjecture $(\kappa_1,\kappa_0)\twoheadrightarrow(\mu_1,\mu_0)$ about elementary submodels, wherein the gap can be readily decreased but not increased.}  This is shown by the construction of Erd\H{o}s and Hajnal \cite{EH}.  For a set of ordinals $A$ and $n\geq 2$, they define a graph $G_n(A)$ on $[A]^n$ by connecting two nodes $\{ \alpha_1,\dots,\alpha_n \},\{ \beta_1,\dots,\beta_n \}$ (written in increasing order) when $\beta_i = \alpha_{i+1}$ for $0<i<n$.  They show that, assuming $\GCH$ holds below $\aleph_\omega$, for all $n,k<\omega$ such that $k \geq 1$, if $|A| = \omega_{n+k}$, then $\chi(G_{k+1}(A)) = \omega_n$. 
    Therefore, under $\GCH$, if $n \leq m$, $k \geq 1$, and $H$ is a subgraph of $G_{k+1}(\omega_{m+k})$ of chromatic number $\omega_n$, then $|V_H| \geq \omega_{n+k}$.  Thus for $n \leq m$ and $k_0 < k_1$, $[\omega_{m+k_1},\omega_m] \twoheadrightarrow [\omega_{n+k_0},\omega_n]$ is inconsistent with $\GCH$.
    
    In contrast to the results above, Shelah \cite{shelahgraph} showed that in $L$, for every successor cardinal $\kappa$, there is a graph $G$ with $|V_G| = \kappa$ and $\chi(G) = \kappa$ such that for every subgraph $H$ with $|V_H| < \kappa$, $\chi(H) \leq \omega$. A similar result was obtained by Shelah (using a non-reflecting stationary set and instances of $\GCH$, with larger graphs) in \cite{Shelah-graphs-2013} and by Lambie-Hanson and Rinot \cite{LHR-chromatic}, using $\square(\kappa)$ and $\GCH$, obtaining various gaps. 
  
\subsection{Partition hypotheses}\label{subsec: partition-hypothesis}

In recent work \cite{BBMT}, Bannister, Bergfalk, Moore, and Todor\v{c}evi\'c isolated a combinatorial principle $\mathrm{PH}_n(\kappa)$.  
To state this principle, let us first define a few notions.  For finite sequences $s,t$, we put $s \lhd t$ when $s$ is a strict subsequence of $t$.
For $1 \leq n < \omega$, a function $F  :  \kappa^{\leq n} \to \kappa$ is called \emph{$n$-cofinal} when 
\begin{itemize}
    \item for all $\alpha<\kappa$, $F(\la\alpha\ra) \geq \alpha$, and 
    \item whenever $s,t \in \kappa^{\leq n}$ and $s \lhd t$, $F(s) \leq F(t)$.
\end{itemize}

$\mathrm{PH}_n(\kappa)$ states that for all colorings $c  :  \kappa^{n+1} \to \omega$, there is $m<\omega$ and an $(n+1)$-cofinal $F  :  \kappa^{\leq n+1} \to \kappa$ such that for all sequences $s_1 \lhd s_2 \lhd \dots\lhd s_{n+1}$ such that $s_i \in \kappa^i$,
$$c\left(F(s_1),F(s_2),\dots,F(s_{n+1})\right) = m.$$
The combinatorial power of $\mathrm{PH}_n(\kappa)$ stems in part from the fact that for any injective sequence $\la\alpha_1,\alpha_2,\dots,\alpha_{n+1}\ra \subseteq \kappa$, there are $(n+1)!$ ways to arrange subsequences $s_1 \lhd s_2 \lhd \dots \lhd s_{n+1} = \la\alpha_1,\alpha_2,\dots,\alpha_{n+1}\ra$ with $\len(s_1) = 1$, and $c\left(F(s_1),F(s_2),\dots,F(s_{n+1})\right)$ will be the same for any such arrangement.

They showed the following results:
\begin{enumerate}
    \item For all $n$, $\mathrm{PH}_n(\omega_n)$ is false.
    \item If $\kappa$ is regular and $n \geq 1$, then $\mathrm{PH}_n(\kappa)$ implies $\neg\square(\kappa)$.
    \item If $\kappa$ is weakly compact, then $\mathrm{PH}_n(\kappa)$ holds for all $n$.
    \item If $1 \leq n < m$ and for $1 \leq k \leq n$, there is a uniform, $\omega_k$-complete, $\omega_k$-dense ideal on $\omega_m$, then $\mathrm{PH}_n(\omega_m)$ holds.
\end{enumerate}
Our model thus answers their question of whether $\mathrm{PH}_n(\omega_m)$ is consistent for $2 \leq n < m$.

In more recent work in preparation \cite{BLZ}, Berkfalk, Lambie-Hanson, and Zhang introduce a similar principle, $\mathrm{PH}^{\mathrm{bd}}_n(\kappa, \mathrm{reg})$.  This states that for all $c  :  \kappa^{n+1} \to \kappa$ such that $c(s) < \min(s)$ for all $s \in \kappa^{n+1}$, there is an $(n+1)$-cofinal $F  :  \kappa^{\leq n+1} \to \kappa$ and an ordinal $\xi < \kappa$ such that for all sequences $s_1 \lhd s_2 \lhd \dots\lhd s_{n+1}$ such that $s_i \in \kappa^i$,
$c\left(F(s_1),F(s_2),\dots,F(s_{n+1})\right) < \xi.$

They apply this principle in the study of homological algebra.  There are some rather concretely defined groups of families of functions on ordinals into a fixed abelian group such that the question of whether certain quotient groups are trivial, namely the subgroup of ``$n$-coherent'' families modulo the subgroup of ``$n$-trivial'' families, is independent of $\ZFC$.  For a given ordinal $\delta$ and fixed abelian group $\mathcal 
 A$, this quotient group is denoted by $\check{\mathrm{H}}^n(\delta,\mathcal A)$.  They show that for $m<n$, $\check{\mathrm{H}}^n(\omega_m,\mathcal A)$ is always trivial, $\check{\mathrm{H}}^n(\omega_n,\mathcal A)$ is nontrivial for some choice of $\mathcal A$, and the status of $\check{\mathrm{H}}^n(\omega_m,\mathcal A)$ for $n<m$ is independent.  They show that the axiom $V=L$ implies that $\check{\mathrm{H}}^n(\omega_m,\mathcal A) \neq 0$ whenever $n<m$, that $\mathrm{PFA}$ implies that $\check{\mathrm{H}}^1(\omega_m,\mathcal A) = 0$ when $m \geq 2$, and that the existence of the kind of uniform ideals on $\omega_m$ in our model implies that $\check{\mathrm{H}}^n(\omega_m,\mathcal A) = 0$ whenever $n<m$.  This argument goes through the principle $\mathrm{PH}^{\mathrm{bd}}_n(\kappa, \mathrm{reg})$ and is currently the only known way to show the consistency of $\check{\mathrm{H}}^n(\omega_m,\mathcal A) = 0$ for $2 \leq n < m$.  For more details, see \cite{BLZ}.

In summary, we have:
\begin{theorem}
    Suppose $\ZFC$ is consistent with a huge cardinal.  Then there is a model of $\ZFC$ in which:
    \begin{enumerate}
        \item $\mathrm{PH}_n(\omega_m)$ holds for all $n<m<\omega$.
        \item For $n,m \in \omega$, $n = m$ if and only if there exists an abelian group $\mathcal A$ such that the cohomology group $\check{\mathrm{H}}^n(\omega_m,\mathcal A)$ is nontrivial.
    \end{enumerate}
\end{theorem}

\section{Open Questions}\label{sec: questions}
We conclude this paper with a list of open questions. 

In \cite{foremanhandbook}, Foreman asks whether the existence of strongly layered ideal on each of the $\aleph_n$ implies that $\aleph_{\omega}$ is J\'onsson. If true, it would imply that $\aleph_\omega$ is J\'onsson in the model obtained in this paper and in particular that many instances of Chang's Conjecture hold in this model.    
\begin{question}\label{question:ideals-implies-cc}
    Assume that for all $n<\omega$ there is a normal $I$ ideal on $\omega_{n+1}$ such that $\p(\omega_{n+1})/I \cong \mathcal B(\col(\omega_n,\omega_{n+1}))$. 

    Are there $0 < k < n < m$ such that $(\omega_m, \omega_n) \twoheadrightarrow (\omega_n, \omega_k)$?
\end{question}
\begin{question}
    Assume that for all $n<\omega$, there is a normal $I$ ideal on $\omega_{n+1}$ such that $\p(\omega_{n+1})/I \cong \mathcal B(\col(\omega_n,\omega_{n+1}))$. 
    \begin{enumerate}
        \item Does $\SCH$ hold at $\aleph_\omega$?
        \item Can there be an $\aleph_{\omega+2}$-saturated ideal on $\aleph_{\omega+1}$?
        \item Can there be an $\omega_1$-indecomposible ultrafilter on $\aleph_{\omega+1}$?
    \end{enumerate}
\end{question}

The following question uses the notations of Subsection \ref{subsec: FK-spectrum}.    
\begin{question}
Is it consistent that for every regular cardinal $\kappa$ below the first $\aleph$-fixed point that is not in $\mathrm{fix}(\omega_1, \omega_2\text{-c.c.})$, there is a uniform, $\aleph_1$-dense, countably complete ideal on $\kappa$? 
\end{question}

The forcing $\Sh(\kappa,\lambda)$ and the anonymous collapse $\mathbb{A}(\kappa,\lambda)$ share many interesting properties. In particular, by the minimality of the anonymous collapse, there is a projection from $\Sh(\kappa,\lambda)$ to $\mathbb{A}(\kappa,\lambda)$.
\begin{question}
    Let $\kappa$ be regular uncountable cardinal, $\lambda > \kappa$ inaccessible. Is $\mathcal B(\Sh(\kappa,\lambda))\cong\mathcal B( \mathbb{A}(\kappa,\lambda))$?
\end{question}
\providecommand{\bysame}{\leavevmode\hbox to3em{\hrulefill}\thinspace}
\providecommand{\MR}{\relax\ifhmode\unskip\space\fi MR }
\providecommand{\MRhref}[2]{%
  \href{http://www.ams.org/mathscinet-getitem?mr=#1}{#2}
}
\providecommand{\href}[2]{#2}

\end{document}